\documentclass[a4paper,11pt]{article}
\usepackage[english]{babel}

\usepackage{natbib}

\usepackage{anysize}
\marginsize{3cm}{3cm}{2.5cm}{2.5cm}
\usepackage{fancyhdr}
\usepackage{amsfonts}
\usepackage{amsmath}
\usepackage{mathtools}
\usepackage{mathrsfs}
\usepackage{tikz,graphicx,subfigure,overpic}
\usepackage{color}
\usepackage{amssymb}
\usepackage{amsthm}
\usepackage[colorlinks=true, linkcolor=blue, citecolor=blue, urlcolor=blue]{hyperref}
\usepackage[fixlanguage]{babelbib}
\usepackage{color}
\usepackage{comment}
\usepackage{subfigure}
\usepackage{authblk}
\usepackage{pifont}
\numberwithin{equation}{section}

\newtheorem{theorem}{Theorem}[]
\newtheorem{lemma}[]{Lemma}
\newtheorem{corollary}[]{Corollary}
\newtheorem{proposition}[]{Proposition}
\newtheorem{assumption}{Assumption}[]
\theoremstyle{definition}
\newtheorem{definition}[]{Definition}
\newtheorem{Remark}[]{Remark}
\newenvironment{remark}{\begin{Remark}\rm}{\end{Remark}}

\newtheorem{Example}[]{Example}
\newcommand{\rev}[1]{#1}
\newcommand{\eq}{\begin{equation}}
\newcommand{\qe}{\end{equation}}
\newcommand{\beq}{\begin{equation}}
\newcommand{\eeq}{\end{equation}}
\newcommand{\beqa}{\begin{eqnarray}}
\newcommand{\eeqa}{\end{eqnarray}}
\newcommand{\beqan}{\begin{eqnarray*}}
\newcommand{\eeqan}{\end{eqnarray*}}

\newcommand{\N}{\mathbb{N}}
\newcommand{\Z}{\mathbb{Z}}
\newcommand{\R}{\mathbb{R}}
\newcommand{\C}{\mathbb{C}}
\newcommand{\E}{\mathbb{E}}
\renewcommand{\P}{\mathbb{P}}

\newcommand{\cO}{\mathcal{O}}

\newcommand{\Supp}{{\rm Supp}}

\newcommand{\bs}{\boldsymbol}
\newcommand{\bv}{\mathbf}

\newcommand{\Var}{\mathbb{V}{\mathrm{ar}}}
\newcommand{\Cov}{\mathbb{C}{\mathrm{ov}}}
\newcommand{\mueqd}{\mu_{eq}^{\otimes d}}

\newcommand{\gfrac}[2]{\genfrac{}{}{0pt}{}{#1}{#2}}

\renewcommand{\leq}{\leqslant}
\renewcommand{\geq}{\geqslant}
\renewcommand{\phi}{\varphi}
\renewcommand{\epsilon}{\varepsilon}
\renewcommand{\d}{ {\rm d}}

\renewcommand{\emptyset}{\varnothing}

\usepackage{url}
\DeclareUrlCommand\email{\urlstyle{rm}}

\def\bx{{\bv x}}
\def\cE{\mathscr{E}}
\def\figdir{.}
\def\twofig{.47\textwidth}

\def\IND{{\bs 1}}
\def\eqsp{\,}

\title{Monte Carlo with Determinantal Point Processes}

\author{
\ R\'emi Bardenet\footnote{Univ. Lille, CNRS, Centrale Lille, UMR 9189 -- CRIStAL, 59651 Villeneuve d’Ascq Cedex, France. Email: \href{mailto:remi.bardenet@gmail.com}{\nolinkurl{remi.bardenet@gmail.com}}}
\;\qquad
\ Adrien Hardy\footnote{Univ. Lille, CNRS, Inria, UMR 8524, Laboratoire Paul Painlev\'e, F-59000 Lille, France. \\ Email: \href{mailto:adrien.hardy@univ-lille.fr}{\nolinkurl{adrien.hardy@univ-lille.fr}}}
}

\begin{document}

\maketitle



\begin{abstract}
We show that repulsive random variables can yield Monte Carlo
  methods with faster convergence rates than the typical  $N^{-1/2}$, where
  $N$ is the number of integrand evaluations. More precisely, we propose
  stochastic numerical quadratures involving determinantal point processes
  associated with multivariate orthogonal polynomials, and we obtain root mean square
  errors that decrease as $N^{-(1+1/d)/2}$, where $d$ is the dimension of the
  ambient space. First, we prove a central limit theorem (CLT) for the linear statistics of a class of determinantal point processes,
  when the reference measure is a product measure supported on a hypercube, which
  satisfies the Nevai-class regularity condition; a result which may be of
  independent interest. Next, we introduce a Monte Carlo method based on these
  determinantal point processes, and prove a CLT with explicit limiting variance
  for the quadrature error, when the reference measure satisfies a stronger
  regularity condition. As a corollary, by taking a specific reference measure
  and using a construction similar to importance sampling, we obtain a general
  Monte Carlo method, which applies to any measure with continuously derivable
  density. Loosely speaking, our method can be interpreted as a stochastic counterpart to
  Gaussian quadrature, which, at the price of some convergence rate, is easily
  generalizable to any dimension and has a more explicit error term.
\end{abstract}

\tableofcontents

\section{Introduction}
\label{s:intro}
Numerical integration, or quadrature, refers to algorithms that approximate integrals
\eq
\label{int f mu}
\int f(x)\mu(\d x),
\qe
where $\mu$ is a finite positive Borel reference measure, and where $f$ ranges
over some class of test functions $\mathscr C$. We assume for convenience that
the support $\Supp(\mu)$ of $\mu$ is included in the $d$-dimensional hypercube $I^d=[-1,1]^d$, since one
can recover this setting in most applications by means of appropriate
transformations. For any given $N$, a quadrature algorithm outputs $N$ nodes $\bv x_1,\ldots,\bv x_N\in I^d$ and weights $w_1,\ldots,w_N\in\R$ so that the approximation
\eq
\label{e:quadrature}
\sum_{i=1}^N w_i f(\bv x_i) \approx \int f(x)\mu(\d x)
\qe
is reasonable for every $f\in\mathscr C$. The nodes and weights depend on $N$, $\mu$,
and can be realizations of random variables, but they are not allowed to depend
on $f$. The quality of a quadrature algorithm is assessed through the
approximation error
\eq
\mathscr E_N(f)= \sum_{i=1}^N w_i f(\bv x_i)-\int f(x)\mu(\d x)
\label{e:error}
\qe
and specifically its behaviour as $N\to\infty$.

 Many quadrature algorithms have been
developed: variations on Riemann summation \citep{DaRa84}, Gaussian quadrature
\citep{Gau04}, Monte Carlo methods \citep{RoCa04}, etc. In the remaining of
Section~\ref{s:intro}, we quickly review three families of such methods to
provide context for our contribution, which we then introduce in
Section~\ref{s:contribution}.

\subsection{Gaussian quadrature}
\label{s:quadrature}
Let us first assume $d=1$, so that $\mu$ is supported on $I=[-1,1]$. Let
$(\phi_k)_{k\in\N}$ be the orthonormal polynomials associated with this measure,
that is, the family of polynomials such that $\phi_k$
has degree $k$, positive leading coefficient, and $\int
\phi_k(x)\phi_\ell(x)\mu(\d x)=\delta_{k\ell}$ for every $k,\ell\in\N$.
\emph{Gaussian quadrature}, see e.g. \citep{DaRa84,Gau04,BrPe11} for
general references, then corresponds to taking for nodes $\bv x_1,\ldots,\bv x_N$ the
zeros of the $N$th degree orthonormal polynomial $\phi_N(x)$, which are real and
simple. As for the weights, Gaussian quadrature corresponds to
\eq
\label{weight Gauss}
w_i = \frac{1}{K_N(\bv x_i, \bv x_i)}\eqsp,
\qe
where we introduced the $N$th Christoffel-Darboux kernel associated  with $\mu$,
\eq
\label{CD Gauss}
K_N(x,y)=\sum_{k=0}^{N-1}\phi_k(x)\phi_k(y).
\qe

This celebrated method is characterized by the property to be exact, i.e.
$\mathscr E_N(f)=0$, for every polynomial function $f$ of degree up to $2N-1$.
This is the highest possible degree such that this holds. Gaussian quadrature is
thus particularly suitable when the test functions $f$ look like polynomials.
For instance, $\mathscr E_N(f)$ decays exponentially fast when $f$ is analytic
\citep{GaVa83}. However, although Gaussian quadrature is now two centuries old
\citep{Gau15}, optimal rates of decay for the error $\cE_N(f)$ do not seem to be
known for less regular test functions, say $f\in\mathscr C^1$, in general.  By
using Jackson's approximation theorem for algebraic polynomials, one can see
that $\mathscr E_N(f)=\cO(1/N)$ when $f\in\mathscr C^1$. Optimal decays have
been recently investigated in the particular case of the Gauss-Legendre
quadrature  \citep{XiBo12,Xia16}. However, even in the familiar Gauss-Jacobi
quadrature, optimal rates are only conjectured.

Efficient computation of the nodes and weights in Gaussian quadrature has been
an active topic of research. Classical approaches are
based on the QR algorithm, such as the Golub-Welsch algorithm, see
e.g. \citep[Section 3.5]{Gau04} for a discussion. The computational cost of
these QR approaches usually scales as $\cO(N^2)$. More recently, $\cO(N)$
approaches have been proposed for specific choices of the reference measure
\citep{GlLiRo07,HaTo13}, with parallelizable methods \citep{Bog14} further taking
down costs.

Let us stress that Gaussian quadrature is intrinsically a one-dimensional method.
Indeed, in the higher-dimensional setting where $\Supp(\mu)\subset I^d$,
although one may define multivariate orthonormal polynomials associated with
$\mu$, it is not possible to take for nodes the zeros of a multivariate
polynomial. However, if $\mu$ is a product measure
$\mu=\mu_1\otimes\cdots\otimes\mu_d$ with each $\mu_j$ supported on $I$, one
could build a grid of nodes using $d$ one-dimensional Gaussian quadratures. But
this has for consequence to rise up the one-dimensional error estimate for
$\mathcal E_N(f)$ to a power $1/d$, which essentially makes Gaussian quadrature
ineffective in higher dimensions than one or two. In fact, the same phenomenon
arises for any other grid-like product of one-dimensional quadratures; this is
commonly referred to as the curse of dimensionality.

\subsection{Monte Carlo methods}
\label{s:monteCarlo}
Monte Carlo methods \citep{RoCa04} correspond to picking up the $N$ nodes $(\bx_i)$ in
\eqref{e:quadrature} as the realizations of random variables in $I^d$. For instance,
assuming $\mu$ in \eqref{e:quadrature} has a density $\omega$ with respect to the
Lebesgue measure, \emph{importance sampling} refers to taking the $(\bx_i)$ to be
i.i.d. realizations with a so-called \emph{proposal} density $q$, and the
weights to be
\eq
\label{e:isWeights}
w_i = \frac{1}{N}\frac{\omega(\bx_i)}{q(\bx_i)}.
\qe
That way, $\cE_N(f)$ has mean zero. Provided that
\eq
\label{variance MC}
\Var \left[\frac{f(X)\omega(X)}{q(X)}\right]<\infty
\qe
where $X$ has density $q$, $\cE_N(f)$ has a standard deviation decreasing as $N^{-1/2}$, and satisfies the
classical central limit theorem: $$\sqrt N\cE_N(f)\xrightarrow[N\to\infty]{law} \mathcal N(0,\sigma_f^2),$$
where $\sigma_f^2$ equals \eqref{variance MC}. Let us stress that the cost of
importance sampling is $\cO(N)$, and that it can be easily parallelized.

When the ambient dimension $d$ becomes large, practitioners typically prefer \emph{Markov chain
  Monte Carlo} (MCMC) methods over importance sampling. This means taking $w_i=1/N$ and nodes $(\bx_i)$ to be
the realization of a Markov chain with stationary distribution $\mu$, such as the
Metropolis-Hastings chain. Under relatively weak conditions on the Markov chain and the integrand,
$\sqrt N\cE_N(f)$ then converges in distribution to a centered Gaussian
variable; see e.g. \citep[Theorem 7.32]{DoMoSt14}. The limiting variance usually grows more slowly
with the dimension $d$ than for
importance sampling; see \citep{BeCh09} for a proof in a simplified setting. This justifies the preferential use of MCMC for large $d$. In any case, the typical order of
magnitude of the error $\cE_N(f)$ for Monte Carlo methods is $N^{-1/2}$, which
is often deemed a rather slow decay.

\rev{To retain the simplicity of implementation of Monte Carlo and achieve rates faster than $N^{-1/2}$, several authors have proposed postprocessing steps.} For instance, \cite{DePo16} proposed a variant of importance sampling that still takes nodes as independent draws from some proposal density $q$, but takes weights to be
\begin{equation}
w_i = \frac{1}{N}\frac{\omega(\bx_i)}{\check{q}^{-i}(\bx_i)},
\label{e:delyonPortierWeights}
\end{equation}
where $\check{q}^{-i}$ is the so-called leave-one-out kernel estimator of the density $q$ of the nodes. Perhaps surprisingly, for smooth enough products $f\omega$ and the right tuning
of kernel parameters, $\sqrt{N}\cE_N(f)$
then converges in probability to zero. Exact rates are investigated by
\cite{DePo16}, and a central limit theorem is proven. We further discuss
their results in Section~\ref{s:discussion}.

\rev{Another postprocessing technique with fast convergence relies on control variates \citep{GlSz02}. \citet*{OaGiCh17}, for instance, sample nodes i.i.d. from $q$, and then split the nodes into two batches. The first batch is used to build an approximation $\hat f$ of the integrand $f$, with the constraint that $\int \hat f(x)\mu(\d x)$ is known. The second batch is used to build an importance sampling estimator like \eqref{e:isWeights} but targeting the residual $f-\hat{f}$. By carefully designing $\hat{f}$ and controlling the rate at which both batch sizes grow, \cite[Theorem 2]{OaGiCh17} obtain a mean square error in $N^{-7/6}$ under rather weak assumptions on measure $\mu$. The assumptions on the integrand $f$ are stronger, with $f$ to belong to a specific reproducing kernel Hilbert space. We note that \cite{LiLe17} also propose a similar postprocessing approach, but the improvement on the rate is less explicit.}

\subsection{Quasi-Monte Carlo methods}
\label{s:qmc}
Quasi-Monte Carlo methods (QMC; \citep{DiPi10, DiKuSl13}) are deterministic constructions that
focus on the uniform case, $\mu(\d x) = \d x$ in \eqref{e:quadrature}. The
cornerstone of classical QMC is the Koksma-Hlawka inequality \cite[Equation
3.15]{DiKuSl13}. This inequality bounds the error $\cE_N(f)$ in \eqref{e:error} by the product of the star discrepancy of the nodes and the Hardy-Krause variation of $f$. The star discrepancy measures the
departure of the empirical measure of the $N$ nodes from the uniform measure.
Classical QMC methods aim at proposing efficient node constructions that
minimize this star discrepancy. Some constructions guarantee a star discrepancy that
asymptotically decreases as fast as $N^{-1}\log^{d-1}N$. This implies the same rate for $\cE_N(f)$
provided $f$ has finite Hardy-Krause variation. While this seems faster than typical Monte
Carlo methods in Section~\ref{s:monteCarlo}, the rate as a function of $N$ does
not decrease until $N$ is exponential in $d$. Moreover, the Hardy-Krause variation is hard to manipulate in practice.

Modern QMC methods come up with more practical rates \citep{DiKuSl13}. For example, scrambled nets \citep{Owe97,Owe08} are randomized QMC methods, meaning that a stochastic
perturbation is applied to a deterministic QMC construction. The perturbation is
built so that $\cE_N(f)$ has mean $0$. \cite{Owe97} shows that
only assuming $f$ is $L^2$, the standard deviation of $\cE_N(f)$ is $o(N^{-1/2})$, that is,
converges to zero faster than the traditional Monte Carlo rate. When $f$ is
smooth enough, which requires at least that all mixed partial derivatives of $f$ of order less than $d$ are continuous, \cite{Owe08} further shows that the standard deviation is $\cO(N^{-3/2-1/d}\log^{(d-1)/2} N)$. Again, this rate decreases only when $N$ is exponential in the dimension, but \cite{Owe97} shows that for finite $N$, randomized QMC cannot perform significantly worse than Monte Carlo.

\rev{Finally, we note that nonparametric control variates have also been studied for QMC and randomized QMC \citep{OaGi16}. While bounds on the error still depend on the rather strong hypotheses of QMC, in particular the smoothness of the integrand, this postprocessing has the advantage of partially bypassing the need for the user to know the degree of smoothness in advance.}

\subsection{Bayesian quadrature}
\label{s:bq}
\cite{Hag91} remarked that if we put a Gaussian process prior \citep{RaWi06} over the integrand, then the conditional of its integral given $N$ evaluations is a univariate Gaussian, with a closed-form mean and variance. Picking up nodes by sequentially minimizing this posterior variance yields a range of recent algorithms, such as kernel herding \citep{ChWeSm10} or Bayesian quadrature \citep{HuDu12}. There is empirical evidence \citep{HuDu12} that the error $\cE_N(f)$ in Bayesian quadrature decreases faster than the Monte Carlo rate $N^{-1/2}$. There are theoretical results for hybrid methods between Monte Carlo and Bayesian quadrature \citep{BOGO15}, see Section~\ref{s:discussion} for further discussion.

The nodes and weights of Bayesian quadrature require inverting an $N\times N$ matrix and thus the computational cost of the method is $\cO(N^3)$. Although this cost may seem prohibitive, the approach is justified in some important applications where this cubic computational cost is negligible compared to the cost of one evaluation of the integrand.

\subsection{Our contribution}
\label{s:contribution}
Our main goal is to leverage repulsive particle systems to build
a Monte Carlo method with standard deviation of the error decaying as $o(N^{-1/2})$. More precisely, the idea is to use correlated random
variables for the quadrature nodes, interacting as strongly repulsive particles. Our
motivation comes from specific models in  random matrix theory (see Section
\ref{sec CLT MOPE} for references), for which the linear statistic $\sum
f(x_i)$ converges in distribution to a Gaussian, without requiring any normalizing
factor. In this work, we focus on \emph{determinantal point processes} (DPPs),
which have received a lot of attention recently in probability and related fields, see for instance
\citep{Sos02,Lyo03,HKPV06, Joh06, KuTa12, LaMoRu15} for a general overview.



In any dimension $d$, we construct DPPs generating the
nodes $\bv x_1,\ldots,\bv x_N$ and appropriate weights $w_i$'s so that the error
$\cE_N(f)$ in \eqref{e:error} decreases rapidly, as
$N\to\infty$.
\rev{The general construction of our DPP for an arbitrary $\mu$ is relatively sophisticated, and will be the topic of Section~\ref{sec intro MOPE}. At this stage, we illustrate our results in the specific case where $\mu$ is the uniform measure on the hypercube $I^d$.}

\rev{
\begin{theorem}
  \label{t:legendre} Let $(P_k)_{k\in\N}$ be the Legendre polynomials defined by recurrence,
$$
P_0(x):=1,\quad P_1(x):=x,\quad P_{k+1}(x):=\frac{2k+1}{k+1} xP_k(x)-\frac k{k+1}P_{k-1}(x),
$$
and consider, for any $M\geq 1$ and  $x,y\in\R^d,$  the kernel
$$
K_N(x,y):=\prod_{j=1}^d\sum_{k=0}^{M-1}\big(k+\frac12\big)P_k(x_j)P_k(y_j),
$$
where $N:=M^d$. Let $\bv x_1,\ldots,\bv x_N$ be sampled with density on $I^d$
\eq
\label{preDPP}
\frac1{N!}\det\Big[K_N(x_k,x_\ell)\Big]_{k,\ell=1}^N.
\qe
Then, for any $\mathscr C^1$ function $f$ that is compactly supported within the open hypercube $(-1,1)^d$,
$$
\E\left[\,\sum_{i=1}^N  \frac{f(\bv x_i)}{K_N(\bv x_i,\bv x_i)}\right]=\int_{I^d} f(x) \,\d x
$$
and the quadratic error satisfies
$$
\E\left[\,\sum_{i=1}^N  \frac{f(\bv x_i)}{K_N(\bv x_i,\bv x_i)}-\int_{I^d} f(x) \,\d x\right]^2=\cO\left( \frac1{N^{1+1/d}}\right),
$$
as $N\to\infty$. Moreover, we have a central limit theorem:
$$
\sqrt{N^{1+1/d}}\left(\,\sum_{i=1}^N  \frac{f(\bv x_i)}{K_N(\bv x_i,\bv x_i)}-\int_{I^d} f(x) \,\d x\right)\xrightarrow[N\to\infty]{law} \mathcal N(0,\Omega_{f,1}^2)\, ,
$$
where $\Omega_{f,1}$ has an explicit expression involving the regularity of $f$, see \eqref{omega f} below.
\end{theorem}
}
\rev{
Sampling the distribution in \eqref{preDPP}, which is called a multivariate Legendre ensemble, can be done exactly in time $\cO(N^3)$, if we neglect the cost of rejection sampling steps, see Section~\ref{s:sampling}. For the sake of illustration, a sample is shown in Figure~\ref{f:legendre2D} in the case $d=2$. For each $i$, the area of the disk centered at $\bx_i$ is proportional to its weight $1/K_N(\bx_i,\bx_i)$ in the quadrature rule of Theorem~\ref{t:legendre}. Points are well-spread throughout the square, more so than under independent uniform sampling, with a visible accumulation along the border of the square compensated by smaller weights. The green lines on the side plots show the marginals of the uniform measure on the square. The histograms on the side plots are weighted empirical histograms of the sample.
}
\begin{figure}
  \centering
  \subfigure[2D sample]{
  \includegraphics[width=\twofig]{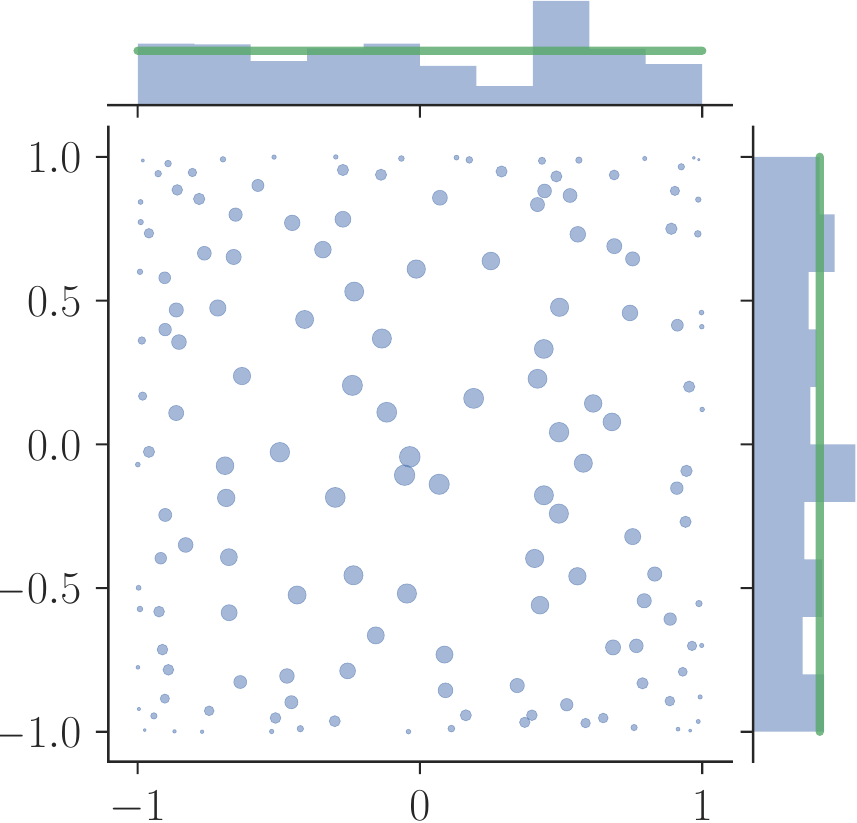}
  \label{f:legendre2D}
  }
  \hfill
  \subfigure[The graded lexicographic order]{
  \includegraphics[width=.4\textwidth]{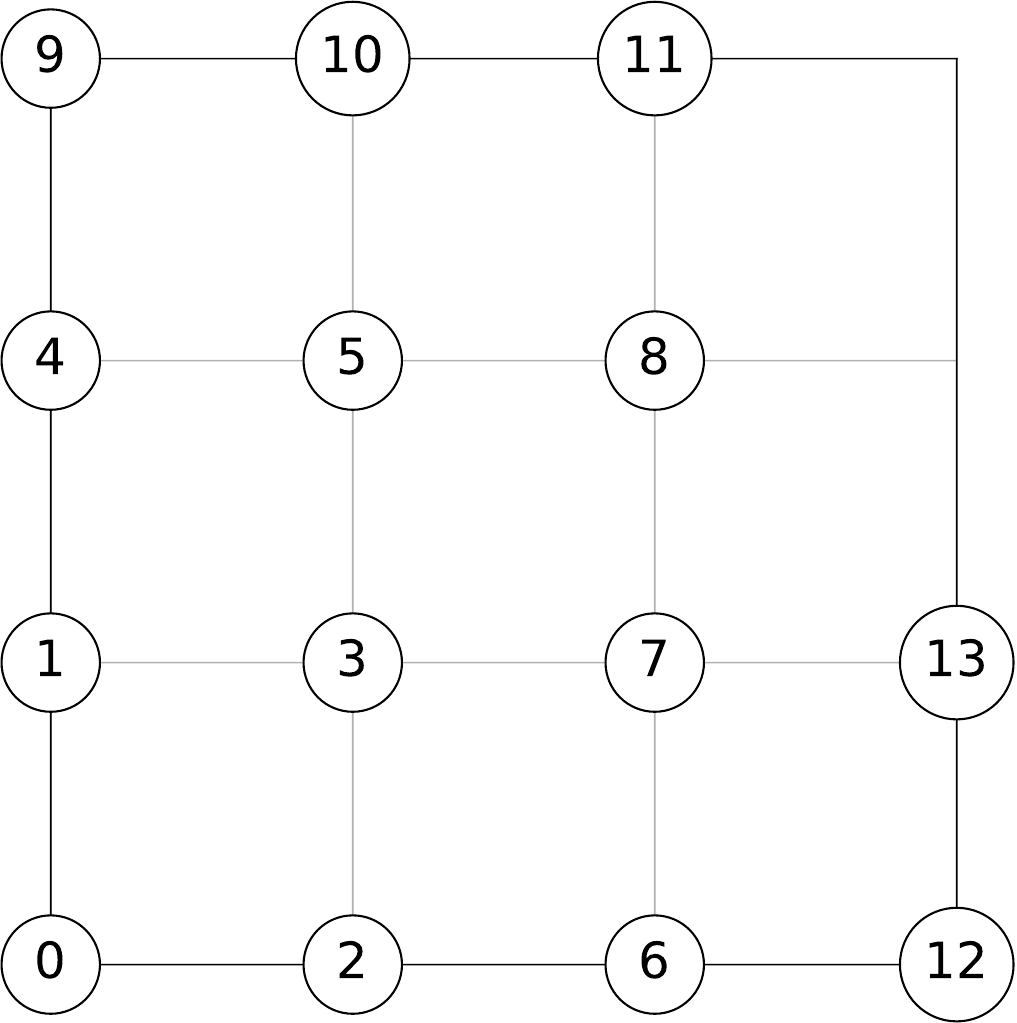}
  \label{f:order}
  }
  \caption{\rev{\ref{f:legendre2D} A weighted sample of size $N=150$ of the bivariate Legendre ensemble. \ref{f:order} The graded lexicographic order in $d=2$: the lower left corner marks the origin $(0,0)\in\mathbb{N}^2$, and the circle with mark $k$ is at coordinates $\frak b(k)$.}}
\end{figure}

In this paper we prove a more general result than Theorem~\ref{t:legendre}, where we consider a general class of measures instead of the uniform, and relax the assumption that the number of nodes $N$ is a $d$th power, see Theorem~\ref{DPPMC Th1}.
For the latter, we need to introduce an ordering on multivariate monomials, see Section~\ref{sec CLT MOPE}. We also provide an importance sampling (Theorem~\ref{DPPMC Th2}) and self-normalized importance sampling (Theorem~\ref{DPPMC Th3}) version of this result, having  applications to Bayesian inference in mind.


Computationally, our method requires sampling from a DPP. The standard algorithm
 is $\cO(N^3)$ \citep{HKPV06}, plus some overhead due to using rejection sampling routines. As it stands, the applications of Monte Carlo with
 DPPs are thus naturally the same settings as Section~\ref{s:bq}, where the
 bottleneck is the evaluation of the integrand, not the generation of the nodes.
 Such settings arise in Bayesian inference for simulation-heavy sciences such as astrophysics \citep{Tro06}, ecology \citep{PSHNTHE13}, or cell biology \citep{FNCFKSTZSCN11}, where one evaluation of the integrand requires numerically solving large systems of differential equations. We comment in
 Section~\ref{s:sampling} on open avenues for faster sampling of specific DPPs,
 which would further widen the applicability of Monte Carlo with DPPs.

It turns out that, when $d=1$, our method is formally very similar to Gaussian
quadrature, as described in Section~\ref{s:quadrature}. We basically replace the
zeros of orthogonal polynomials by particles sampled from \emph{orthogonal polynomial
ensembles} (OP Ensembles), DPPs whose building blocks are orthogonal polynomials. Our
contribution also has the advantage of generalizing more
naturally to higher dimensions than Gaussian quadrature through
multivariate OP Ensembles.

Monte Carlo with DPPs is to be classified somewhere between classical Monte Carlo
methods and QMC methods, respectively described in Sections~\ref{s:monteCarlo}
and \ref{s:qmc}. It is very much similar to importance sampling, but with
negatively correlated nodes. Simultaneously, it is more Monte Carlo than
scrambled nets, as it does not randomize \emph{a posteriori} a low discrepancy
deterministic set of points, but rather incorporate the low discrepancy constraint
into the randomization procedure. Our approach also bears similarity with postprocessing approaches to Monte Carlo with fast convergence of Section~\ref{s:monteCarlo}. Indeed, the fast convergence in our CLT is due to approximation results, much like nonparametric control variates. We further comment on this in Section~\ref{s:discussion}.

The rest of the paper is organized as follows. In Section~\ref{s:results}, we
state our quadrature rules and theoretical results on the convergence of its
error. In Section~\ref{s:experiments}, we demonstrate our results in simple
experimental settings. We conclude with some
perspectives in Section~\ref{s:discussion}. In Appendix~\ref{s:outlineOfProofs}, we introduce key technical notions and give the outline of our proofs, the technical parts of the proofs being detailed in Appendices~\ref{proof th1 section} and \ref{proof th2 section}. Appendix~\ref{a:figures} contains additional experimental results.

\section{Statement of the results}
\label{s:results}

\paragraph{Notation.}
All along this work, we write for convenience $I=[-1,1]$ and $ I^d=[-1,1]^d$.
Also, for any $0<\epsilon<1$, we set  $I_\epsilon=[-1+\epsilon,1-\epsilon]$ and
$I_\epsilon^d=[-1+\epsilon,1-\epsilon]^d$. Finally, except when specified
otherwise, a \emph{reference measure} is a positive finite Borel measure with support  \rev{ $\Supp(\mu)\subset I^d$}.

\subsection{Determinantal point processes and multivariate OP Ensembles}

\label{sec intro MOPE}
\subsubsection{Point processes and determinantal correlation functions}

A \textit{simple point process} (hereafter \emph{point process}) on $I^d$ is a
probability distribution $\P$ on  finite subsets $\bv S$ of $I^d$. A classical exhaustive reference is \citep{DaVe03}, but we also refer the reader to \citep{HKPV06}, which is shorter and contains everything needed in this paper. Given a reference measure $\mu$, a point process has a $n$-\textit{correlation function} $\rho_n$ if one has
\eq
\label{def k cor}
\mathbb E\left[\sum_{ \substack{ \rev{ \bv x_{i_1}, \ldots, \bv x_{i_n}\in\bv S}\\
\bv x_{i_1}\neq \cdots\neq \bv x_{i_n} }}
\phi( \bv x_{i_1}, \ldots, \bv x_{i_n})\right]=\int_{(I^d)^n}\phi(x_1,\ldots,x_n)\rho_n(x_1,\ldots,x_n)\mu^{\otimes n}(\d x_1,\ldots,\d x_n)
\qe
for every bounded Borel function $\phi:(I^d)^n\to\R$, where the sum in \eqref{def k
  cor} ranges over all pairwise distinct $n$-uplets of the random  finite subset $\bv S$. The function $\rho_n$, provided it exists, thus encodes the correlations between distinct $n$-uplets  of the random  set $\bv S$. For instance, \rev{a Poisson process with intensity $\lambda:I^d\rightarrow \mathbb{R}_+$ is characterized by
  \eq
  \rho_n(x_1,\dots,x_n) = \prod_{i=1}^n \lambda(x_i)
  \label{e:poissonCorr}
  \qe
   and $\d \mu(x) = \bs{1}_{I^d}(x)\d x$. In that particular case, the correlation functions \eqref{e:poissonCorr} are products of univariate terms, which can be paraphrased as \emph{there is no interaction between points in a Poisson point process}. Finally, our definition \eqref{def k cor} is easily seen to be equivalent to \cite[Definition 1]{HKPV06}, where correlation functions are also called \emph{joint intensities}. For ease of reference, we also note that correlation functions are called \emph{factorial moment densities} in \cite[Section 5.4]{DaVe03}}.

A point process is \textit{determinantal} (DPP) if there exists an appropriate kernel $K:I^d\times I^d\to\R$ \rev{or $\C$} such that the $n$-correlation function exists for every $n$ and reads
\eq
\label{def cor DPP}
\rho_n(x_1,\ldots,x_n)=\det\Big[K(x_k,x_\ell)\Big]_{k,\ell=1}^n,\qquad x_1,\ldots,x_n\in I^d.
\qe
\rev{The kernel of a DPP thus encodes how the points in the random configurations interact. The existence of a point process with \eqref{def cor DPP} as its correlation functions is, in general, a difficult question. It is easy to see that the kernel has to be positive definite, so that the right-hand side of \eqref{def cor DPP} is always non-negative. But non-negativity is not sufficient for \eqref{def cor DPP} to consistently define a point process.}

A canonical way to construct DPPs is to define so-called \emph{projection DPPs}, which generate configurations of $N$ points $\P$-almost surely, i.e. $\bv S=\{\bv
x_1,\ldots,\bv x_N\}$. More precisely, consider $N$ orthonormal functions
$\phi_0,\ldots,\phi_{N-1}$ in $L^2(\mu)$, that is, $\int
\phi_k(x)\phi_\ell(x)\mu(\d x)=\delta_{k\ell}$, and take for kernel
\eq
\label{kernel DPP}
K_N(x,y)=\sum_{k=0}^{N-1}\phi_k(x)\phi_k(y).
\qe
In this setting, it turns out that the (permutation invariant) random variables $\bv x_1,\ldots,\bv x_N$ with joint probability distribution
\eq
\label{density DPP proj}
\frac1{N!}\det\Big[K_N(x_i,x_\ell)\Big]_{i,\ell=1}^N\prod_{i=1}^N\mu(\d x_i)
\qe
generate a DPP with kernel $K_N(x,y)$. \rev{\cite[Section 2]{HKPV06} gives a proof that \eqref{density DPP proj} yields \eqref{def cor DPP} with $K=K_N$.}

For further information on determinantal
point processes, we refer the reader to \citep{Mac75, HKPV06, Joh06, Sos00b, Lyo03,KuTa12,LaMoRu15}.

\subsubsection{Multivariate OP Ensembles}
\label{def MOP}
In the one-dimensional setting, we can for instance build a DPP using
\eqref{density DPP proj} with
$\phi_0,\text{...},\phi_{N-1}$ the $N$ lowest degree orthonormal polynomials
associated with the reference measure $\mu$. Such DPPs are known as \emph{OP Ensembles} and have been popularized by random matrix theory, see e.g.
\citep{Kon05} for an overview.

Our contribution involves a higher-dimensional generalization of OP Ensembles,
relying on multivariate orthonormal polynomials, which we now introduce. Given a
reference measure $\mu$, assume it has well-defined multivariate orthonormal
polynomials, meaning that $\int P^2(x) \mu(\d x)>0$ for every non-trivial polynomial
$P$. This is for instance true if $\mu(A)>0$ for some non-empty open set
$A\subset I^d$. Now choose an ordering for the multi-indices
$(\alpha_1,\ldots,\alpha_d)\in\N^d$, that is, pick a bijection $\frak b:\N\to\N^d$. This gives an ordering of the monomial functions $(x_1,\ldots,x_d)\mapsto x_1^{\alpha_1}\cdots x_d^{\alpha_d}$,
to which one applies the Gram-Schmidt algorithm. This yields a sequence of
orthonormal polynomial functions $(\phi_k)_{k\in\N}$, the multivariate
orthonormal polynomials. In this work, we use a specific
bijection $\frak b$ defined in Section~\ref{graded lex order}.

Equipped with this sequence $(\phi_k)_{k\in\N}$ of multivariate orthonormal
polynomials,  we finally consider for every $N$ the DPP associated with the
associated kernel \eqref{kernel DPP}, that we refer to as the \emph{multivariate OP
Ensemble} associated with a reference measure $\mu$. When $d=1$, it reduces to
the classical OP Ensemble.

\subsubsection{The graded lexicographic order and the bijection $\frak b$}
\label{graded lex order}
We consider the bijection $\frak b$ associated with the graded (with respect to
the sup norm) alphabetic order on $\N^d$. We start with the usual lexicographic
order on $\N^d$, defined by saying that $(\alpha_1,\ldots,\alpha_d)<_{{\rm
    lex}}(\beta_1,\ldots,\beta_d)$ if there exists $j\in\{1,\ldots,d\}$ such
that $\alpha_i=\beta_i$ for every $i<j$ and $\alpha_j<\beta_j$. Now we define
the \emph{graded lexicographic order} as follows. We say that $(\alpha_1,\ldots,\alpha_d)<(\beta_1,\ldots,\beta_d)$ if either $\max\{\alpha_1,\ldots,\alpha_d\}<\max\{\beta_1,\ldots,\beta_d\}$ or $
\max\{\alpha_1,\ldots,\alpha_d\}=\max\{\beta_1,\ldots,\beta_d\}$ and $ (\alpha_1,\ldots,\alpha_d)<_{{\rm lex}}(\beta_1,\ldots,\beta_d)$.
Moreover, from now on we specify the bijection $\frak b$ to be the unique
bijection $\N\to\N^d$ increasing for this order. Otherly put, set $\frak b(0)=(0,\ldots,0)$ and
$\frak b(n)=\min \;\N^d\setminus\big\{\frak b(0),\ldots, \frak b(n-1)\big\}$ by
induction, where the minimum refers to the graded lexicographic order. An
important feature of this ordering on which our proofs rely is that, for every $M\geq 1$, the set of the first $M^d$ indices $\{\frak b(0),\ldots,\frak b(M^d-1)\}$ matches the  discrete hypercube
\eq
\label{def CM}
\mathcal C_M= \Big\{\bs n\in\N^d :\; 0\leq n_1,\ldots, n_d\leq M-1\Big\}.
\qe
The indices between $\frak b(M^d-1)$ and $\frak b(M^{d+1}-1)$ then fill the layer $\mathcal C_{M+1}\setminus \mathcal C_M$ by following the usual lexicographic order. \rev{For better intuition, we illustrate the order $\frak b$ for $d=2$ in Figure~\ref{f:order}; observe how each layer is filled one after the other.}

We are now in position to state our first result on multivariate OP Ensembles,
which is the cornerstone for the Monte Carlo methods we introduce later in Section~\ref{s:dppmc}.

\subsection{A central limit theorem for multivariate OP Ensembles}
\label{sec CLT MOPE}
Several central limit theorems (CLTs) have been obtained for determinantal point processes and related
models in random matrix theory, but only when the random configurations lie in a one- or two-dimensional
domain. See for instance \citep{Joh97, Joh98, DiEv01, Sos02, Pas06, Rivi07,
  Pop09, KrSh10, AmHeMa11, AmHeMa15, Ber12, Shc13, BrDu13, BrDu14, JoLa15,
  Lam15a, Lam15b} for a non-exhaustive list. Although DPPs on higher-dimensional supports have attracted attention in complex
geometry \citep{Ber09,Ber09b, Ber13, Ber14}, in statistics \citep{LaMoRu15, MNPR15}, and in physics \citep{ToScZa08, ScZaTo09},
it seems no CLT has been established yet when $d\geq 3$.

Our first result for multivariate OP Ensembles is a CLT for $\mathscr C^1$ test
functions when the reference measure $\mu$ is a product of $d$ Nevai-class
probability measures on $I$. The exact definition of the Nevai class is
postponed until Definition \ref{def Nevai class}, but we now give a simple
sufficient condition. As a
consequence of Denisov--Rakhmanov's theorem (see Theorem \ref{DR theorem}), if a
measure on $I$ has for Lebesgue decomposition $\mu(\d x)=\omega(x)\d x +
\mu_{s}$ (where $\mu_s$ is orthogonal to the Lebesgue measure) with
$\omega(x)>0$ almost everywhere, then  $\mu$ is Nevai-class. Denote by $(T_k)_{k\in\mathbb{N}}$
the normalized Chebyshev polynomials, defined on $I$ by
$$
T_0=1,\qquad T_k(\cos\theta)=\sqrt 2\cos(k\theta),\qquad k\geq 1.
$$

\begin{theorem}
\label{th CLT general}
Let $\mu$ be a reference measure with $\Supp(\mu)\subset I^d$, and assume
$\mu=\mu_1\otimes \cdots\otimes\mu_d$ where each $\mu_j$ is Nevai class (see
Definition \ref{def Nevai class}). If $\bv x_1,\ldots,\bv x_N$ stands for the
associated multivariate OP Ensemble associated with $\mu$, then for every $f\in\mathscr C^1(I^d,\R)$, we have
$$
 \frac1{\sqrt{N^{1-1/d}}}\left(\ \sum_{i=1}^Nf(\bv x_i) -  \E\left[\,\sum_{i=1}^Nf(\bv x_i)\right]\right)\xrightarrow[N\to\infty]{law } \mathcal N(0,\sigma_f^2),
$$
where
\eq
\sigma_f^2=\frac12\sum_{k_1,\ldots,k_d=0}^\infty (k_1+\cdots+k_d){ \hat f(k_1,\ldots,k_d)}^2
\qe
and
\eq
\label{f hat}
\hat f(k_1,\ldots,k_d)=\int_{I^d} f(x_1,\ldots,x_d)\prod_{j=1}^d T_{k_j}(x_j) \frac{\d x_j}{\pi\sqrt{1-x_j^2}}.
\qe
\end{theorem}

When $d=1$, Theorem \ref{th CLT general} was obtained by \cite{BrDu13}, see also
\citep{Lam15b} for an alternative proof, but the higher-dimensional case $d\geq
2$ is novel. We shall restrict to $d\geq 2$ for the proof of the theorem, which
is deferred to Section \ref{proof th1 section}. Let us now make a few remarks
concerning the statement of Theorem \ref{th CLT general}.

\begin{remark} The limiting variance $\sigma_f^2$ does not depend on the reference measure $\mu$. \end{remark}

\begin{remark} By making the change of variables $x_j=\cos\theta_j$, we obtain
$$
\hat f(k_1,\ldots,k_d)=\frac {(\sqrt 2)^{|\{ j :\; k_j\neq 0 \}|}}{\pi^{d}}\int_{[0,\pi]^d} f(\cos\theta_1,\ldots,\cos\theta_d)\prod_{j=1}^d \cos(k_j\theta_j)\d \theta_j,
$$
which is, up to a multiplicative factor, a usual Fourier coefficient.
\end{remark}

\rev{As a side note, we obtain that the limiting variance in Theorem~\ref{th CLT general} is dominated by an explicit integral, that may be of interest to bound $\sigma_f^2$ in practice.}
\begin{proposition}
\label{bound limit var}
For any $f\in\mathscr C^1(I^d,\R)$,  we have the inequality
\eq
\label{Poincare ineq2}
\sigma_f^2   \leq \frac 1{2} \sum_{\alpha=1}^d \int_{I^d}  \Big( \sqrt{1-x_\alpha^2}\,\partial_\alpha f(x_1,\ldots,x_d)\Big)^2 \,\prod_{j=1}^d\frac{\d x_j}{\pi\sqrt{1-x_j^2}}.
\qe
\end{proposition}

It will appear from the proof we provide in Section~\ref{s:boundOnLimitVar} that this inequality is sharp, since equality holds whenever $f$ is a linear combination of monomials $x_1^{\alpha_1}\cdots x_d^{\alpha_d}$ with $\alpha_j\in\{0,1\}$; see \eqref{e:identityToProve}.

We now turn to Monte Carlo methods based on Theorem~\ref{th CLT general}.

\subsection{Monte Carlo methods based on determinantal point processes}
\label{s:dppmc}
Consider a reference measure $\mu$ with support inside $I^d$, having
well-defined multivariate orthonormal polynomials (say, $\mu(A)>0$ for some open
set $A\subset I^d$). Let $K_N(x,y)$ be the $N$th Christoffel-Darboux kernel for the associated multivariate OP Ensemble, namely
\eq
K_N(x,y)=\sum_{k=0}^{N-1}\phi_k(x)\phi_k(y),
\label{e:CDKernel}
\qe
where $(\phi_k)_{k\in\N}$ is the sequence of multivariate orthonormal
polynomials associated with $\mu$ and the graded lexicographic order, see
Section~\ref{graded lex order}. Our
quadrature rule is as follows: take for nodes $\bv x_1,\ldots,\bv x_N$ the
random points coming from the multivariate OP Ensemble, namely with joint
density \eqref{density DPP proj}, and for weights $w_i=1/K_N(\bv x_i,\bv x_i)$.
Thus, for any $\mu$-integrable function $f$, our estimator of $\int f (x)\mu(\d x)$ reads
\eq
\label{estimator}
\sum_{i=1}^N \frac{f(\bv x_i)}{K_N(\bv x_i,\bv x_i)}.
\qe
One can readily see by taking $n=1$ in \eqref{def k cor}--\eqref{kernel DPP}
that the estimator \eqref{estimator} is unbiased,
\eq
\E \left[\ \sum_{i=1}^N \frac{f(\bv x_i)}{K_N(\bv x_i,\bv x_i)}\ \right] = \int f(x)\mu(\d x).
\label{e:weightedUnbiasedness}
\qe

\begin{remark} For $d=1$, comparing \eqref{estimator} to \eqref{weight
    Gauss}--\eqref{CD Gauss} yields that our method matches Gaussian quadrature
  except for the nodes, since we replace the zeros of the univariate orthogonal
  polynomial $\phi_N$ by random points drawn from an OP Ensemble. In fact, this
  replacement is not aberrant since zeros of orthogonal polynomials and
  particles of associated OP Ensembles get arbitrarily close with high
  probability as $N\to\infty$, see \citep{Har15} for further information and
  generalizations. Notice however that our quadrature rule has the advantage to
  make sense in any dimension $d$.
\label{r:gauss}
\end{remark}

Our next result is a CLT for \eqref{estimator}, thus giving a precise rate of decay for the error made in the approximation, provided we make regularity assumptions on $\mu$ and on the class $\mathscr C$ of test functions $f$. More precisely, recalling the notation $I_\epsilon^d=[-1+\epsilon,1-\epsilon]^d$, we consider
\eq
\label{class test functions}
\mathscr{C}=\Big\{ f\in \mathscr C^1(I^d,\R) :\; {\rm Supp}(f)\subset I_\epsilon^d \mbox{ for some }\epsilon>0\Big\}.
\qe
As for the reference measure, we shall assume $\mu$ is a product measure with a
density which is $\mathscr C^1$ and positive on the open set $(-1,1)^d$. Set for convenience
\eq
\label{e:muEqTenseur}
\omega_{eq}^{\otimes d}(x)=\prod_{j=1}^d\frac{1}{\pi \sqrt{1-x_j^2}} \ ,\qquad x\in I^d.
\qe

\begin{theorem}[\rev{Crude Monte Carlo with OP Ensembles}]
\label{DPPMC Th1}
Let $\mu(\d x)=\omega(x)\d x$ be a product reference measure with $\omega(x)=\omega_1(x_1)\cdots\omega_d(x_d)$ and  $\Supp(\mu)\subset I^d$. Assume $\omega$ is $\mathscr C^1$ and positive on the open set $(-1,1)^d$, and satisfies: \rev{for every $\epsilon>0$,
\eq
\label{a:assump2}
\frac1N\sup_{x\in I_\epsilon^d}\left|\nabla K_N(x,x)\right|<\infty.
\qe}
If $\bv x_1,\ldots,\bv x_N$ stands for the multivariate OP Ensemble associated with $\mu$, then for every $f\in\mathscr C$,
we have for the mean square error of the estimator,
\eq
\lim_{N\to\infty}N^{1+1/d}\ \mathbb E\left[\left(\ \sum_{i=1}^N \frac{f(\bv x_i)}{K_N(\bv x_i,\bv x_i)}- \int f(x)\mu(\d x)\right)^2\ \right]=\Omega_{f,\omega}^2\ .
\qe
where, see \eqref{f hat},
\eq
\label{omega f}
\Omega_{f,\omega}^2= \frac1{2}\sum_{k_1,\ldots,k_d=0}^\infty (k_1+\cdots+k_d)  \widehat{ \left(\frac{f \omega}{\omega_{eq}^{\otimes d}}\right)}(k_1,\ldots,k_d)^2.
\qe
Moreover, we have the central limit theorem,
$$
\sqrt{N^{1+1/d}}\left(\ \sum_{i=1}^N \frac{f(\bv x_i)}{K_N(\bv x_i,\bv x_i)}- \int f(x)\mu(\d x)\right)\xrightarrow[N\to\infty]{law } \mathcal N\big(0,\Omega_{f,\omega}^2\big),
$$
\end{theorem}

\rev{We will discuss the assumptions of Theorem \ref{DPPMC Th1} in Section \ref{reg
  cond Th2} but let us already state that, as it will appear in the proof, \eqref{a:assump2}  can be replaced by the weaker  but technically involved Assumption~\ref{a:regularity assumption}. We restricted ourselves to \eqref{a:assump2} for the sake of presentation, as it already covers the Jacobi case. Indeed, we prove the following result in Section~\ref{reg cond Th2}.}
\begin{proposition}
\label{Jacobi = regular}
Given any parameters $\alpha_1,\beta_1,\ldots,\alpha_d,\beta_d>-1$, the
reference measure
\eq
\label{Jacobi measure}
\mu(\d x)=\prod_{j=1}^d(1-x_j)^{\alpha_j}(1+x_j)^{\beta_j} \bs 1_{I}(x_j)\d x_j,
\qe
 satisfies the assumptions of Theorem \ref{DPPMC Th1}.
\end{proposition}

Hereafter, we call measures of the form \eqref{Jacobi measure} \emph{Jacobi
  measures}. From a practical point of view, Theorem~\ref{DPPMC Th1} requires
knowledge on
the measure $\mu$, in particular all its moments should be known, since we need
the corresponding orthonormal polynomials. This is the case for most
applications of Gaussian quadrature, where the reference measure is such that
orthonormal polynomials are computable, like Jacobi measures
\eqref{Jacobi measure} for instance. \rev{When the moments of $\mu$ are not known or when $\mu$ is not separable, we propose an importance sampling result in Theorem~\ref{DPPMC Th2}, which shifts most hypotheses onto an instrumental density $q$. Note however that Theorem~\ref{DPPMC Th2} still requires that we can evaluate the density $\omega$ of $\mu$ pointwise.
}
\begin{theorem}[\rev{Importance sampling with OP Ensembles}]
\label{DPPMC Th2}
Let $\mu(\d x)=\omega(x)\d x$ be a reference measure on $I^d$ with a $\mathscr
C^1$ density $\omega$ on the open set $(-1,1)^d$.  Consider a  measure  $q(x)\d
x$ satisfying the assumptions of Theorem \ref{DPPMC Th1}, let $K_N(x,y)$ be the
$N$th Christoffel-Darboux kernel associated with $q(x)\d x$, and $\bv x_1,\ldots,\bv x_N$ the associated multivariate OP Ensemble. Then, for every $f\in\mathscr C$, we have
\eq
\E \left[\ \sum_{i=1}^N \frac{f(\bv x_i)}{K_N(\bv x_i,\bv x_i)}\frac{\omega(\bv x_i)}{q(\bv x_i)} \ \right]= \int f(x) \mu(\d x) ,
\qe
and we have for the mean square error of the estimator,
\eq
\lim_{N\to\infty}N^{1+1/d}\ \mathbb E\left[\left(\ \sum_{i=1}^N \frac{f(\bv x_i)}{K_N(\bv x_i,\bv x_i)}\frac{\omega(\bv x_i)}{q(\bv x_i)} - \int f(x)\mu(\d x)\right)^2\ \right]=\Omega_{f,\omega}^2\ ,
\qe
where $\Omega_{f,\omega}^2$ is the same as \eqref{omega f}.  Moreover, we have the central limit theorem,
\eq
\sqrt{N^{1+1/d}}\left( \ \sum_{i=1}^N \frac{f(\bv x_i)}{K_N(\bv x_i,\bv
    x_i)}\frac{\omega(\bv x_i)}{q(\bv x_i)} - \int f(x) \mu(\d
  x)\right)\xrightarrow[N\to\infty]{law } \mathcal
N\big(0,\Omega_{f,\omega}^2\big).
\label{e:weightedCLT}
\qe
\end{theorem}

Indeed, Theorem \ref{DPPMC Th2} follows from Theorem \ref{DPPMC Th1} by taking
$f\omega/q$ for test function with $f\in\mathscr C$ and $q(x)\d x$ for reference
measure.

\begin{remark}
From a classical importance sampling perspective, it is surprising that the limiting variance in \eqref{e:weightedCLT} does not depend on the proposal density $q$.
\end{remark}


\rev{
In most applications to Bayesian inference, $\mu$ is a probability measure, but its density $\omega$ can only be evaluated up to a multiplicative constant. A classical trick is to rely on self-normalized importance sampling. Theorem~\ref{DPPMC Th3} states a central limit theorem for such an estimator.
\begin{theorem}[Self-normalized importance sampling with OP Ensembles]
\label{DPPMC Th3}
Let $\mu(\d x)=\omega(x)\d x = \omega_u(x) \d x/Z$ be a reference probability measure on $I^d$ with a $\mathscr C^1$ density $\omega$ on the open set $(-1,1)^d$. We further assume that $\mu$ is supported in $I_\epsilon^d = [-1+\epsilon,1-\epsilon]^d$. As in Theorem~\ref{DPPMC Th2}, consider a measure $q(x)\d x$ satisfying the assumptions of Theorem \ref{DPPMC Th1}, let $K_N(x,y)$ be the $N$th Christoffel-Darboux kernel associated with $q(x)\d x$, and $\bv x_1,\ldots,\bv x_N$ the associated multivariate OP Ensemble. Finally, for convenience, we let
$$ \mathscr{I}_N(g) = \sum_{i=1}^N \frac{g(\bv x_i)}{K_N(\bv x_i,\bv
    x_i)}\frac{\omega_u(\bv x_i)}{q(\bv x_i)}, \quad g:I^d\rightarrow\mathbb{R}.$$
Then, for every $f\in\mathscr C$, we have
\begin{align}
\sqrt{N^{1+1/d}}\left( \mathscr{I}_N(f)/\mathscr{I}_N(1) - \int f(x) \mu(\d
  x)\right) \xrightarrow[N\to\infty]{law } \mathcal
N\big(0,\Xi_{f,\omega}^2\big),
\label{e:selfNormalizedWeightedCLT}
\end{align}
where $$\Xi_{f,\omega}^2 = \Omega^2_{f,\omega} -2c_{f,\omega}\int f\d\mu + \Omega_{1,\omega}^2\left(\int f\d\mu\right)^2 \geq 0,$$
and
$$ c_{f,\omega} = \frac1{2}\sum_{k_1,\ldots,k_d=0}^\infty (k_1+\cdots+k_d)  \widehat{ \left(\frac{f \omega}{\omega_{eq}^{\otimes d}}\right)}(k_1,\ldots,k_d) \times \widehat{ \left(\frac{\omega}{\omega_{eq}^{\otimes d}}\right)}(k_1,\ldots,k_d).$$
\end{theorem}
}
\rev{
We prove Theorem~\ref{DPPMC Th3} in Section~\ref{proofSelfnormal} using the same arguments as for classical importance sampling \cite[Section 7.1.3]{Sch12}, but replacing the standard CLT by Theorem~\ref{DPPMC Th2}.
}

\subsection{Sampling a multivariate OP Ensemble}
\label{s:sampling}
For Monte Carlo with DPPs to be a practical tool, we need to be able to sample
realizations of the random variables $\bv x_1,\ldots,\bv x_N$ with joint density \eqref{density DPP
  proj}. \cite{HKPV06} give an algorithm for sampling generic DPPs, which we use
here; see also \citep{ScZaTo09, LaMoRu15, OlNaTr15} for more details. In terms of code, \rev{a companion Python package to the current paper is available\footnote{\url{https://github.com/rbardenet/dppmc}}, which implements the OPE sampling described in this section and used later for the experiments in Section~\ref{s:experiments}}. A more efficient implementation of the same OPE sampling algorithm, along with most known DPP sampling algorithms, can also be found in the Python package DPPy\footnote{\url{https://github.com/guilgautier/DPPy}} \citep*{GaBaVa18Sub}.

The algorithm is based on the fact that the chain rule for the joint distribution \eqref{density DPP
  proj} is available as
\eq
\label{joint density DPP}
\frac1{N!}\det\Big[K_N(x_i,x_\ell)\Big]_{i,\ell=1}^N \prod_{i=1}^N\mu(\d x_i)= \prod_{i=1}^N
\frac{1}{N-i+1} \Big\Vert P_{H_{i-1}} K_N(x_i,\cdot)\Big\Vert_{L^2(\mu)}^2 \mu(\d x_i).
\qe
In \eqref{joint density DPP}, $P_H$ is the orthogonal projection onto a subspace $H$ of $L^2(\mu)$,
$$ H_0 = \text{Span}(\phi_0,\dots,\phi_{N-1}),$$
and $H_{i-1}$ is the orthocomplement in $H_0$ of
$$
\text{Span}\left(K_N(x_\ell,\cdot), \; 1\leq \ell\leq i-1\right)
$$
for every $i>1$. In particular, all the terms in the product of the RHS of \eqref{joint density
  DPP} are probability measures \citep[Proposition 19]{HKPV06}. Notice that the factorization \eqref{joint density DPP} is the equivalent of the ``base
times height'' formula that computes the squared volume of the parallelotope
generated by the vectors $(\phi_0(x_i),\ldots,\phi_{N-1}(x_i))$ for $1\leq
i\leq N$.

Using the normal equations, we can also rewrite
each term in the product \eqref{joint density DPP}
\eq
\Big\Vert P_{H_{i-1}} K_N(x_i,\cdot)\Big\Vert_{L^2(\mu)}^2 =
\begin{cases}
K_N(x_1,x_1) & \text{if } i=1,\\
 K_N(x_i,x_i) - {\bf k}_{1:i-1}(x_i)^T {\bf K}_{1:i-1}^{-1} {\bf k}_{1:i-1}(x_i) & \text{else,}
\end{cases}
\label{e:chainRuleGPForm}
\qe
where
$$ {\bf k}_{1:i-1}(\cdot) = \left( K_N(x_1,\cdot),\dots,K_N(x_{i-1},\cdot) \right)^T $$
and
$$ {\bf K}_{1:i-1} = \Big[K_N(x_k,x_\ell)\Big]_{1\leq k,\ell \leq i-1}.$$
\begin{remark}
Equation~\eqref{e:chainRuleGPForm} will be familiar to users of Gaussian
processes (GPs; \citealt{RaWi06}): the unnormalized conditional densities
\eqref{e:chainRuleGPForm} are the incremental posterior variances in a GP model
with the same kernel.
\end{remark}
\rev{
\begin{remark}
Evaluating \eqref{e:chainRuleGPForm} requires evaluating $K_N$, or equivalently the polynomials $\phi_k$ for $k=0,\dots,N-1$. This can be efficiently implemented using the three-term recurrence relations for orthogonal polynomials, when the recurrence coefficients are known; see e.g. \citep[Section 1.3]{Gau04} for whom this recurrence is ``arguably the single most important piece of information for the constructive and computational use of orthogonal polynomials".
\end{remark}
}
\rev{In a nutshell, sampling a multivariate OPE amounts to sampling from each conditional \eqref{e:chainRuleGPForm} in the chain rule \eqref{joint density DPP}, one after the other. The only thing left to specify is how we sample each conditional. In this paper, we propose to sample each conditional by rejection sampling \citep[Section 2.3]{RoCa04}.} This requires proposal densities $(q_i)_{1\leq
  i\leq N}$ and tight bounds on the density ratios
\eq
\label{e:toBound}
\frac{\Big\Vert P_{H_{i-1}} K_N(x,\cdot)\Big\Vert_{L^2(\mu)}^2 \omega(x)}
{q_i(x)}, \quad 1\leq i \leq N,
\qe
when $\mu(\d x)=\omega(x)\d x$. \rev{A theorem of Totik, which we recall later as Theorem~\ref{Totik asymp unif}, gives light conditions on $\omega$, under which}
\rev{
$$
\frac{N}{K_N(x,x)}\rightarrow \frac{\omega(x)}{\omega_{\text{eq}}(x)},
$$
}
\rev{
uniformly on $I_\epsilon^d$. This suggests choosing
$$q_i(x) = q(x) = \omega_{eq}^{\otimes d}(x) = \prod_{j=1}^d \frac{1}{\pi\sqrt{1-x_j^2}}\IND_{[-1,1]}(x_j), \quad
1\leq i\leq N.$$
}
To bound \eqref{e:toBound}, it is enough to bound
$K_N(x,x)\omega(x)/\omega_{eq}^{\otimes d}(x)$ since $K_N$ is a positive definite
kernel. Obtaining tight bounds is
problem-dependent. Interestingly, for Jacobi measures, these bounds have been an active
topic of research and we can use e.g. the bounds in \citep{Gau09} for our
rejection sampling. This means that in practice, we can apply our method in the
classical cases where Gaussian quadrature is applied.

We now discuss the cost of sampling a multivariate OP Ensemble. Without taking into
account the evaluation of orthogonal polynomials nor rejection sampling\footnote{The cost of the rejection steps, in particular, depends on the tightness of the bound of \eqref{e:toBound} and would need further study.}, the
number of basic operations is as much as for Gram-Schmidt orthogonalization of
$N$ vectors of dimension $N$, that is of order $N^3$ \cite[Section 5.2]{GoVa12}.
This means that Monte Carlo with DPPs is to be used when the gain in accuracy in
Theorem~\ref{DPPMC Th2} is worth spending a cubic computational
 budget to obtain the quadrature nodes. Such settings arise in Bayesian
 inference with expensive models in the natural sciences, where evaluating the
 integrand once can easily be a question of hours or more. Those are the same
 application areas as discussed for Bayesian quadrature in Section~\ref{s:bq}.

\rev{Additionally, we note that our central limit Theorems~\ref{DPPMC Th1}, \ref{DPPMC Th2} and \ref{DPPMC Th3} are independent of the algorithm we use to sample the multivariate OP Ensemble. Should a faster algorithm come out, this would further augment the applicability of Monte Carlo with DPPs.
 Fast sampling algorithms are out of the scope of this paper, but there are
 reasons to think they do exist. First, when $d=1$ and the reference measure is Jacobi \eqref{Jacobi measure}, sampling the OP Ensemble can already be done rejection-free and in time $\mathcal{O}(N^2)$ by diagonalizing a tridiagonal random matrix that only requires sampling independent beta variables \citep{KiNe04}.} Second, some discrete examples
 of DPPs can also be sampled in time $\cO(N\log N)$ \cite[Chapter 4]{LyPe16}.
 Third, since Monte Carlo with DPPs closely connects with methods such as
 QMC (Section~\ref{s:qmc}) and Bayesian quadrature (Section~\ref{s:bq}),
 inspiration could be drawn from fast methods that exist for these families of
 algorithms \citep{DiKuSl13,BaLaOb12,BOGOS15}.

\section{Experimental illustration}
\label{s:experiments}
\rev{In this section, we illustrate Theorems~\ref{DPPMC Th1} and \ref{DPPMC Th2} with three toy experiments}. In particular, for both CLTs, we investigate how
fast the Gaussian limit appears in each theorem and we estimate the rate of decay of the variance.

\subsection{The common setting}
\label{s:commonSetting}
We consider OP Ensembles with reference measure the product Jacobi measure
\eqref{Jacobi measure} with $\alpha_1=\beta_1=-1/2$, and $\alpha_j,\beta_j$
drawn i.i.d. uniformly on $[-1/2,1/2]$ for $1 < j\leq d$. As proposed in
Section~\ref{s:sampling}, we use \rev{$\omega_{eq}^{\otimes d}$} for the density of the proposal in the rejection sampling steps, and the bounds in \citep{Gau09}. For various $N\in [10,400]$ and each dimension $d\in \{1,2,3\}$, we sample $N_{\text{repeat}} = 100$ independent realizations of $\{\bx_1,\dots,\bx_N\}$. We refer to Section~\ref{s:sampling} for details on the sampling algorithm and references to implementations.

Figure~\ref{f:sample} depicts one of the obtained samples when $d=2$ and $N=150$. Each disk is centered at a node $\bv x_i$ in the sample, and the area of the disk is proportional to the weight $1/K_N(\bv x_i,\bv x_i)$. The marginal plots on each axis
depict the marginal histograms of the  weighted sample, with a green curve
indicating the density of the marginal Jacobi measures corresponding to $j=1,2$
in \eqref{Jacobi measure}. Good agreement is observed for the marginals, as
expected from the unbiasedness in \eqref{e:weightedUnbiasedness}.

We now proceed to investigating the Gaussianity and estimating the variance decay of the linear statistics in Theorems~\ref{DPPMC Th1} and \ref{DPPMC Th2}, with three integration tasks. All our confidence intervals include a Bonferroni correction to take into account the fact that these three experiments share the same OPE samples.

\subsection{Crude Monte Carlo: illustrating Theorem~\ref{DPPMC Th1}}
\label{s:crudeMCExperiment}

We define a simple ``bump'' test function that is $\mathscr C^\infty$ on
$I^d=[-1,1]^d$ and vanishes outside $I_\epsilon^d=[-1+\epsilon,1-\epsilon]^d$,
\eq
f(x) = \IND_{I_\epsilon^d}(x) \prod_{j=1}^d  \exp\left( - \frac{1}{1-\epsilon-x_j^2}  \right),
\label{e:testFunction}
\qe
so that \rev{$f\in\mathscr C$} and thus satisfies the assumptions of Theorem~\ref{DPPMC Th1}. We set $\epsilon=0.05$ and plot $f$ for $d=2$ in Figure~\ref{f:testFunction}.

We summarize our results for each dimension $d$ in Figure~\ref{f:results}. On each quadrant and for each $N$, we plot in black circles the sample variance of $$ \sum_{i=1}^N \frac{f(\bx_i)}{K_N(\bx_i,\bx_i)},$$ computed over the $N_{\text{repeat}}$ realizations. Blue and red dots indicate standard confidence intervals, for indication only. \rev{For comparison, we also plot in white circles the sample variance of a standard crude Monte Carlo estimator of the same integral, using i.i.d. samples of the reference measure.}

\begin{figure}
\subfigure["Bump" test function]{
  \includegraphics[width=\twofig]{\figdir/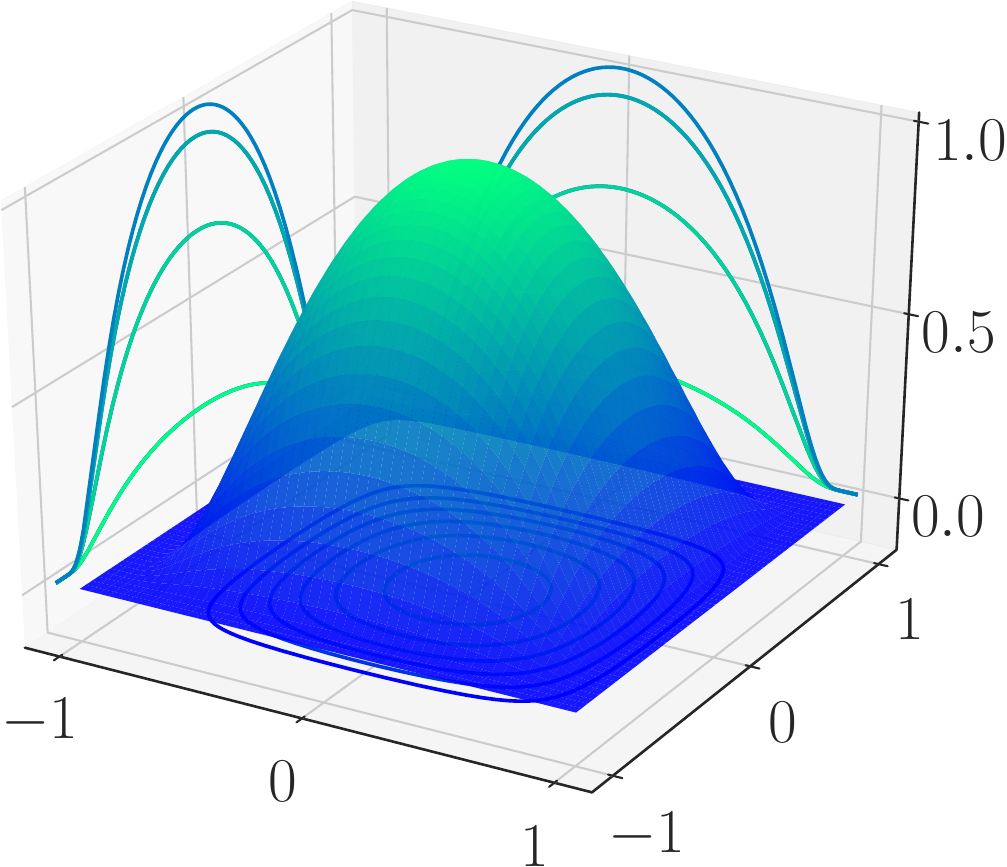}
  \label{f:testFunction}
}
\subfigure[$\Vert\cdot\Vert_1$ test function]{
  \includegraphics[width=\twofig]{\figdir/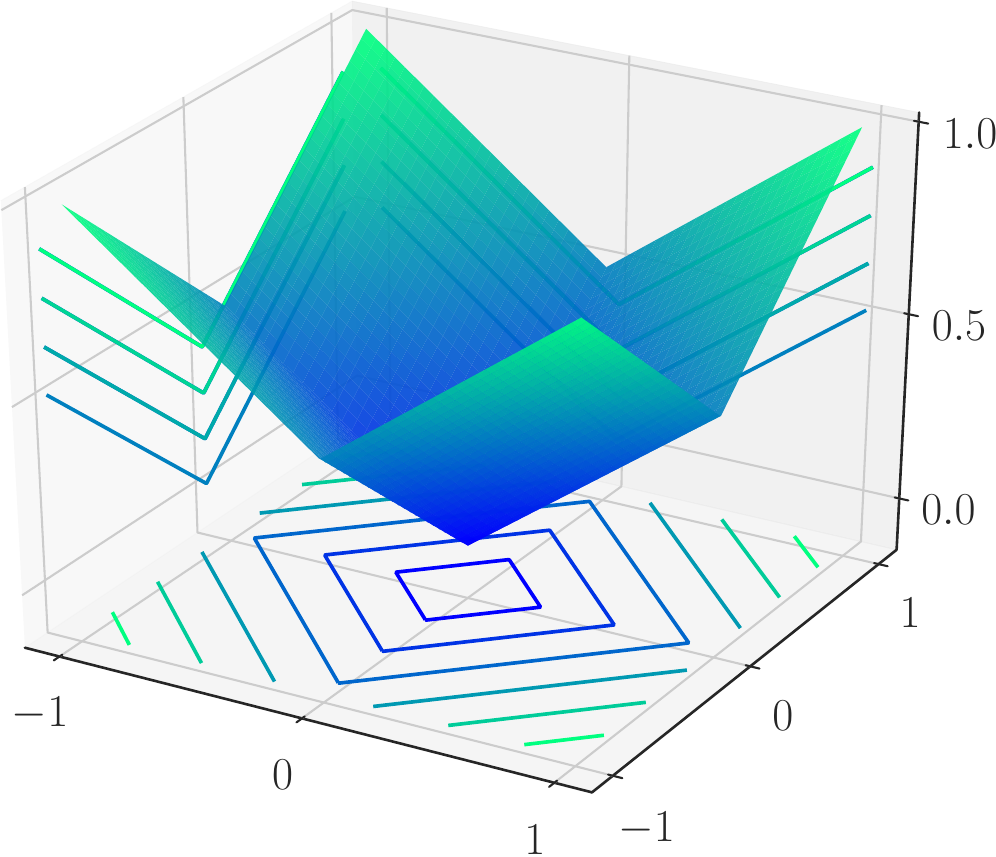}
  \label{f:testFunctionabs}
}
\caption{The two test functions}
\label{f:testFunctions}
\end{figure}

\begin{figure}
\subfigure[2D sample]{
\includegraphics[width=\twofig]{\figdir/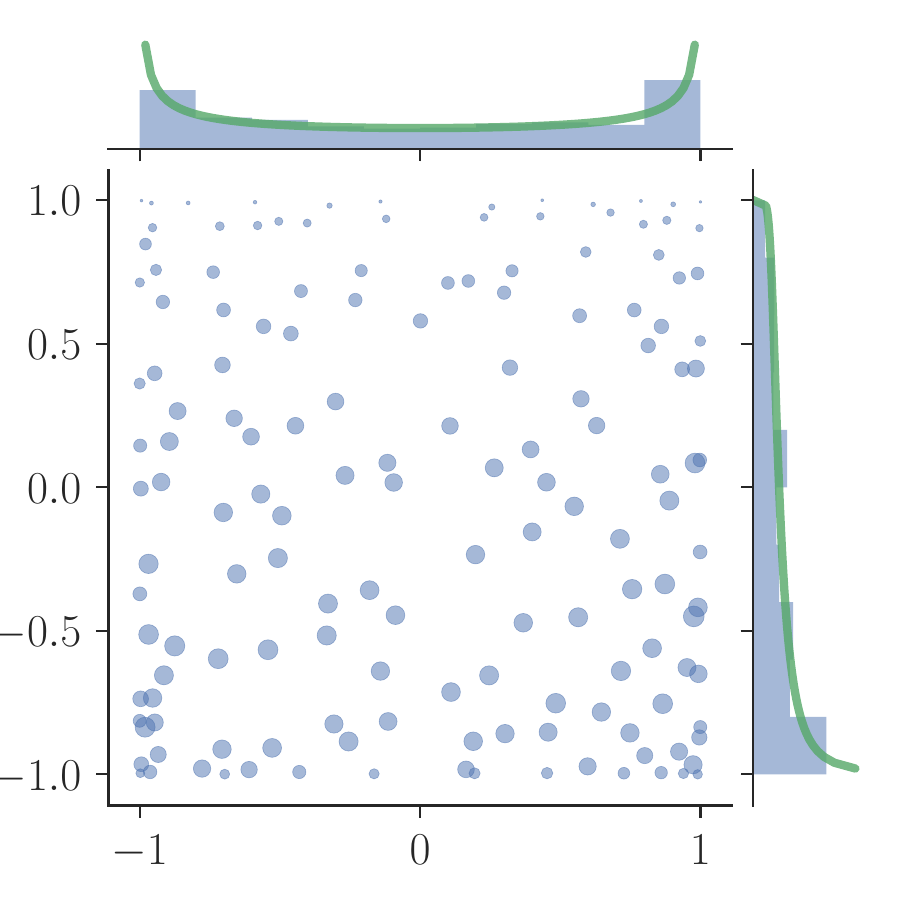}
\label{f:sample}
}
\subfigure[importance sampling setting]{
 \includegraphics[width=\twofig]{\figdir/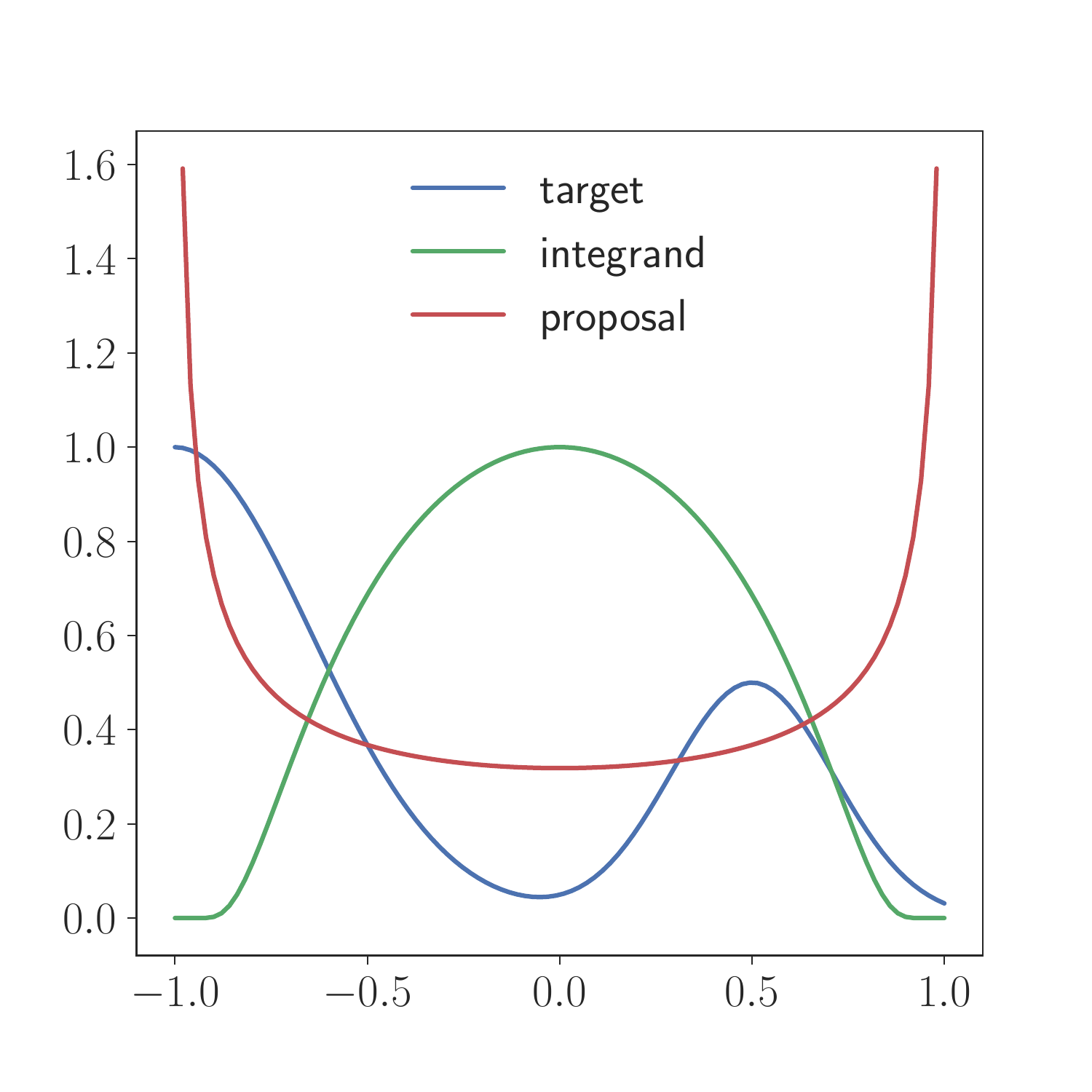}
 \label{f:importanceSamplingSetting}
 }
\caption{\ref{f:sample} A weighted sample of the Jacobi OPE described in Section~\ref{s:commonSetting}. \rev{\ref{f:importanceSamplingSetting} The test function \eqref{e:testFunction} when
  $d=1$, along with the proposal and the target used for demonstrating the importance sampling procedure in Theorem~\ref{DPPMC Th2}.}}
\end{figure}

\begin{figure}
\subfigure[$d=1$]{
\includegraphics[width=\twofig]{\figdir/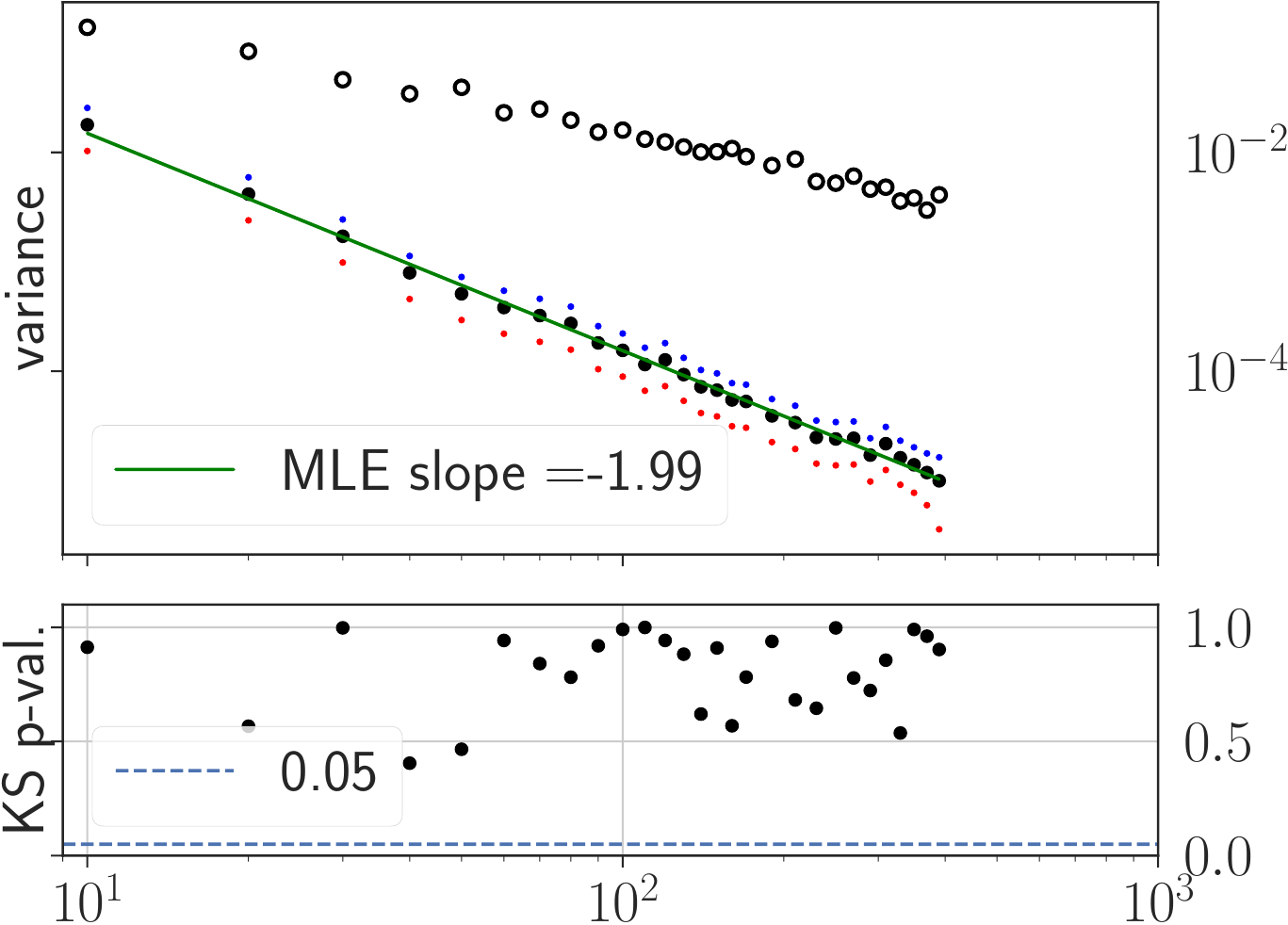}
\label{f:results1D}
}
\subfigure[$d=2$]{
\includegraphics[width=\twofig]{\figdir/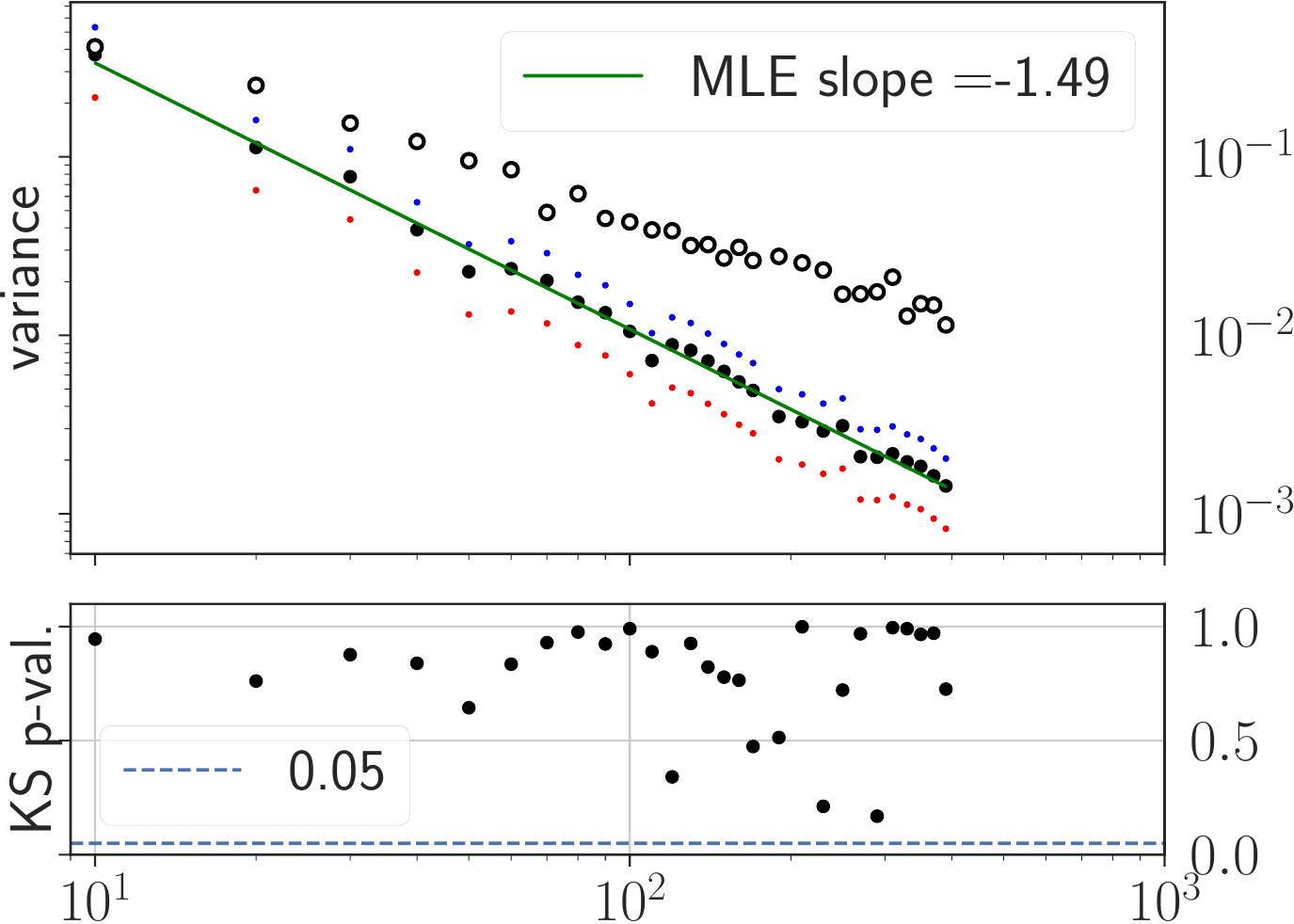}
\label{f:results2D}
}\\
\centering
\subfigure[$d=3$]{
\includegraphics[width=\twofig]{\figdir/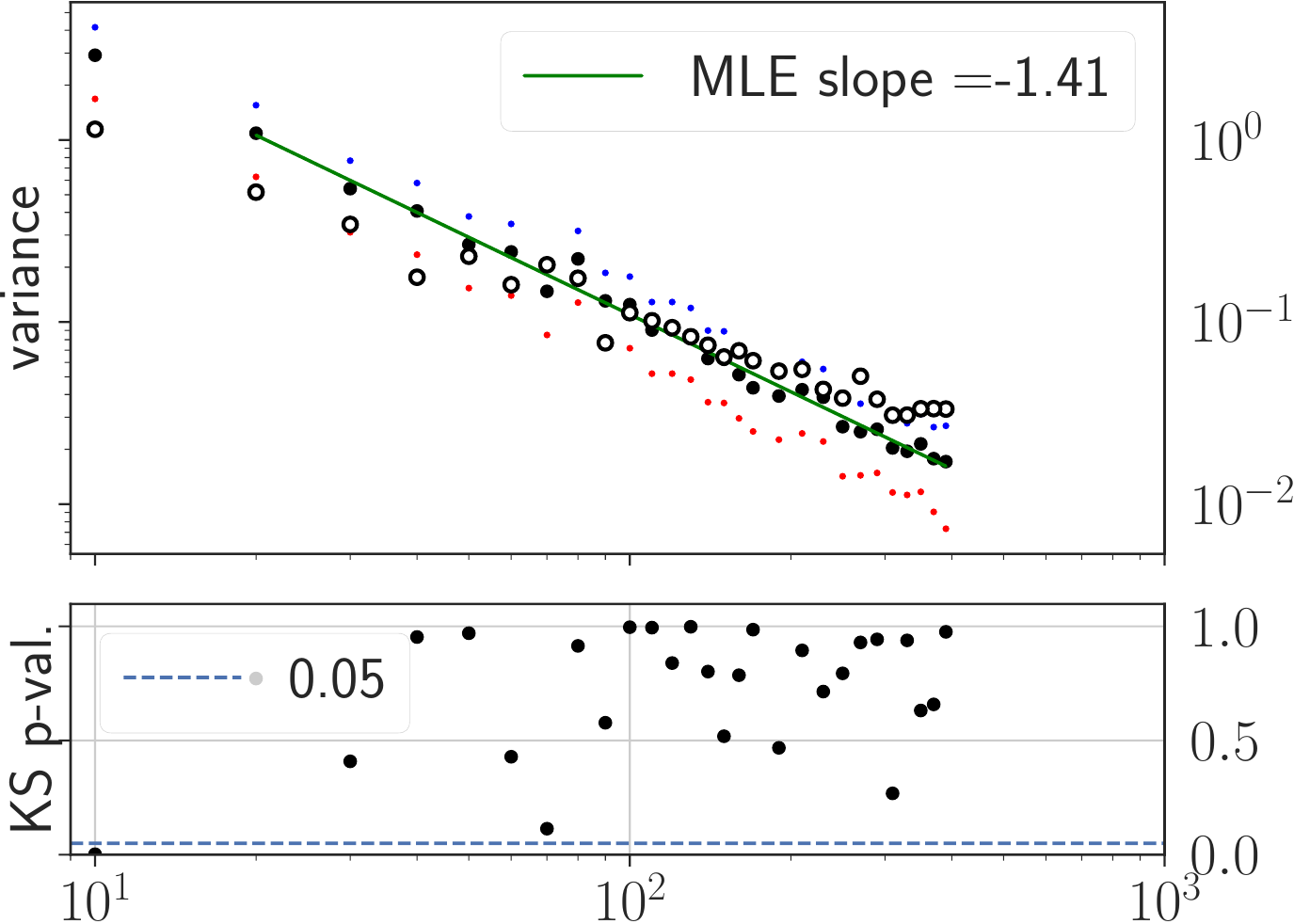}
\label{f:results3D}
}
\caption{Summary of the crude Monte Carlo results.}
\label{f:results}
\end{figure}

For a given dimension $d$, we want to infer the rate of decay of the
variance, in order to confirm the rate in the CLT of Theorem~\ref{DPPMC Th1}.
We proceed as follows. We first select the values of $N$ for which the
$N_{\text{repeat}}$ realizations give a $p$-value larger than $0.05$ in a
Kolmogorov-Smirnov test of Gaussianity. This is meant to eliminate the small values of $N$ for which the Gaussian in the CLT \eqref{e:weightedCLT} is a bad approximation for
our samples. We do not claim to perform any multiple testing, but rather use the
$p$-value as a loose indicator of Gaussianity. The bottom plot of each quadrant of
Figure~\ref{f:results} shows the $p$-values as a function of $N$. Note how
Gaussianity is hinted even for small $N$ in $d=1,2$, while for $d=3$, it takes larger $N$ to kick in. \rev{Then, we perform a standard frequentist linear regression} on the selected log variances vs. $\log(N)$. For visualization, we plot on each quadrant of Figure~\ref{f:results} the maximum likelihood (MLE) line in green and indicate its slope in the legend. The usual Student-t confidence intervals for the slope are given in the column of Table~\ref{t:confidenceIntervals} labeled \emph{Crude MC}.

\begin{table}[!h]
\centering
\begin{tabular}{|c|c|c|c|c|}
\hline
$d$ & $-1-1/d$ & Crude MC & Importance sampling & Assumption violation \\
\hline
\hline
1 & $-2$ & $[-2.05, -1.92]$ & $[-2.13, -1.99]$ & $[-2.08, -1.92]$\\
2 & $-1.5$ & $[-1.55, -1.43]$ & $[-1.61, -1.44]$ & $[-1.60, -1.42]$\\
3 & $-1.33$ & $[-1.49, -1.33]$ & $[-1.42, -1.16]$ & $\bf{[-1.24, -1.07]}$\\
\hline
\end{tabular}
\caption{\rev{Confidence intervals for the variance decay, for all three experiments of Section~\ref{s:experiments}. In bold, we highlight the only confidence interval that does not contain the corresponding theoretical rate given in the second column.}}
\label{t:confidenceIntervals}
\end{table}


The confidence intervals are in very good agreement with Theorem~\ref{DPPMC Th1} for each dimension $d$. The combined plots in Figure~\ref{f:results} hint that the CLT
approximation is strikingly accurate for all $d$, even for small $N$. For $d=3$, the Gaussianity appears slightly later in terms of $N$, which confirms the intuition that the convergence to a Gaussian is slower when the dimension increases. \rev{Relatedly, the intercept of the various straight lines increases with $d$, and this increase seems to be faster for DPPs than crude Monte Carlo. This entails that the value of $N$ above which OPEs becomes significantly more efficient than crude Monte Carlo increases with $d$, as can be seen in Figure~\ref{f:results}.}

\subsection{Importance sampling: illustrating Theorem~\ref{DPPMC Th2}}
\rev{
We now illustrate the importance sampling result in Theorem~\ref{DPPMC Th2}. As proposal reference measure, we use the OPEs described in Section~\ref{s:commonSetting}. More precisely, we take $q(x)\d x$ in Theorem~\ref{DPPMC Th2} to be the product Jacobi measure described in Section~\ref{s:commonSetting}. The goal is still to estimate the integral of $\phi$ in \eqref{e:testFunction}, but this time with respect to a target distribution $\mu$ that is a truncated mixture of two Gaussians, with density
}
\rev{
$$ \omega(x) = \frac12 \exp\left[-\frac{\Vert x-0.5\Vert^2}{(0.3)^2}\right] + \exp\left[-\frac{\Vert x+1\Vert^2}{(0.5)^2}\right]$$
}
\rev{with respect to the Lebesgue measure on $I^d$.
The target measure $\mu$, the proposal reference measure and the test function are illustrated in $d=1$ in Figure~\ref{f:importanceSamplingSetting}.
}

\rev{
For the sake of shortness, we defer the figure displaying the results of the regression to Figure~\ref{f:resultsIS} in Appendix~\ref{a:figures}. We only report here the resulting confidence intervals on the rate of decay of the variance, in the column of Table~\ref{t:confidenceIntervals} labeled \emph{importance sampling}. Again, the confidence intervals are in very good agreement with the CLT in Theorem~\ref{DPPMC Th2}.
}

\subsection{An integrand that violates the assumptions of Theorem~\ref{DPPMC Th1}}
\label{s:ass}
\rev{We again copy the setting of Section~\ref{s:commonSetting}, but we change the integrand to
$$ f(x) = d^{-1}\Vert x\Vert_1.$$
This test function is plotted in Figure~\ref{f:testFunctionabs} for $d=2$. It does not belong to the class $\mathscr{C}$ that we authorize in Theorem~\ref{DPPMC Th1}: it is not $\mathscr C^1$, and it does not vanish at the border of $I^d$.
}
\rev{
Again, we defer the display of the regression to Figure~\ref{f:resultsAss} in Appendix~\ref{a:figures}, and we limit ourselves here to the confidence intervals on the slope, which are given in the last column of Table~\ref{t:confidenceIntervals}. This time, while the decay of the mean square error is still significantly better than crude Monte Carlo, the confidence intervals do not all match the conclusion of Theorem~\ref{DPPMC Th1}: we have written in bold the confidence interval for $d=3$, which suggests a rate of decay that is slower than $-1-1/d$. This is a hint that the assumptions of Theorem~\ref{DPPMC Th1} are essentially tight.
}

\section{Discussion and perspectives}
\label{s:discussion}
As detailed in Remark~\ref{r:gauss}, Monte Carlo with DPPs is a stochastic
counterpart to Gaussian quadrature, introduced in Section~\ref{s:quadrature}.
Compared to the Monte Carlo methods introduced in Section~\ref{s:monteCarlo}, and
\ref{s:qmc}, Theorem \ref{DPPMC Th2} is an importance sampling procedure, with negatively
correlated importance samples. This negative correlation results in a variance
reduction that impacts the decay rate of the variance. Loosely speaking, this is
reminiscent of the surprising kernel density approach to importance sampling of
\cite{DePo16} described in Section~\ref{s:monteCarlo}. Our rates are better for
equivalent smoothness in $d=1$, but for $d>1$, the theoretical comparison is less clear. In terms of sampling cost, \citep{DePo16} scales as $\cO(N^2)$ not taking into account the tuning of the kernel parameters. Naively sampling orthogonal polynomial ensembles is $\cO(N^3)$, without taking rejection sampling into account. Tackling the cubic cost of sampling DPPs is a natural sequel to our work, see also Section~\ref{s:sampling}.


Monte Carlo with DPPs is also reminiscent of randomized quasi-Monte
Carlo methods such as scrambled nets \citep{Owe97}, discussed in Section~\ref{s:qmc}. The important difference is that randomness and discrepancy are tied in our DPP proposal. The similarities with QMC are an interesting lead for future research. In particular, fast constructions of nets in QMC \citep{DiKuSl13} could yield fast sampling algorithms for DPPs.

Monte Carlo with DPPs also connects with Bayesian quadrature, introduced in
Section~\ref{s:bq}. As pointed out in Section~\ref{s:sampling}, sampling
projection DPPs is related to sequentially maximizing the variance of a Gaussian process,
while Bayesian quadrature is about sequentially minimizing the variance of the
integral of $f$ when a Gaussian process prior is assumed on $f$, see
Section~\ref{s:bq}. A formal connection with Bayesian quadrature would
facilitate the transfer of CLTs such as our Theorem~\ref{DPPMC Th2} to
Bayesian quadrature. Conversely, the efficient Frank-Wolfe optimization procedures
given for herding by \cite{BaLaOb12,BOGOS15} could influence fast sampling algorithms for DPPs.

Additionally, theoretical rates have been provided for hybrid integral
estimators in \cite[Theorem 1]{BOGO15}, which mix classical Monte Carlo nodes
with Bayesian quadrature weights, effectively reducing the influence of close-by
pairs of Monte Carlo samples. Together with \citep{DePo16}, the latter approach
uses a kernel to introduce anti-correlation in a \emph{postprocessing} reweighting
step. In comparison, Monte Carlo with DPPs also uses a kernel to encode repulsiveness, but the repulsiveness not only appears in the weights: it is embedded in the sampling procedure.

\rev{The same reason that sampling and approximation are both encapsulated in the same DPP object is the major difference with the postprocessing approaches of \citep{OaGiCh17,LiLe17}. With DPPs, we lose the algorithmic simplicity and often low cost of postprocessing, but we gain a CLT with weaker, \emph{dimension-independent} smoothness requirements. Also, our importance reweighted estimator in Theorem~\ref{DPPMC Th2} bypasses the need to know the moments of the target measure, which is akin to removing the constraint of knowing the integral of the control approximation to the integrand in \citep{OaGiCh17}. There are interesting avenues to try to obtain a hybrid algorithm that would get the best of both worlds.}

\rev{Finally, we comment on our focus on orthogonal polynomials and projection kernels.} Relying on orthogonal polynomials made available technology that we extensively used in the proofs, such as precise asymptotic results and the use of recurrence coefficients. First, with slightly stronger estimates, one may allow the reference measure to depend on $N$, so as to replace the equilibrium measure $\mu_{eq}^{\otimes d}$ by a measure putting less mass on the boundary of the integration domain. This would prevent part of the quadrature nodes to clutter close to the boundary of the hypercube. Second, any projection kernel onto an $N$-dimensional subspace of $L^2(\mu)$ yields an appropriate DPP for numerical integration, and orthogonal polynomials may not be the most natural choice of basis for a given integrand. It is easy to imagine kernels built on wavelets or other bases of $L^2(\mu)$, and a clever choice of basis may also yield faster sampling algorithms, but the difficulty lies in obtaining such variance estimates as we obtained for orthogonal polynomials.
\rev{Third, one should keep in mind that not every projection kernel leads to small variance. For instance, the projection kernel onto $$\mathrm{Span}(e^{2ik\pi\theta}:\; k=0,\ldots,N-1)$$ yields a DPP on $[-1,1]$ such that $\Var[\sum f(\bv x_i)] \sim N$  for any odd function of $L^2(-1,1)$.
On the positive side, extra complex structure can buy us further variance reduction in even dimensions for analytic test functions. For instance, one can show along the lines of Section~\ref{section Cov Cheby}, that the projection onto $\mathrm{Span}(z^k:\; k=0,\ldots,N-1)$ yields a DPP on the unit disc with $\Var[\sum f(\bv x_i)] \sim 1$ instead of $N^{1-1/2}$ provided the function $f$ is analytic. Fourth, one may be tempted to use contraction kernels instead of projection kernels, but besides the problem that the number of points drawn from the DPP becomes random, contraction DPPs are bound to augment the variance of linear statistics.}

To conclude, DPPs are a new way to connect numerical integration with rich analytic tools.




\paragraph{Acknowledgments:} AH would like to thank Bernhard
Beckermann, Jeff Geronimo and  Kurt Johansson for useful and stimulating discussions. This work
started while RB and AH were respectively at University of Oxford and KTH Royal Institute of Technology,
funded by EPSRC grant number EP/I017909/1 and grant KAW 2010.0063
from the Knut and Alice Wallenberg Foundation. RB also acknowledges support from
ANR grant \textsc{Bnpsi} ANR-13-BS03-0006, and AH acknowledges support from Labex CEMPI ANR-11-LABX-0007-01. Finally, both authors acknowledge support
from CNRS through PEPS JCJC \textsc{DppMc} and from ANR through grant
\textsc{BoB} ANR-16-CE23-0003.

\appendix

\newpage

\begin{center}
\Huge{Appendix}
\end{center}

\section{Preliminary material}
\label{s:outlineOfProofs}

In this section, we provide some general background on orthogonal polynomials, \rev{we prove short results and we outline the proofs of the main theorems.}

\subsection{Orthogonal polynomials and the Nevai class}
\label{background sec}
In the following, we use the equilibrium measure $\mu_{eq}$ of $I$, defined by
\eq
\label{mueq}
\mu_{eq}(\d x)=\omega_{eq}(x)\d x,\qquad \omega_{eq}(x)=\frac{1}{\pi\sqrt{1-x^2}}\ \bs 1_{I}(x).
\qe
The name comes from its characterization as the unique minimizer of the logarithmic energy $\iint
\log|x-y|^{-1}\mu(\d x)\mu(\d y)$ over Borel probability measures $\mu$ on $I$ \citep{SaTo97}. It is also the image of the uniform measure on the unit circle through the map $e^{i\theta}\mapsto x=\cos\theta$. The associated orthonormal polynomials are the normalized Chebyshev polynomials of the first kind, defined on $I$ by
\eq
\label{Cheby cos}
T_k(\cos\theta)=
\begin{cases}
\sqrt 2 \cos(k\theta) & \mbox{if }k\geq 1\\
1 & \mbox{if }k=0
\end{cases}\;,\qquad \theta \in [0,\pi].
\qe
They satisfy the three-term recurrence relation
\eq
\label{Cheby 3-rec}
x T_k(x)= a_k^* T_{k+1}(x) + b_k^* T_k(x) + a_{k-1}^*T_{k-1}(x),\qquad k\in\N,
\qe
where
\eq
\label{rec coef Cheby}
a_k^*=
\begin{cases}
0 & \mbox{if } k=-1\\
1/\sqrt 2& \mbox{if } k=0\\
1/2 & \mbox{if } k\geq 1
\end{cases}
\qquad \mbox{and }\qquad b_k^*=0.
\qe

More generally, given a reference measure $\mu$ on $I$ with orthonormal polynomials $(\phi_k)$, we always have the three-term recurrence relation
\eq
\label{3-term rec}
x\phi_k(x)=a_{k}\phi_{k+1}(x)+b_k\phi_k(x)+a_{k-1}\phi_{k-1}(x),\qquad k\in\N,
\qe
where $a_{-1}=0$ and  $a_k>0$ and $b_k\in\R$ for every $k\geq 0$. The existence of the recurrence coefficients $(a_k)_{k\in\N}$ and $(b_k)_{k\in \N}$ follows by decomposing the polynomial $x\phi_k$ into the orthonormal family $(\phi_\ell)_{\ell=0}^{k+1}$ of $L^2(\mu)$ and  observing that $\langle x\phi_k,\phi_\ell\rangle=\langle \phi_k,x\phi_\ell\rangle=0$ as soon as $\ell<k-1$ by  orthogonality.

\begin{definition}
\label{def Nevai class}
A measure $\mu$ supported on $I$ is \emph{Nevai-class} if the  recurrence coefficients for the associated orthonormal polynomials satisfy
\[
\lim_{k\rightarrow\infty}a_k=1/2,\qquad \lim_{k\rightarrow\infty}b_k=0.
\]
\end{definition}

Notice the respective limits of the $a_k$'s and $b_k$'s for Nevai class measures
are the recurrence coefficients \eqref{rec coef Cheby} of the measure $\mu_{eq}$
when $k\geq 1$.

The next theorem gives a sufficient condition for a measure to be Nevai class  \cite[Theorem 1.4.2]{Sim11}.

\begin{theorem}{\bf (Denisov-Rakhmanov)}
\label{DR theorem}
\label{Rakhmanov-Denisov}
Let $\mu$ be a reference measure on $I$ with Lebesgue decomposition $\mu(\d
x)=\omega(x)\d x + \mu_{s}$. If $\omega(x)>0$ almost everywhere, then  $\mu$ is Nevai-class.
\end{theorem}

Consider now the Christoffel-Darboux kernel
\eq
\label{CD kernel 2}
K_N(x,y)=\sum_{k=0}^{N-1}\phi_k(x)\phi_k(y),
\qe
and notice $\frac1N K_N(x,x)\mu(\d x)$ is a probability measure. One of the
interesting properties of Nevai-class measures is that this probability measure
has $\mu_{eq}$ for weak limit as $N\to\infty$ \citep{StTo92}.

\begin{theorem}
\label{weak conv eq 1D}
Assume $\mu$ supported on $I$ is Nevai-class. Then, for every $f\in\mathscr C^0(I,\R)$,
$$
\int f(x)  \frac1NK_N(x,x)\mu(\d x) \xrightarrow[N\to\infty]{} \int f(x)\mu_{eq}(\d x).
$$
\end{theorem}

Now, consider instead a reference measure $\mu$ on $I^d$ with associated
multivariate orthogonal polynomials $(\phi_k)_{k\in\mathbb{N}}$ (see Section~\ref{sec intro
  MOPE}) and Christoffel-Darboux kernel $K_N(x,y)$ defined as in \eqref{CD kernel 2}. Assume further that $\mu=\mu_1\otimes\cdots\otimes\mu_d$ is a product of $d$ measures $\mu_j$ on $I$, and denote by $\phi_k^{(j)}$ and $K_N^{(j)}(x,y)$ the respective orthogonal polynomials and  Christoffel-Darboux kernel associated with $\mu_j$. Then, we have
\eq
\phi_k(x)=\phi_{k_1}^{(1)}(x_1)\cdots\phi_{k_d}^{(d)}(x_d)
\label{e:mope}
\qe
where $(k_1,\ldots,k_d)=\frak b(k)$. Moreover,
\eq
\label{CD kernel split}
K_{M^d}(x,y)=\prod_{j=1}^d K_M^{(j)}(x_j,y_j).
\qe
As a consequence, Theorem \ref{weak conv eq 1D} easily yields the following.

\begin{corollary}
\label{weak conv eq}
Let $\mu=\mu_1\otimes\cdots\otimes\mu_d$ with $\mu_j$ supported on $I$ and Nevai-class. Then, for every $f\in\mathscr C^0(I^d,\R)$,
\eq
\label{weak conv eq explicit}
\int f(x)  \frac1NK_N(x,x)\mu(\d x) \xrightarrow[N\to\infty]{} \int f(x)\mu_{eq}^{\otimes d}(\d x).
\qe
\end{corollary}

\begin{proof} By the Stone-Weierstrass theorem, it is enough to show \eqref{weak
    conv eq explicit} when $f(x)= \prod_{j=1}^d f_j(x_j)$ with $f_j\in \mathscr C^0(I,\R)$. Without loss of generality, one can further assume the functions $f_j$ are non-negative. Let $M=\lfloor N^{1/d}\rfloor$ be the unique integer satisfying $M^d\leq N < (M+1)^d$. Since we have $K_{M^d}(x,x)\leq K_N(x,x)\leq K_{(M+1)^d}(x,x)$ and, by \eqref{CD kernel split},
\begin{multline}
\frac{M^d}{N}\prod_{j=1}^d\int f_j(x) \frac1{M} K_{M}^{(j)}(x,x)\mu_j(\d x)\\
\leq \int f(x) \frac1N K_N(x,x)\mu(\d x)\leq \frac{(M+1)^d}{N}\prod_{j=1}^d\int f_j(x) \frac1{M+1} K_{M+1}^{(j)}(x,x)\mu_j(\d x),
\end{multline}
Corollary \ref{weak conv eq} follows from Theorem \ref{weak conv eq 1D}.
\end{proof}

The next lemma is yet another aspect of Nevai-class measures that is relevant to our
proofs, and may be of independent interest.
\begin{lemma}
\label{convergence CD out diag}
Assume $\mu$ supported on $I$ is Nevai-class. We have the weak convergence of
\eq
\label{LN d=1}
Q_N(\d x,\d y)=(x-y)^2K_N(x,y)^2\mu(\d x)\mu(\d y)
\qe
towards
\eq
\label{L d=1}
L(\d x,\d y)=\frac 1 2 (1-xy)\mu_{eq}(\d x)\mu_{eq}(\d y).
\qe
\end{lemma}

\begin{proof} First, the Christoffel-Darboux formula reads
\eq
\label{LN CD}
(x-y)^2K_N(x,y)^2= a_N^2\big(\phi_N(x)\phi_{N-1}(y)-\phi_{N-1}(x)\phi_N(y)\big)^2,
\qe
\rev{which follows by computing $xK_N(x,y)$ using the recurrence relation \eqref{3-term rec} and witnessing several cancellations when subtracting $yK_N(x,y)$}.
Thus, by the orthonormality conditions, we see $\iint Q_N(\d x,\d y)= 2a_N^2$. Since $\mu$ is Nevai-class, the former converges to $1/2=\iint L(\d x,\d y)$. This allows us to use the usual weak topology (i.e. the topology coming by duality with respect to the continuous functions) for bounded Borel measures. \\

\noindent \textbf{Step 1.} We first prove the lemma when $\mu=\mu_{eq}$, so that the $\phi_k$'s are the Chebyshev polynomials $T_k$, see \eqref{Cheby cos}. By \eqref{LN CD}, the push-forward of \eqref{LN d=1} by the map $(x,y)\mapsto(\cos\theta,\cos\eta)$, where $\theta,\eta\in[0,\pi]$, reads
\eq
\label{transformed LN}
\frac1{\pi^2} \big\{\cos(N\theta)\cos((N-1)\eta)- \cos((N-1)\theta\cos(N\eta)\big\}^2\d\theta\d\eta.
\qe
This measure  has for Fourier transform
\begin{multline*}
\frac1{\pi^2} \int_0^\pi\int_0^\pi e^{i(\theta u+ \eta v)}\big\{\cos( N\theta)\cos((N-1)\eta)- \cos((N-1)\theta)\cos(N\eta)\big\}^2\d\theta\d\eta\\
= \frac1{\pi^2} \int_0^\pi\int_0^\pi\cos(\theta u+ \eta v)\big\{\cos( N\theta)\cos((N-1)\eta)- \cos((N-1)\theta)\cos(N\eta)\big\}^2\d\theta\d\eta.
\end{multline*}
By developing the square in the integrand and linearizing the products of
cosines, we see that the non-vanishing contribution  as $N\to\infty$ of the
Fourier transform are the terms which are independent on $N$. Indeed, the $N$-dependent terms come up with a factor $1/N$ after integration. Thus, the Fourier transform equals, up to $\cO(1/N)$, to
$$
\frac1{2\pi^2} \int_0^\pi\int_0^\pi\cos(\theta u+ \eta v)\big(1-\cos\theta\cos\eta\big)\d\theta\d\eta.
$$
This yields the weak convergence of \eqref{transformed LN} towards $(2\pi^2)^{-1}(1-\cos\theta\cos\eta)\d\theta\d\eta$, and the lemma follows, in the  case where  $\mu=\mu_{eq}$, by taking the image of the measures by the inverse map $(\cos\theta,\cos\eta)\mapsto(x,y)$.\\

\noindent \textbf{Step 2.}  We now prove the lemma for a general Nevai-class measure $\mu$ on $I$. Let us denote by $Q_N^\mu$ the measure \eqref{LN d=1} in order  to stress the dependence on $\mu$.  Thanks to Step 1, it is enough to prove that for every $m,n\in\N$, we have
\[
\lim_{N\to\infty}\left|\iint x^m y^n Q_N^\mu(\d x,\d y)-\iint x^m y^n Q_N^{\mu_{eq}}(\d x,\d y)\right|=0,
\]
in order to complete the proof of the lemma.  Recalling \eqref{LN d=1}, \eqref{LN CD}, and that  $a_N\to 1/2$, it is enough to show that for every $m\in\N$,
\eq
\label{moment 1}
\lim_{N\to\infty}\left|\int x^m \phi^2_N(x)\mu(\d x) -\int x^m T^2_N(x)\mu_{eq}(\d x)\right| =0
\qe
and
\eq
\label{moment 2}
\lim_{N\to\infty}\left|\int x^m \phi_N(x)\phi_{N-1}(x)\mu(\d x) -\int x^m T_N(x)T_{N-1}(x)\mu_{eq}(\d x)\right| =0.
\qe
To do so, we first complete for convenience the sequences of recurrence coefficients $(a_n)_{n\in\N}$ and $(b_n)_{n\in\N}$ introduced in \eqref{3-term rec} as bi-infinite sequences $(a_n)_{n\in\Z}$, $(b_n)_{n\in\Z}$, where we set $a_n=b_n=0$ for every $n<0$. It follows inductively from the three-term recurrence relation \eqref{3-term rec} that for every $k,\ell,m\in\N$,
\eq
\label{path rep}
\int x^m \phi_k(x)\phi_{\ell}(x)\mu(\d x)=\sum_{\gamma:(0,k)\rightarrow (m,\ell)} \prod_{e\in\gamma}\omega(e)_{\{(a_n),\,(b_n)\}},
\qe
where the sum ranges over all the paths $\gamma$ lying on the oriented graph with vertices $\Z^2$ and edges $(i,j)\to (i+1,j+1)$, $(i,j)\to (i+1,j)$ and $(i,j)\to (i+1,j-1)$ for $(i,j)\in\Z^2$, starting from $(0,k)$ and ending at $(m,\ell)$. For every edge $e$ of $\Z^2$, we introduced the weight associated with the sequences $(a_n)=(a_n)_{n\in\Z}$, $(b_n)=(b_n)_{n\in\Z}$ defined by
\eq
\label{weights paths}
\omega(e)_{\{(a_n),\,(b_n)\}}=
\begin{cases}
a_j & \mbox{if } e=(i,j)\to (i+1,j+1)\\
b_j & \mbox{if } e=(i,j)\to (i+1,j)\\
a_{j-1} & \mbox{if } e=(i,j)\to (i+1,j-1),
\end{cases}
\qe
see also \citep{Har15}. Now, observe that the set of all paths $\gamma$ satisfying $\gamma:(0,k)\rightarrow (m,\ell)$ only depends on $k,\ell$ through  $|k-\ell|$ and is empty as soon as $|k-\ell|>m$. Thus it is a finite set, and moreover, by translation of the indices, for every $k,\ell,m\in\N$ we have
\eq
\label{nice path}
\int x^m \phi_k(x)\phi_{\ell}(x)\mu(\d x)=\bv 1_{|k-\ell|\leq m} \sum_{\gamma:(0,k-\ell)\rightarrow (m,0)} \prod_{e\in\gamma}\omega(e)_{\{(a_{n+\ell}),\,(b_{n+\ell})\}}.
\qe
In particular, see  \eqref{Cheby 3-rec}--\eqref{rec coef Cheby},
\eq
\label{nice path Cheby}
\int x^m T_k(x)T_{\ell}(x)\mu_{eq}(\d x)=\bv 1_{|k-\ell|\leq m} \sum_{\gamma:(0,k-\ell)\rightarrow (m,0)} \prod_{e\in\gamma}\omega(e)_{\{(a^*_{n+\ell}),\,(b^*_{n+\ell})\}}.
\qe
Finally,  by combining  \eqref{nice path} and \eqref{nice path Cheby}, we obtain
\begin{multline}
\label{key ineq 3-rec}
\left| \int x^m \phi_k(x)\phi_{\ell}(x)\mu(\d x) -  \int x^m T_k(x)T_{\ell}(x)\mu_{eq}(\d x)\right| \\
\leq \sum_{\gamma:(0,k-\ell)\rightarrow (m,0)} \Big|\prod_{e\in\gamma} \omega(e)_{\{(a_{n+\ell}),\,(b_{n+\ell})\}}- \prod_{e\in\gamma}\omega(e)_{\{(a^*_{n+\ell}),\,(b^*_{n+\ell})\}}\Big|.
\end{multline}
Together with the Nevai-class assumption for $\mu$, which states that $a_n- a_n^*\to 0$ and $b_n- b_n^*\to 0$ as $n\to\infty$, it follows that \eqref{moment 1} and \eqref{moment 2} hold true by taking $k=\ell=N$, or $k=N$ and $\ell=N-1$, in \eqref{key ineq 3-rec}. This completes the proof of Lemma \ref{convergence CD out diag}.

\end{proof}

\subsection{Sketch of the proof of Theorem \ref{th CLT general}}
\label{s:sketch proof CLT MOPE}

\subsubsection{Reduction to probability reference measures}
\label{s:reductionToProba}
First, in the statement of Theorem \ref{th CLT general}, we can assume the reference measure  $\mu$ is a probability measure without loss of generality. This will simplify notation in the proof of Theorem \ref{th CLT general}.

Indeed,  for any positive measure $\mu$ on $I^d$ with (multivariate) orthonormal
polynomials $\phi_k$ and any $\alpha>0$, the orthonormal polynomials associated
with $\alpha\mu$ are $\phi_k/\sqrt\alpha$. Thus, if we momentarily denote by $K_N(\mu;x,y)$ the $N$th Christoffel-Darboux kernel associated with a measure $\mu$, we have $K_N(\alpha\mu;x,y)=K_N(\mu;x,y)/\alpha$. As a consequence, for every $n\geq 1$, the correlation measures
\[
\det\Big[K_N(\mu; x_i,x_\ell)\Big]_{i,\ell=1}^n\prod_{i=1}^n\mu(\d x_i),
\]
remain unchanged if we replace $\mu$ by $\alpha\mu$ for any $\alpha>0$. Hence,
multivariate OP Ensembles are invariant under $\mu\mapsto\alpha\mu$.

\subsubsection{Soshnikov's key theorem}
As stated previously,  Theorem \ref{th CLT general} has already been proven when
$d=1$ by \cite{BrDu13}, as a consequence of a generalized strong Szeg\H{o}
theorem they obtained. The difficulty in proving Theorem \ref{th CLT general}
when $d\geq 2$ turns out to be of different nature than the one-dimensional
setting. Indeed,  the next result due to Soshnikov essentially states that the
cumulants of order three and more of the linear statistic $\sum f(\bv x_i)$
decay to zero as $N\to\infty$ as soon as its variance goes to infinity, \rev{and we will show the variance indeed diverges when $d\geq 2$}. Thus, a CLT follows easily as soon as one can obtain asymptotic estimates on the variance. However, if obtaining such variance estimates is relatively easy when $d=1$, the task becomes more involved in higher dimension.

More precisely, the general result \cite[Theorem 1]{Sos02} has the following consequence.

\begin{theorem}{\bf (Soshnikov)}
\label{Sosh} Let $\bv x_1,\ldots,\bv x_N$ form a multivariate OP Ensemble with respect to a given reference measure $\mu$ on $I^d$. Consider a sequence $(f_N)$ of uniformly bounded and measurable real-valued functions on $I^d$ satisfying, as $N\to\infty$,
\eq
\label{Sosh cond}
\Var\left[\, \sum_{i=1}^Nf_N(\bv x_i)\right]\longrightarrow\infty,
\qe
and,  for some $\delta>0$,
\eq
\label{Sosh cond 2}
\E\left[\, \sum_{i=1}^N\big|f_N(\bv x_i)\big|\right]=O\left(\Var\left[\, \sum_{i=1}^Nf_N(\bv x_i)\right]^\delta\right).
\qe
Then, we have
$$
\frac{\displaystyle\sum_{i=1}^Nf_N(\bv x_i)-\E\left[\, \sum_{i=1}^Nf_N(\bv x_i)\right]}{\displaystyle\sqrt{\Var\left[\, \sum_{i=1}^Nf_N(\bv x_i)\right]}}
\xrightarrow[N\to\infty]{law}\mathcal N(0,1).
$$
\end{theorem}

\subsubsection{Variance asymptotics}
In order to prove Theorem \ref{th CLT general} it is enough to show the following asymptotics.

\begin{proposition}
\label{variance general CLT easy}
Assume $\mu$ and $\bv x_1,\ldots,\bv x_{N}$ satisfy the hypothesis of  Theorem \ref{th CLT general}.  Then, for every $f\in\mathscr C^1(I^d,\R)$, we have
\eq
\label{variance asympt}
\lim_{N\to\infty}\frac{1}{N^{1-1/d}}\Var\left[\, \sum_{i=1}^{N}f(\bv x_i)\right]=\sigma_f^2.
\qe
\end{proposition}
Indeed, for any $d\geq 2$ and any $f\in \mathscr C^1(I^d,\R)$, Corollary
\ref{weak conv eq} and Proposition \ref{variance general CLT easy} imply
\eqref{Sosh cond} and \eqref{Sosh cond 2} with $f_N=f$ and $\delta=d/(d-1)$.
Thus, we can apply Theorem~\ref{Sosh} to obtain Theorem~\ref{th CLT general}.

Proposition \ref{variance general CLT easy} is the main technical result of this work. Consider the $d$-fold product of the equilibrium measure \eqref{mueq}, namely the probability measure on $I^d$ given by
\eq
\label{eq measure d}
\mu_{eq}^{\otimes d}(\d x)=\omega_{eq}^{\otimes d}(x)\d x,\qquad \omega_{eq}^{\otimes d}(x)= \prod_{j=1}^d\frac{1}{\pi\sqrt{1-x_j^2}}\ \bs 1_{I^d}(x).
\qe
In our proof of Proposition~\ref{variance general CLT easy}, we start by
investigating the limit \eqref{variance asympt} when
$\mu=\mu_{eq}^{\otimes d}$, since algebraic identities are available for this
reference measure. Then, we use comparison estimates to prove \eqref{variance asympt} in the general case.








 \subsection{A proof for the upper bound on the limiting variance}
\label{s:boundOnLimitVar}

As stated in Proposition \ref{bound limit var}, one can bound the limiting
variance $\sigma_f^2$ by a Dirichlet energy. Besides providing some control on
the amplitude of $\sigma_f^2$, we will need this inequality in the proof of
Proposition \ref{variance general CLT easy}. We now give a proof for this proposition.

\begin{proof}[Proof of Proposition \ref{bound limit var}]

Let $\mu_{sc}(\d x)=\pi^{-1}\sqrt{1-x^2}\IND_{I}(x)\d x$ be the semi-circle
measure. The associated orthonormal polynomials are the
so-called Chebyshev polynomials of the second kind
$$U_{k}(\cos \theta) = \sqrt{2}\,\frac{\sin((k+1)\theta)}{\sin\theta}.$$
For any $1\leq j\leq d$, define the measure
$$ \nu^j(\d x) = \mu_{eq}(\d x_1) \cdots \mu_{eq}(\d x_{j-1}) \mu_{sc}(\d x_j)
\mu_{eq}(\d x_{j+1})\cdots  \mu_{eq}(\d x_d),$$
so that the RHS of \eqref{Poincare ineq2} becomes
$$
\frac{1}{2}\sum_{j=1}^d \int_{I^d}  \Big(\partial_j
f(x)\Big)^2 \,\nu^j(\d x).
$$
For any $\bs
  k=(k_1,\ldots,k_d)\in\N^d$,  set $T_{\bs k}(x)= T_{k_1}(x_1)\cdots
  T_{k_d}(x_d)$, where $T_k$ are the Chebyshev polynomials \eqref{Cheby cos},
  and let
$$
V_{\bs k}^j(x) = T_{k_1}(x_1) \cdots T_{k_{j-1}}(x_{j-1}) U_{k_j}(x_j)
T_{k_{j+1}}(x_{j+1})\cdots  T_{k_d}(x_d).
$$
Thus, $(T_{\bs k})_{\bs k\in\N^d}$ and $(V^j_{\bs k})_{\bs k\in\N^d}$ respectively form an
orthonormal Hilbert basis
of $L^2(\mu_{eq}^{\otimes d})$ and $L^2(\nu^j)$. Let $f\in\mathscr C^1(I^d,\R)$, so that
$f=\sum_{\bs k\in\N^d} \hat f(\bs k) T_{\bs k}$ where $\hat f(\bs k)$ is as in \eqref{f
    hat}. Using the identity
$T_k'=kU_{k-1}$, it comes
$$
\partial_j f(x) =  \sum_{\bs k\in\N^d} k_j \hat f(\bs k) V_{\bs k}^j(x).
$$
Then, Parseval's identity in $L^2(\nu^j)$ yields
$$
\int_{I^d} \Big(\partial_j
f(x)\Big)^2 \,\nu^j(\d x) = \sum_{\bs k\in\N^d} k_j^2 \hat f(\bs k)^2.
$$
Summing over $1\leq j \leq d$, the RHS of \eqref{Poincare ineq2} equals
\eq
\label{e:identityToProve}
\frac{1}{2}\sum_{j=1}^d \int_{I^d}  \Big(\partial_j
f(x)\Big)^2 \,\nu^j(\d x)
 = \frac 12\sum_{\bs k\in\N^d} (k_1^2+\,
\cdots \,+ k_d^2)\hat f({\bs k})^2,
\qe
from which Proposition~\ref{bound limit var} easily follows.

\end{proof}

\subsection{Assumptions of Theorem~\ref{DPPMC Th1} and  outline of the proof }
\label{reg cond Th2}

We now discuss the assumptions and proof of Theorem~\ref{DPPMC Th1}.

Assume  the reference measure $\mu$ is a product of $d$ measures on $I$, and also that  $\mu$ has a density $\omega$. Then, Corollary \ref{weak conv eq} suggests that, as $N\to\infty$,
\eq
\label{Heuristic Totik}
\frac N{K_N(x,x)}\approx \frac{\omega(x)}{\omega_{eq}^{\otimes d}(x)}.
\qe
This heuristic would yield for the variance of the estimator \eqref{estimator},
\eq
\label{Heuristic}
\Var\left[\ \sum_{i=1}^N \frac{f(\bv x_i)}{K_N(\bv x_i,\bv x_i)} \ \right]\approx \frac1{N^2}\Var\left[\ \sum_{i=1}^N  f(\bv x_j)\frac{\omega(\bv x_j)}{\omega_{eq}^{\otimes d}(\bv x_j)} \ \right]\approx \frac{\Omega_{f,\omega}^2}{N^{1+1/d}}\ ,
\qe
where for the last approximation we used Proposition \ref{variance general CLT easy} with test function $f\omega/\omega_{eq}^{\otimes d}$, recalling $\Omega_{f,\omega}$ was defined in \eqref{omega f}. This would essentially yield the CLT in Theorem~\ref{DPPMC Th1} by applying Theorem \ref{Sosh} to $f_N(x)=Nf(x)/K_N(x,x)$. To make the approximation \eqref{Heuristic} rigorous, we will need extra regularity assumptions on $\mu$.

First, regarding the approximation \eqref{Heuristic Totik}, we have the following result.

\begin{theorem}{\bf (Totik)}
\label{Totik asymp unif}
Assume $\mu(\d x)=\omega(x)\d x$ with $\omega(x)=\omega_1(x_1)\cdots\omega_d(x_d)$, and that $\omega_j$ is continuous and positive on $I$. Then, for every $\epsilon>0$, we have
\eq
\label{Totik asymp unif explicit}
\frac N{K_N(x,x)} \xrightarrow[N\to\infty]{} \frac{\omega(x)}{\omega_{eq}(x)}
\qe
uniformly  for $x\in I^d_\epsilon$.
\end{theorem}

For a proof of Theorem~\ref{Totik asymp unif} when $d=1$, see \cite[Section 3.11]{Sim11} and
references therein. The case $d\geq 2$ follows by the same arguments as in the proof of Corollary \ref{weak conv eq}.

\begin{remark} It is because of Theorem \ref{Totik asymp unif} that we restrict
  $\mathscr C^1(I^d,\R)$  to the class $\mathscr C$ defined in \eqref{class test
    functions} in the assumptions of Theorem \ref{DPPMC Th1}. Unfortunately,
  there are examples of reference measures $\mu$ on $I$ such that the
  convergence \eqref{Totik asymp unif explicit} is not uniform on the whole of
  $I$. However, in order to extend $\mathscr C$ to $\mathscr C^1(I^d,\R)$ in the
  statement of Theorem \ref{DPPMC Th1}, it would be enough to have  $\sup_{x\in
    I^d_{\epsilon_N}}|N/K_N(x,x) - \omega(x)/\omega_{eq}^{\otimes d}(x)|\to 0$
  for some sequence $\epsilon_N$ going to zero as $N\to\infty$, but we were not
  able to locate such a result in the literature.
\end{remark}

Next, the first approximation in \eqref{Heuristic} requires a control on the
rate of change of $N/K_N(x,x)$. To this end, we introduce an extra assumption on
the reference measure $\mu$. More precisely, let us denote
\eq
\mathscr D_{N}(x,y)=\frac{N/K_N(x,x)- N/K_N(y,y)}{\|x-y\|},
\label{e:defDN}
\qe
and further consider the sequence of measures on $I^d\times I^d$
\eq
\label{e:defQN}
Q_N(\d x,\d y)= \frac1{N^{1-1/d}}\|x-y\|^2 K_N(x,y)^2\mu(\d x)\mu(\d y).
\qe
Our extra assumption on $\mu$ is then the following.
\begin{assumption}
\label{a:regularity assumption}
The measure $\mu$ satisfies
\eq
\label{regularity assumption}
\lim_{C\to\infty}\limsup_{\delta\to 0}\limsup_{N\to\infty}\iint_{I^d_{\epsilon}\times I^d_{\epsilon},\ \|x-y\|\leq \delta} \bv 1_{\left|\mathscr D_N(x,y)\right|> C}\ \mathscr D_N(x,y)^2\ Q_N(\d x,\d y)=0.
\qe
\end{assumption}
In plain words, this means the squared rate of change $\mathscr D_{N}(x,y)^2$ is
uniformly integrable with respect to the measures $Q_N$, at least on the
restricted domain where $\|x-y\|$ is small enough and where $x$ and $y$ are not
allowed to reach the boundary of $I^d$.

\begin{remark} When $d=1$, Lemma \ref{convergence CD out diag} states that if $\mu$ is Nevai-class then  $Q_N$
converges weakly as $N\to\infty$  towards
$$
L(\d x,\d y)=\frac1{2\pi^2}\frac{1-xy}{\sqrt{1-x^2}\sqrt{1-y^2}} \ \bs 1_{I\times I}(x,y)\d x\d y.
$$
Because the density of $L$ is smooth within $I_\epsilon\times I_\epsilon$ for
every $\epsilon>0$, one may at least heuristically understand that
\eqref{regularity assumption} reduces to the uniform integrability of $\mathscr D_N(x,y)^2$ with
respect to the Lebesgue measure instead. In higher dimension, a similar guess
can be made
, but we do not pursue
this reasoning here.
\end{remark}

We now discuss sufficient conditions for \eqref{regularity assumption} to hold true.
\begin{remark} Since, for any $\kappa>0$, we have
$$
\bv 1_{\left|\mathscr D_{N}(x,y)\right|> C}\ \mathscr D_{N}(x,y)^2\leq \frac1{C^\kappa}\left|\mathscr D_{N}(x,y)\right|^{2+\kappa},
$$
we see that condition \eqref{regularity assumption} holds true as soon as, for every $\epsilon>0$, there exists $\kappa,\delta>0$ satisfying
$$
\limsup_{N\to\infty}\iint_{I^d_{\epsilon}\times I^d_{\epsilon},\ \|x-y\|\leq \delta} \left|\mathscr D_{N}(x,y)\right|^{2+\kappa}\  Q_N(\d x,\d y)<\infty.
$$
Namely, condition \eqref{regularity assumption} is satisfied if the
$L^{2+\kappa}(Q_N)$ norm of the rate of change of $N/K_N(x,x)$ is bounded, at
least on the restricted domain where $\|x-y\|$ is small enough and $x,y$ away
from the boundary of $I^d$.
\label{r:sufficientConditionI}
\end{remark}

The following assumption, \rev{which appears in Theorems~\ref{DPPMC Th1} to \ref{DPPMC Th3},} is much stronger than Assumption~\ref{a:regularity
  assumption}, but it is easier to check in practice.

\begin{assumption}
The measure $\mu$ satisfies
\begin{itemize}
\item[\rm{(a)}] $\mu(\d x)=\omega(x)\d x$ with $\omega$ positive and continuous on $(-1,1)^d$.
\item[\rm{(b)}] For every $\epsilon>0$, the sequence
$$
\frac1N \sup_{x\in I_\epsilon^d}\Big\|\nabla K_N(x,x)\Big\|
$$
is bounded.
\end{itemize}
\label{a:regularity assumption easy}
\end{assumption}
Indeed, thanks to the rough upper bound
$$
|\mathscr D_{N}(x,y)|\leq  \sup_{x\in I_\epsilon^d}\left\| \nabla  (N/K_N(x,x))\right\|,\qquad x,y\in\bs 1_{I_\epsilon^d},
$$
we see that Assumption~\ref{a:regularity assumption} holds true as soon
as for every $\epsilon>0$, $\sup_{x\in I_\epsilon^d}\| \nabla (N/K_N(x,x))\|$ is
bounded. Under Assumption~\ref{a:regularity assumption easy}(a),
the latter follows from Assumption~\ref{a:regularity assumption easy}(b). Indeed, Theorem
\ref{Totik asymp unif} and Assumption~\ref{a:regularity assumption easy}(a)
together yield that for every $\epsilon>0$, there exists $c>0$
independent of $N$ such that $\frac1NK_N(x,x)>c$ for every $x\in I_\epsilon^d$.

We conclude this section by proving that Jacobi measures~\eqref{Jacobi measure}
satisfy Assumption~\ref{a:regularity assumption easy}, which proves our
Proposition~\ref{Jacobi = regular}. We start with a general lemma.

\begin{lemma} Assume the measures $\mu_1,\dots,\mu_d$ on $I$ satisfy
  Assumption~\ref{a:regularity assumption easy}. Then the measure $\mu_1\otimes\dots\otimes
  \mu_d$ on $I^d$ satisfies Assumption~\ref{a:regularity assumption easy}.
\label{l:separability}
\end{lemma}

\begin{proof}
We decompose the set $\Gamma_N = \{\frak b(0),\ldots,\frak b(N-1)\}\subset \N^d$ in a
convenient way. To do so, set $\sigma_j(\bs
k)=(k_1,\ldots,k_{j-1},k_{j+1},\ldots,k_d)$ and say that $\bs k\sim \bs \ell$ if
and only if $\sigma_j(\bs k)=\sigma_j( \bs \ell)$, that is, they have same
coordinates except maybe the $j$th one. We denote by $[\bs k]$ the equivalence
class under this relation. Set $N_j([\bs k])=\max\{\ell_j : \bs \ell\in[\bs
k]\cap\Gamma_N\}$. Using the notation introduced in \eqref{e:mope} and \eqref{CD kernel
  split}, it comes
\begin{align}
\partial_j K_N(x,x) & =2\sum_{\bs k=\frak b(0)}^{\frak b(N-1)}\phi_{k_j}^{(j)}(x_j)\frac{\d}{\d x_j}\phi_{k_j}^{(j)}(x_j)\prod_{\alpha\neq j}\phi_{k_\alpha}^{(\alpha)}(x_\alpha)^2\nonumber\\
& = 2\sum_{[\bs k] \in \Gamma_N/\sim}\;\prod_{\alpha\neq j}\phi^{(\alpha)}_{k_\alpha}(x_\alpha)^2 \sum_{k_j=0}^{N_j([\bs k])}\phi_{k_j}^{(j)}(x_j)\frac{\d}{\d x_j}\phi_{k_j}^{(j)}(x_j)\nonumber\\
&= \sum_{[\bs k] \in \Gamma_N/\sim}\;\prod_{\alpha\neq j}\phi^{(\alpha)}_{k_\alpha}(x_\alpha)^2 \, \frac{\d}{\d x_j} \left[ K^{(j)}_{N_j([\bs k])+1}(x_j,x_j)\right].\label{e:boundingTool}
\end{align}
Let now $\epsilon>0$. Since $\mu_j$ satisfies Assumption~\ref{a:regularity
  assumption easy}, there exists $C>0$ such that for all $x\in I_\epsilon$ and $n\in\mathbb{N}$,
$$ \left\vert \frac{\d}{\d x} \left[ K^{(j)}_{n}(x,x)\right]\right\vert \leq Cn. $$
Let $M=\lfloor N^{1/d}\rfloor$, so that $\Gamma_N\subset \mathcal{C}_{M+1}$, see
\eqref{def CM}. Thus, $N_j([\bs k])\leq M$ for all $\bs k\in \Gamma_N$. By \eqref{e:boundingTool},
\begin{eqnarray*}
\left\vert \partial_j K_N(x,x)\right\vert &\leq& C  (M+1) \sum_{[\bs k] \in
  \Gamma_N/\sim}\;\prod_{\alpha\neq j}\phi^{(\alpha)}_{k_\alpha}(x_\alpha)^2\\
&\leq&  C (M+1) \sum_{[\bs k] \in \mathcal{C}_{M+1}}\;\prod_{\alpha\neq
       j}\phi^{(\alpha)}_{k_\alpha}(x_\alpha)^2\\
&=& C (M+1) \prod_{\alpha\neq j} K_{M+1}^{(\alpha)}(x_\alpha,x_\alpha).
\end{eqnarray*}
Hence
$$
\frac 1N \left\vert \partial_j K_N(x,x)\right\vert \leq C \frac{M+1}{M}\prod_{\alpha\neq j} \frac{1}{M}K_{M+1}^{(\alpha)}(x_\alpha,x_\alpha),
$$
\rev{and the lemma follows with Theorem~\ref{Totik asymp unif}.}
\end{proof}

\begin{lemma}
\label{Bound derivative Jacobi} Let $\alpha,\beta>-1$, then the measure
$$
(1-x)^\alpha (1+x)^\beta \IND_I(x)\d x
$$
satisfies Assumption~\ref{a:regularity assumption easy}.
\end{lemma}

\begin{proof} Let $\epsilon>0$ be fixed. For convenience, Section~\ref{s:reductionToProba} allows
  us to work with the probability measure
$$
\mu^{(\alpha,\beta)}(\d x)=\omega^{(\alpha,\beta)}(x) \d x,\qquad \omega^{(\alpha,\beta)}(x)=\frac1{c_{\alpha,\beta}}(1-x)^\alpha(1+x)^\beta,
$$
where the normalization constant reads
$$
c_{\alpha,\beta}=
2^{\alpha+\beta+1}\frac{\Gamma(\alpha+1)\Gamma(\beta+1)}{\Gamma(\alpha+\beta+1)}
$$
and $\Gamma$ is the Euler Gamma function.

Denote by $(\phi_n^{(\alpha,\beta)})_{n\in\N}$ the associated orthonormal polynomials. They satisfy
$$
\phi_n^{(\alpha,\beta)}(x)=\frac{P_n^{(\alpha,\beta)}(x)}{\sqrt{h_n^{(\alpha,\beta)}}},
$$
where the $P_n^{(\alpha,\beta)}$'s are the Jacobi polynomials (we refer to \citep{Sze75} for definitions and properties) and
$$
h_n^{(\alpha,\beta)}=\|P_n^{(\alpha,\beta)}\|^2_{L^2(\mu^{(\alpha,\beta)})}=\frac1{n! (\alpha+\beta+2n+1)}\frac{\Gamma(\alpha+\beta+1)\Gamma(\alpha+n+1)\Gamma(\beta+n+1)}{\Gamma(\alpha+1)\Gamma(\beta+1)\Gamma(\alpha+\beta+n+1)},
$$
and moreover
\eq
\label{jacobi derivative}
(\phi_n^{(\alpha,\beta)})'=\frac{u_{\alpha,\beta}}2\sqrt{n(n+\alpha+\beta+1)}\ \phi_{n-1}^{(\alpha+1,\beta+1)},\qquad u_{\alpha,\beta}=\sqrt{\frac{(\alpha+\beta+1)(\alpha+\beta+2)}{(\alpha+1)(\beta+1)}}.
\qe
This yields
\begin{align}
\label{KN' Jacobi}
\frac\d{\d x}K_N(x,x)& =2\sum_{k=1}^{N-1}\phi_k^{(\alpha,\beta)}(x)(\phi_k^{(\alpha,\beta)})'(x)\nonumber\\
& =u_{\alpha,\beta}\sum_{k=1}^{N-1}\sqrt{k(k+\alpha+\beta+1)}\ \phi_k^{(\alpha,\beta)}(x) \phi_{k-1}^{(\alpha+1,\beta+1)}(x).
\end{align}
By \citep{KMVV04}, we have as $k\to\infty$, uniformly in $x=\cos\theta\in I_\epsilon$,
\eq
\label{asympt Jacobi}
\phi_k^{(\alpha,\beta)}(\cos\theta)=\sqrt{\frac{2}{ \omega^{(\alpha,\beta)}(x) \pi\sqrt{1-x^2}}} \ \cos\left(\big(k+\frac12(\alpha+\beta+1)\big)\theta-\frac\pi2(\alpha +\frac12)\right)+O(1/k).
\qe
As a consequence, we obtain in the same asymptotic regime,
\eq
\label{asympt Jacobi prime}
\phi_{k-1}^{(\alpha+1,\beta+1)}(\cos\theta)=\sqrt{\frac{2}{\omega^{(\alpha,\beta)}(x) \pi\sqrt{1-x^2}}} \ \sin\left(\big(k+\frac12(\alpha+\beta+1)\big)\theta-\frac\pi2(\alpha +\frac12)\right)+O(1/k).
\qe
Now \eqref{asympt Jacobi} implies that the
$P_k^{(\alpha,\beta)}(x)$'s are bounded uniformly for $x\in I$ and $k\in\N$. Using moreover that $2\sin(u)\cos(u)=\sin(2u)$ and combining \eqref{KN' Jacobi}--\eqref{asympt Jacobi prime}, we obtain for some $C_1,C_2>0$ that
$$
\sup_{x\in I_\epsilon}\left|\frac\d{\d x}K_N(x,x)\right|\leq C_1\sup_{x\in I_\epsilon}\left|\sum_{k=1}^{N-1} k \sin\left(\big(2k+\alpha+\beta+1\big)\theta-\pi(\alpha +\frac12)\right)\right| +C_2
$$
where we recall the relation $x=\cos\theta$. Next, we write
$$
\left|\sum_{k=1}^{N-1} k \sin\left(\big(2k+\alpha+\beta+1\big)\theta-\pi(\alpha +\frac12)\right)\right|\leq \left|\sum_{k=1}^{N-1} k e^{i((2k+\alpha+\beta+1)\theta-\pi(\alpha +\frac12))}\right|^{1/2}
$$
and then
\begin{align}
\label{sum trigo}
\sum_{k=1}^{N-1} k e^{i((2k+\alpha+\beta+1)\theta-\pi(\alpha +\frac12))}
& =e^{i((\alpha+\beta+1)\theta-\pi(\alpha +\frac12))}\sum_{k=1}^{N-1} k e^{2ik\theta} \nonumber\\
&= \frac{1}{2i}e^{i((\alpha+\beta+1)\theta-\pi(\alpha +\frac12))}\frac{\d}{\d\theta}\sum_{k=0}^{N-1} e^{2ik\theta}\nonumber\\
&= \frac{1}{2i}e^{i((\alpha+\beta+1)\theta-\pi(\alpha +\frac12))}\frac{\d}{\d\theta}\Big(\frac{e^{i(N-1)\theta}\sin (N\theta)}{\sin(\theta)}\Big).
\end{align}
Since the absolute value of the right hand side of \eqref{sum trigo} is bounded by $CN/\sin^2(\theta)$ for some $C>0$ independent on $N$ and $\theta$, the lemma follows.

\end{proof}

Lemmas~\ref{l:separability} and \ref{Bound derivative Jacobi} combined yield
Proposition~\ref{Jacobi = regular}.

\subsection{Proof of Theorem~\ref{DPPMC Th3}}
\label{proofSelfnormal}
\begin{proof}[Proof of Theorem~\ref{DPPMC Th3}]
\rev{
We proceed in two steps. The first step is to prove a bivariate CLT for the vector $(\mathscr{I}_N(f), \mathscr{I}_N(1))$, and the second step is to apply the delta method to obtain a central limit theorem for the ratio.
}

\rev{
Let $t,v\in\mathbb{R}$. We apply Theorem~\ref{DPPMC Th2} to $tf + vh$, where $h$ is any function in $\mathscr C$ that is equal to $1$ on $I_\epsilon^d$. This leads to
  $$\sqrt{N^{1+1/d}}\left(t\mathscr{I}_N(f) +v\mathscr{I}_N(1) - t\int f\omega_u\d x - vZ\right) \rightarrow \mathcal N(0,Z^2\Omega_{tf+vh,\omega}^2),$$
  where
  $$ \Omega_{tf+vh,\omega}^2 = t^2\Omega_{f,\omega}^2 + v^2 \Omega_{1,\omega}^2 + 2tv c_{f,\omega}.$$
  By the Cr\'amer-Wold theorem \cite[Corollary 5.5]{Kal06}, we obtain a bivariate CLT
  $$
  \sqrt{N^{1+1/d}}\left[\begin{pmatrix} \mathscr{I}_N(f)\\ \mathscr{I}_N(1)\end{pmatrix} - \begin{pmatrix}\int f \omega_u \d x\\ Z\end{pmatrix}\right] \rightarrow \mathcal N\left(0,Z^2\begin{pmatrix} \Omega_{f,\omega}^2 & c_{f,\omega} \\ c_{f,\omega} & \Omega_{1,\omega}^2\end{pmatrix} \right).
  $$
  }
\rev{
Now we apply the delta method \cite[Section 7.1.3]{Sch12} to obtain
$$
\sqrt{N^{1+1/d}}\left[\frac{\mathscr{I}_N(f)}{\mathscr{I}_N(1)} - \int f\d \mu\right] \rightarrow \mathcal{N}\left(0, Z^2 \begin{pmatrix}\frac1Z &-\frac{\int f\omega_u\d x}{Z^2}\end{pmatrix} \begin{pmatrix} \Omega_{f,\omega}^2 & c_{f,\omega} \\ c_{f,\omega} & \Omega_{1,\omega}^2\end{pmatrix} \begin{pmatrix}\frac1Z \\-\frac{\int f\omega_u\d x}{Z^2}\end{pmatrix}\right),
$$
which is precisely \eqref{e:selfNormalizedWeightedCLT}.
}
\end{proof}

\section{CLT for multivariate OP Ensembles: proof of Theorem~\ref{th CLT general}}
\label{proof th1 section}

In this section we prove Proposition~\ref{variance general CLT easy}. As explained in Section
\ref{s:sketch proof CLT MOPE}, Theorem \ref{th CLT general} follows from this proposition.

\label{full proof}

\subsection{A useful representation of the covariance}

\begin{lemma} Let $\bv x_1,\ldots, \bv x_N$ be random variables drawn from a multivariate OP Ensemble with reference measure $\mu$. For any   multivariate polynomials $P,Q$, we have
\label{cov representation}
\[
\mathbb C{\rm ov}\left[\, \sum_{i=1}^NP(\bv x_i),\sum_{i=1}^NQ(\bv x_i)\right]=\sum_{n=0}^{N-1}\sum_{m=N}^\infty \langle P\phi_n,\phi_m\rangle\langle Q\phi_n,\phi_m\rangle,
\]
where $\langle \, \cdot \, , \, \cdot\, \rangle$ refers to the scalar product of $L^2(\mu)$.
\end{lemma}

\begin{proof}
We  start from the standard formula
\begin{multline}
\label{cov def}
\mathbb C{\rm ov} \left[\, \sum_{i=1}^NP(\bv x_i),\sum_{j=1}^NQ(\bv x_i)\right]\\
=\int P(x)Q(x)K_N(x,x)\mu(\d x)-\iint P(x)Q(y)K_N(x,y)^2\mu(\d x)\mu(\d y),
\end{multline}
which follows from \eqref{def k cor}--\eqref{def cor DPP} with $n=1,2$ and that $K_N(x,y)$ is symmetric.
On the one hand, it  follows from the definition of $K_N$ that
\eq
\label{cov part 1}
\iint P(x)Q(y)K_N(x,y)^2\mu(\d x)\mu(\d y)=\sum_{n=0}^{N-1}\sum_{m=0}^{N-1}\langle P\phi_n,\phi_m\rangle\langle Q\phi_n,\phi_m\rangle.
\qe
On the other hand, by using the decomposition (where the sum is finite since $P$ is polynomial)
\[
P\phi_n=\sum_{m=0}^\infty\langle P\phi_n,\phi_m\rangle \phi_m
\]
together with the identity
\[
\int P(x)Q(x)K_N(x,x)\mu(\d x)=\sum_{n=0}^{N-1}\langle P\phi_n,Q\phi_n\rangle,
\]
we obtain
\eq
\label{cov part 2}
\int P(x)Q(x)K_N(x,x)\mu(\d x)=\sum_{n=0}^{N-1}\sum_{m=0}^\infty\langle P\phi_n,\phi_m\rangle\langle Q\phi_n,\phi_m\rangle.
\qe
Lemma \ref{cov representation} then follows by combining \eqref{cov def}, \eqref{cov part 1} and \eqref{cov part 2}.
\end{proof}

\subsection{Covariance asymptotics: the Chebyshev case}
\label{section Cov Cheby}

We first investigate the case of the product measure $\mu^{\otimes d}_{eq}$,
where $\mu_{eq}$ defined in \eqref{mueq} is the equilibrium measure of $I$.
Recalling the definition \eqref{Cheby cos}, the multivariate Chebyshev polynomials
\eq
\label{multivariate Cheby}
T_{\bs k}(x_1,\ldots,x_d)=T_{k_1}(x_1)\cdots T_{k_d}(x_d),\qquad \bs k=(k_1,\ldots k_d)\in\N^d,
\qe
satisfy the orthonormality conditions
$$
\int T_{\bs k}(x)T_{\bs \ell}(x)\mueqd(\d x)=\delta_{\bs k \bs \ell},\qquad \bs k ,\bs \ell\in\N^d.
$$
We shall see that the family $(T_{\bs k})_{\bs k\in\N^d}$ diagonalizes the
covariance structure associated with our point process.

\begin{proposition}
\label{cov cheby prop}
Let $\bv x_1^*,\ldots,\bv x_N^*$ be drawn according to the multivariate OP
Ensemble associated with $\mueqd$. Then, given any multi-indices $\bs k=(k_1,\ldots,k_d)\in\N^d$ and $\bs \ell=(\ell_1,\ldots,\ell_d)\in\N^d$, we have
 \[
\lim_{N\rightarrow\infty}\frac{1}{N^{1-1/d}}\,\mathbb C{\rm ov}\left[\,\sum_{i=1}^NT_{\bs k}(\bv x_i^*),\sum_{i=1}^NT_{\bs \ell}(\bv x_i^*)\right]=\begin{cases} \displaystyle
\frac{1}{2}(k_1+\,\cdots\,+k_d) & \mbox{if }\bs k = \bs\ell,\\
\displaystyle 0 & \mbox{if } \bs k\neq \bs \ell.
 \end{cases}
\]
\end{proposition}

As a warm-up, let us first prove the proposition when $d=1$.

\begin{proof}[Proof of Proposition \ref{cov cheby prop} when $d=1$] Throughout
  this proof, $\langle \cdot,\cdot\rangle$ denotes the inner product in
  $L^2(\mueqd)$. For every $k,\ell\in\N$, Lemma \ref{cov representation} provides
\eq
\label{cov cheby d=1}
\mathbb C{\rm ov}\left[\,\sum_{i=1}^NT_{ k}(\bv x_i^*),\sum_{i=1}^NT_{ \ell}(\bv x_i^*)\right]=\sum_{n=0}^{N-1}\sum_{m=N}^\infty \langle T_{ k}T_{ n},T_{ m}\rangle\langle T_{ \ell}T_{ n},T_{ m}\rangle.
\qe
First, notice that if $k$ or $\ell$ is zero, then the right-hand side of \eqref{cov cheby d=1} vanishes because $\langle T_n,T_m\rangle=\delta_{nm}$, and hence we can assume both $k,\ell$ are non-zero.  Next,  \eqref{Cheby cos}  yields the multiplication formula
\eq
\label{product Cheby}
T_kT_n=\frac{1}{\sqrt 2}\, T_{n+k}\,\bv 1_{kn\neq 0}+ \left(\frac{1}{\sqrt{2}}\right)^{\bv 1_{nk\neq 0}\bv 1_{n\neq k}}T_{|n-k|},\qquad  k,n\in\N.
\qe
Combined with the orthonormality relations, this yields for any $n,m\in\N$  and $k>0$
\eq
\label{product Cheby 2}
\langle T_{ k}T_{ n},T_{ m}\rangle\\
= \frac 1{\sqrt 2} \, \bs 1_{n+k=m}\bs 1_{n\neq 0}+\left(\frac{1}{\sqrt 2}\right)^{\bs 1_{n\neq 0}\bs 1_{n\neq k}} \bs 1_{|n-k|=m}.
\qe
Hence, if $n,m\in\N$  moreover  satisfy $n<m$ and $m>\max(k,\ell)$, then we have
\eq
\label{Cheby d=1 mult}
\langle T_{ k}T_{ n},T_{ m}\rangle\langle T_{ \ell}T_{ n},T_{ m}\rangle=  \frac{1}{2} \bs 1_{n\neq0} \bs 1_{n+k=m}\bs 1_{\ell+n=m}.
\qe
By plugging \eqref{Cheby d=1 mult} into \eqref{cov cheby d=1},  we obtain  for every $N> \max(k,\ell)$,
\begin{align*}
\mathbb C{\rm ov}\left[\,\sum_{i=1}^NT_{ k}(\bv x_i^*),\sum_{i=1}^NT_{ \ell}(\bv x_i^*)\right]
& =\frac 12 \sum_{n=1}^{N-1}\sum_{m=N}^\infty\bs 1_{k+n=m}\bs 1_{\ell+n=m}\\
&  = \frac 12 k \, \bs 1_{k= \ell},
 \end{align*}
and the proposition follows when $d=1$.
\end{proof}

We now provide a proof for the higher-dimensional case. We also use the multiplication formula \eqref{product Cheby} in an essential way, although the setting is  more involved. We recall that we introduced the bijection $\frak b:\N\to\N^d$ associated with the graded lexicographic order in Section \ref{graded lex order}.

\begin{proof}[Proof of Proposition \ref{cov cheby prop} when $d\geq 2$] Fix multi-indices $\bs k=(k_1,\ldots,k_d)\in\N^d$ and $\bs \ell=(\ell_1,\ldots,\ell_d)\in\N^d$, and also set
\[
S=\big\{j: \; k_j\neq 0\big\},\qquad  S'=\big\{j: \; \ell_j\neq 0\big\}.
\]
Thanks to Lemma \ref{cov representation}, we can write
\eq
\label{cov tcheby}
\mathbb C{\rm ov}\left[\,\sum_{i=1}^NT_{\bs k}(\bv x_i^*),\sum_{i=1}^NT_{\bs \ell}(\bv x_i^*)\right]= \sum_{(\bs n,\bs m)\in\mathbb A_N} \langle T_{\bs k}T_{\bs n},T_{\bs m}\rangle\langle T_{\bs \ell}T_{\bs n},T_{\bs m}\rangle,
\qe
where we introduced for convenience the set
\eq
\label{def AN}
\mathbb A_N=\Big\{(\bs n,\bs m)\in\N^d\times\N^d : \quad \bs n\leq \frak b(N-1),\quad \bs m\geq \frak b(N)\Big\}.
\qe
Next, using \eqref{multivariate Cheby}, the orthonormality
relations for the Chebyshev polynomials and \eqref{product Cheby 2}, we  obtain
\begin{align}
\label{split formula}
\langle T_{\bs k}T_{\bs n},T_{\bs m}\rangle = &\; \; \langle T_{k_1}T_{n_1},T_{m_1}\rangle_{L^2(\mu_{eq})}\cdots \langle T_{k_d}T_{n_d},T_{m_d}\rangle_{L^2(\mu_{eq})}\\
 \; = &\quad \Big(\prod_{j\notin S}\bv 1_{n_j=m_j}\Big) \sum_{P \subset S}
\left(\frac{1}{\sqrt{2}}\right)^{| P| +|\{j\in S\setminus P :\; n_j\neq 0, \, n_j\neq k_j\}|}
\nonumber\\
\label{mult formula d>1}
 & \times \Big(\prod_{j\in P}\bv 1_{n_j+k_j=m_j}\bv 1_{n_j\neq 0}\Big)
 \Big(\prod_{j\in S\setminus P}\bv 1_{|n_j-k_j|=m_j}\Big),
\end{align}
where $|A|$ stands for the cardinality of the set $A$.

First, notice that if $S\neq S'$ then the right hand side of
\eqref{cov tcheby} vanishes.  Indeed, if $S\neq S'$, then there exists
$\alpha\in\{1,\ldots,d\}$ such that $k_\alpha=0$ and $\ell _\alpha\neq
0$ (or the other way around, but the argument is symmetric). It then
follows from \eqref{mult formula d>1} that $\langle T_{\bs k}T_{\bs
  n},T_{\bs m}\rangle$ vanishes except if $n_\alpha=m_\alpha$, and
moreover that $\langle T_{\bs \ell}T_{\bs n},T_{\bs m}\rangle$
vanishes except if  $|n_\alpha\pm \ell_\alpha |=m_\alpha$. Since
$\ell_\alpha\neq 0$, it holds  $\langle T_{\bs k}T_{\bs n},T_{\bs
  m}\rangle\langle T_{\bs \ell}T_{\bs n},T_{\bs m}\rangle=0$   for
every $(\bs n,\bs m)\in\N^d\times\N^d$, and our claim
follows. Moreover, because $(\bs n,\bs m)\in\mathbb A_N$ yields the
existence of $\alpha\in\{1,\ldots,d\}$ such that $n_\alpha<m_\alpha$,
one can see from \eqref{split formula} that $\langle T_{\bs
  k}T_{\bs n},T_{\bs m}\rangle$ vanishes for every $(\bs n,\bs
m)\in\mathbb A_N$ if $\bs k=(0,\ldots,0)$.  We
henceforth assume that $S=S'\neq \emptyset$, for the covariance not to be trivial.

By combining  \eqref{cov tcheby} with \eqref{mult formula d>1},   we obtain

\[
 \mathbb C{\rm ov}\left[\, \sum_{i=1}^NT_{\bs k}(\bv x_i^*),\sum_{i=1}^NT_{\bs \ell}(\bv x_i^*)\right] = \sum_{P,Q\subset S}\sum_{(\bs n,\bs m)\in \mathbb A_N[P,Q]}\left(\frac{1}{ \sqrt 2}\right)^{\sigma[P, Q](\bs n)},
\]
where we introduced the subsets
\begin{multline}
\label{def sets}
\mathbb A_N[P,Q]=
\Bigg\{(\bs n,\bs m)\in\mathbb A_N \;\Bigg| \;
\begin{matrix*}
n_j+k_j=m_j,\\
n_j+\ell_j=m_j,
\end{matrix*}
\quad
\begin{matrix*}
n_j\neq 0,\\
n_j\neq 0,
\end{matrix*}
\quad
\begin{matrix*}
\mbox{ if } j\in P\\
\mbox{ if } j\in Q
\end{matrix*} \; ;
\\
\begin{matrix*}
|n_j-k_j|=m_j,\\
|n_j-\ell_j|=m_j,
\end{matrix*}
\quad
\begin{matrix*}
\mbox{ if } j\in S\setminus P\\
\mbox{ if } j\in S\setminus Q
\end{matrix*}\; ;
\qquad
n_j=m_j,\quad \mbox{ if } j\notin S \;\Bigg\}
\end{multline}
and set for convenience
\begin{multline}
\label{def sigma}
\sigma[P,Q](\bs n)=| P | + | Q |
+|\big\{j\in S\setminus P :\; n_j\neq 0,\, n_j\neq k_j\big\}|\\+|\big\{j\in S\setminus Q :\; n_j\neq 0,\, n_j\neq \ell_j\big\}|.
\end{multline}

Notice from \eqref{def sets}  if $k_\alpha=\ell_\alpha\neq 0$ and
$\mathbb A_N[P,Q]\neq \emptyset$ then necessarily $\alpha\in P\cap Q$
or $\alpha\in (S\setminus P)\cap(S\setminus Q)$. In particular, if
$\bs k=\bs \ell$ then $\mathbb
  A_N[P,Q]=\emptyset$ unless $P=Q$. Thus,
\begin{multline}
\label{cov ugly sum}
 \mathbb C{\rm ov}\left[\, \sum_{i=1}^NT_{\bs k}(\bv x_i^*),\sum_{i=1}^NT_{\bs \ell}(\bv x_i^*)\right]  = \bs 1_{\bs k\, =\, \bs \ell}\sum_{P\subset S}\sum_{(\bs n,\bs m)\in \mathbb A_N[P,P]}\left(\frac{1}{ \sqrt 2}\right)^{\sigma[P,P](\bs n)}\\
 +\bs 1_{\bs k\, \neq\, \bs \ell}\sum_{P,Q\subset S}\sum_{(\bs n,\bs m)\in \mathbb A_N[P,Q]}\left(\frac{1}{ \sqrt 2}\right)^{\sigma[P,Q](\bs n)}.
\end{multline}
Our goal is now to show that for every $P,Q\subset S$ the following holds true. As $N\to\infty$,
if we assume  $\bs k=\bs \ell$, then
\begin{align}
\label{main contrib}
\sum_{(\bs n,\bs m)\in \mathbb A_N[P,P]}\left(\frac{1}{ \sqrt 2}\right)^{\sigma[P,P](\bs n)}
=\left(\frac{1}{2}\right)^{ |S|}\Big(\sum_{j\in P} k_j\Big)N^{1-1/d}+o(N^{1- 1/d}),
\end{align}
and, if instead $\bs k\neq \bs \ell$, then
\eq
\label{residual contrib}
 \big|\mathbb A_N[P,Q]\big|=o(N^{1- 1/d}).
\qe
Since an easy rearrangement argument together with the definition of $S$ yield
\begin{align*}
\sum_{P\subset S}\Big( \sum_{j\in P}k_j\Big) & =\frac 12 \sum_{P\subset S}\Big( \sum_{j\in S}k_j\Big)\\
 & =\Big(\sum_{j\subset S} k_j\Big) 2^{|S|-1} = \Big(\sum_{j=1}^d k_j\Big) 2^{|S|-1} ,
\end{align*}
Proposition \ref{cov cheby prop} would then follow from \eqref{cov ugly sum}--\eqref{residual contrib}.

We now turn to the proof of \eqref{main contrib} and \eqref{residual contrib}.

\paragraph*{Truncated sets and consequences.}
Given distinct $\alpha_1,\ldots,\alpha_p\in\{1,\ldots,d\}$, we  introduce the truncated sets
\begin{multline}
\label{def trunc sets}
\mathbb A_N[P,Q ; \alpha_1,\ldots,\alpha_p]\\
=\mathbb A_N[P,Q] \cap \big\{n_{\alpha_1}\leq \max (k_{\alpha_1},\ell_{\alpha_1})\big\}\cap\, \cdots \,\cap \big\{n_{\alpha_p}\leq \max (k_{\alpha_p},\ell_{\alpha_p})\big\}.
\end{multline}
By definition of $\frak b$ and $\mathbb A_N$, if  $N=M^d$ then $\mathbb
A_{M^d}=\mathcal C_M\times (\N^d\setminus\mathcal C_M)$ where we recall
\eq
\label{def C_M}
\mathcal C_M=\Big\{\bs n\in\N^d:\; 0\leq n_1,\ldots, n_d\leq M-1\Big\}.
\qe
Moreover, if for an arbitrary $N$  we denote by $M=\lfloor
N^{1/d}\rfloor$ the integer satisfying $M^d\leq N <(M+1)^d$, then
$\frak b(N)\in \mathcal C_{M+1}$ and thus, for any $(\bs n,\bs
m)\in\mathbb A_N$, we  have $\bs n\in\mathcal C_{M+1}$.
As a consequence, for every $P,Q\subset S$,  we have the rough upper
bound $|\mathbb A_N[P,Q ;
\alpha_1,\ldots,\alpha_p]|=\cO(M^{d-p})$. In particular,
\eq
\label{rough UB}
\big|\mathbb A_N[P,Q ; \alpha_1,\ldots,\alpha_p]\big|=o(N^{1- 1/d})\quad \mbox{ for every } p\geq 2.
\qe

In order to restrict ourselves to the easier setting where $N$ is a power of $d$, we will use the following lemma. Its proof uses in a crucial way the graded lexicographic order we chose to equip $\N^d$ with, and it is deferred to the end of the present proof.

\begin{lemma}
\label{non-decreasing} Assume $\bs k=\bs \ell$. For every $P\subset S$, $\alpha\in S\setminus P$ and for every  $N>\max(k_1^d,\ldots,k_d^d)$, we have

\begin{itemize}
\item[{\rm (a)}]
$\big| \mathbb A_{N}[P,P] \big| \leq \big|\mathbb A_{N+1}[P,P]\big|$,\\
\item[{\rm (b)}]
$\big|\mathbb A_{N}[P,P;\alpha]\big|\leq \big|\mathbb A_{N+1}[P,P;\alpha]\big|$.
\end{itemize}
\end{lemma}\

\rev{\paragraph{Proof of \eqref{main contrib}.}} Assume $\bs k=\bs \ell$. As a consequence of Lemma \ref{non-decreasing} (a), if we set $M=\lfloor N^{1/d}\rfloor$ then we have for every $N$ large enough
\begin{align*}
\sum_{(\bs n,\bs m)\in \mathbb A_{M^d}[P,P]}\left(\frac{1}{ \sqrt 2}\right)^{\sigma[P, P](\bs n)} & \leq \sum_{(\bs n,\bs m)\in \mathbb A_N[P,P]}\left(\frac{1}{ \sqrt 2}\right)^{\sigma[P, P](\bs n)}  \leq \sum_{(\bs n,\bs m)\in \mathbb A_{(M+1)^d}[P,P]}\left(\frac{1}{ \sqrt 2}\right)^{\sigma[P, P](\bs n)}.
\end{align*}
Thus, it is enough to prove that, as $M\to\infty$,
\begin{align}
\label{main contrib M}
\sum_{(\bs n,\bs m)\in \mathbb A_{M^d}[P,P]}\left(\frac{1}{ \sqrt 2}\right)^{\sigma[P, P](\bs n)}
=\left(\frac{1}{2}\right)^{| S|}\Big(\sum_{j\in P} k_j\Big)M^{d-1}+o(M^{d-1}),
\end{align}
in order  to establish \eqref{main contrib}. To do so, for any $P\subset S$ and $\alpha\in S\setminus P$, we set
\begin{align}
\label{def A*}
\mathbb A_{M^d}^*[P]& =\mathbb A_{M^d}[P,P]\cap\big\{n_j>k_j \mbox{ for all }j\in S\setminus P\big\},\\
\label{def A* trunc}
\mathbb A_{M^d}^*[P;\alpha]& =\mathbb A_{M^d}[P,P;\alpha]\cap\big\{n_j>k_j \mbox{ for all }j\in S\setminus (P\cup\{\alpha\})\big\},
\end{align}
and use the following lemma; its proof is deferred to the end of the actual proof.

\begin{lemma}
\label{estimates M}
Assume $\bs k=\bs \ell$. For every $P\subset S$ and $\alpha\in S\setminus P$, we have as $M\to\infty$
\begin{align}
\label{estimates M main}
\big| \mathbb A_{M^d}^*[P] \big|& =\Big(\sum_{j\in P}k_j\Big)M^{d-1}+o(M^{d-1}),\\
\label{estimates M residual}
\big| \mathbb A_{M^d}^*[P;\alpha]\big|& =o(M^{d-1}).
\end{align}
\end{lemma}\

\noindent Next, as a consequence of the rough upper bound \eqref{rough UB} and \eqref{estimates M residual}, we can write
\begin{align}
\label{A* win}
\big| \mathbb A_{M^d}[P,P] \big|  =  & \; \big|\mathbb A_{M^d}[P,P]\cap\big\{n_\alpha>k_\alpha \mbox{ for all }\alpha\in S\setminus P\big\}\big|\nonumber\\
& +  \big|\mathbb A_{M^d}[P,P]\cap\big\{n_\alpha\leq k_\alpha \mbox{ for at least one }\alpha\in S\setminus P\big\}\big|\nonumber\\
= & \; \big|\mathbb A^*_{M^d}[P]\big|+\sum_{\alpha\in S\setminus P}\big|\mathbb A^*_{M^d}[P;\alpha]\big|+o(M^{d-1})\nonumber\\
= & \;  \big|\mathbb A^*_{M^d}[P]\big|+o(M^{d-1}).
\end{align}
Since for any $(\bs n,\bs m)\in \mathbb A^*_{M^d}[P]$  we have
$\sigma[P, P](\bs n)=2 |S|$, see \eqref{def sigma} and \eqref{def A*},
the estimate  \eqref{main contrib M} follows from \eqref{A* win} and
\eqref{estimates M main}, and the proof of \eqref{main contrib} is therefore complete.

\rev{\paragraph*{Proof of  \eqref{residual contrib}.}}
 Assume now that $\bs
k\neq\bs \ell$. Since $\bs k$ and $\bs \ell$ have the same zero
components, it follows that neither $k_\alpha$ nor $\ell_\alpha$ is
zero. Thus, \eqref{def sets} yields that if  $k_\alpha\neq
\ell_\alpha$ and  $\mathbb A_{N}[P,Q]\neq \emptyset$, then either
$\alpha\in P\cap (S\setminus Q)$ or $\alpha\in Q\cap (S \setminus P)$ and moreover, for any $(\bs n ,\bs m)\in \mathbb A_{N}[P,Q]$, we have
\[
2n_\alpha=|k_\alpha-\ell_\alpha|,\qquad 2m_\alpha=k_\alpha+\ell_\alpha.
\]
In particular $\mathbb A_{M^d}[P,Q]=\mathbb A_{M^d}[P,Q;\alpha]$. Thus,  by virtue of the rough upper bound \eqref{rough UB}, we can assume in the proof of \eqref{residual contrib} that $\bs k$ and $\bs \ell$ differ by exactly one coordinate, namely there exists $\alpha\in\{1,\ldots,d\}$ such that $k_\alpha\neq \ell_\alpha$ and  $k_j=\ell_j$ for every $j\neq \alpha$.
In this setting, $\mathbb A_{N}[P,Q]\neq \emptyset$ then yields $P\setminus\{\alpha\}=Q\setminus\{\alpha\}$ and, if $(\bs n,\bs m)\in\mathbb A_N[P,Q]$, then  $(n_\alpha,m_\alpha)$ satisfies the equations
\[
\begin{cases}
n_\alpha+k_\alpha=m_\alpha,\qquad \ell_\alpha-n_\alpha=m_\alpha,\qquad n_\alpha\neq 0,\qquad \mbox{ if } \alpha\in P\\
n_\alpha+\ell_\alpha=m_\alpha,\qquad k_\alpha-n_\alpha=m_\alpha,\qquad n_\alpha\neq 0,\qquad \mbox{ if } \alpha\in Q
\end{cases} .
\]
By weakening these constraints to
\[
\begin{cases}
| \ell_\alpha-n_\alpha|=m_\alpha,\qquad n_\alpha\leq \ell_\alpha,\qquad  \mbox{ if } \alpha\in P\\
 |k_\alpha-n_\alpha|=m_\alpha,\qquad n_\alpha\leq k_\alpha,\qquad  \mbox{ if } \alpha\in Q
\end{cases},
\]
we obtain the upper bound
\eq
\label{residual A}
\big | \mathbb A_{N}[P,Q] \big | \leq
\begin{cases}
\big | \mathbb A_{N}^{(\bs \ell,\bs \ell)}[Q,Q;\alpha] \big |& \mbox{ if } \alpha\in P\\
\big | \mathbb A_{N}^{(\bs k,\bs k)}[P,P;\alpha] \big |& \mbox{ if } \alpha\in Q
\end{cases}
\qe
where $\mathbb A_{N}^{(\bs k,\bs \ell)}[P,Q;\alpha]$ is defined as in \eqref{def sets},\eqref{def trunc sets} but we emphasized the multi-indices $\bs k,\bs \ell$ which are involved.  By setting  $M=\lfloor N^{1/d}\rfloor+1$,  we obtain from Lemma \ref{non-decreasing} (b), the rough upper bound \eqref{rough UB} and  \eqref{def A* trunc},\eqref{estimates M residual} that,  as $N\to\infty$,
\begin{align}
& \big | \mathbb A_{N}^{(\bs k,\bs k)}[P,P;\alpha] \big |\nonumber\\
 & \qquad \leq \; \big | \mathbb A_{M^d}^{(\bs k,\bs k)}[P,P;\alpha] \big | \nonumber\\
  &\qquad =  \; \big | \mathbb A_{M^d}^{(\bs k,\bs k)}[P,P;\alpha]\cap\big\{n_j>k_j \mbox{ for all }j\in S\setminus (P\cup\{\alpha\})\big\}\big |+o(N^{1-1/d})\nonumber\\
 & \qquad = \; \quad o(N^{1-1/d}),
\end{align}
and similarly
\eq
\label{residual B}
\big |\mathbb A_{N}^{(\bs \ell,\bs \ell)}[Q,Q;\alpha] \big |= o(N^{1- 1/d}).
\qe
By combining \eqref{residual A}--\eqref{residual B}, we have finally proved \eqref{residual contrib} and the proof of Proposition \ref{cov cheby prop} is thus complete, up to the proof of Lemmas \ref{non-decreasing} and \ref{estimates M}.
\end{proof}

We now provide proofs for the remaining lemmas.

\begin{proof}[Proof of Lemma \ref{non-decreasing}]
Let $(\bs n,\bs m)\in \mathbb A_{N}[P,P]$. Then $(\bs n,\bs m)\notin \mathbb A_{N+1}[P,P]$ if and only if $\bs m = \frak b(N)$. Since  $\frak b(N)\in \mathcal C_{M+1}\setminus \mathcal C_M$, where $M=\lfloor N^{1/d}\rfloor$ and $\mathcal C_M$ has been introduced in \eqref{def C_M}, there exists $j\in\{1,\ldots,d\}$ such that $\frak b(N)_j=M$; let $j_*$ be the smallest $j$ satisfying this property. Notice also  $\bs n\in \mathcal C_{M+1}$ and $\bs m\in\mathbb N^d\setminus\mathcal C_M$. As soon as $M>\max(k_1,\ldots,k_d)$, that we assume from now, the equality $\bs m = \frak b(N)$  can only happen if $j_*\notin S\setminus P$.  Indeed, if $j_* \in S\setminus P$, then $m_{{j_*}}= |n_{{j_*}}-k_{{j_*}}|\leq  \max(n_{{j_*}}-1,k_{j_*})\leq M-1$. As a consequence,
\[
\big| \mathbb A_{N}[P,P] \big|\leq \big|\mathbb A_{N+1}[P,P] \big|\quad  \mbox{ if }  {j_*}\in S\setminus P.
 \]

 Next, assume that ${j_*} \in P$ or ${j_*}\notin S$.  We claim  that if  we set
\eq
\label{def tilde m}
 \tilde m_j=
\begin{cases} m_j+k_j & \mbox{ if } j\in P,\\
  |m_j-k_j| & \mbox{ if } j\in S\setminus P,\\
  m_j & \mbox{ if } j\notin S,
 \end{cases}
\qe
then $(\bs m,\tilde {\bs m})\in \mathbb A_{N+1}[P,P]\setminus\mathbb A_{N}[P,P]$. This  would show in particular that
\[
\big|\mathbb A_{N}[P,P] \big| \leq  \big|\mathbb A_{N+1}[P,P] \big|\quad \mbox{ if either } {j_*}\in P \mbox { or } {j_*}\notin S,
 \]
 and thus complete the proof of (a). That $(\bs m,\tilde {\bs m})\in \mathbb A_{N+1}[P,P]\setminus\mathbb A_{N}[P,P]$ is by construction obvious provided one can show  $\tilde{\bs m}\geq \frak b(N+1)$.

 If ${j_*}\in P$,  then we have
 \[
 \tilde m_{j_*}=m_{j_*}+k_{j_*}=M+k_{j_*}> M
 \]
 and thus $\tilde{\bs m}\in \N^d\setminus \mathcal C_{M+1}$. As a consequence, there exists $m\geq(M+1)^d$ such that $\tilde {\bs m}=\frak b(m)$ and, since  $N+1\leq (M+1)^d$, we have shown $\tilde{\bs m}\geq \frak b(N+1)$.

If  ${j_*}\notin S$, we argue by contradiction and assume   $\tilde {\bs m}\leq \frak b(N)=\bs m$. We shall see this is not compatible with the graded lexicographic order. Indeed, since by construction $\tilde {\bs m}\neq \bs m$ and $\bs n\neq \bs m$  (because $\bs k\neq (0,\ldots,0)$ by assumption), we actually have $\tilde {\bs m}< \bs m$ and  $\bs n< \bs m$.  Because ${j_*}\notin S$ by assumption, we moreover have  $n_{j_*}=m_{j_*}=\tilde m_{j_*}=M$ and thus $\bs n,\bs m,  \tilde {\bs m}\in \mathcal C_{M+1}\setminus\mathcal C_M$. As a consequence, $ {\bs n}<_{{\rm lex}}\bs m$ and $\tilde {\bs m}<_{{\rm lex}}\bs m$ in the  lexicographic order. This means there exists $\gamma\in\{1,\ldots,d\}$ such that $n_i=m_i$ for every $i<\gamma$ and $n_\gamma<m_\gamma$, and equivalently $i\notin S$ when $i<\gamma$ and $\gamma\in P$. Similarly, there exists $\eta\in\{1,\ldots,d\}$ such that $\tilde m_i=m_i$ for every $i<\eta$ and $\tilde m_\eta<m_\eta$, and thus $i\notin S$ when $i<\eta$ and $\eta\in S\setminus P$. But this is impossible and thus  $\tilde {\bs m}\geq \frak b(N+1)$, which completes the proof of (a).

Part (b) is proved by following the exact same line of arguments; in this setting one should also check that if  $(\bs n,\bs m)\in\mathbb A_{N}[P,P;\alpha]$ then $m_\alpha\leq k_\alpha$  in order to show $(\bs m ,\tilde {\bs m})$ (with $\tilde{\bs m}$ defined in \eqref{def tilde m})  actually belongs to $\mathbb A_{N+1}[P,P;\alpha]$. Recalling  $\alpha\in S\setminus P$ by assumption this is clear, indeed $m_\alpha=|n_\alpha-k_\alpha|$ together with $n_\alpha\leq k_\alpha$ yield $m_\alpha=k_\alpha-n_\alpha\leq k_\alpha$.
\end{proof}

\begin{proof}[Proof of Lemma \ref{estimates M}] To prove (a), fix $P\subset S$ and assume $M>\max(k_1,\ldots,k_d)$. It follows from the definitions \eqref{def sets},\eqref{def A*} that
\eq
\label{A* easy}
\mathbb A_{M^d}^*[P]= \Bigg\{(\bs n,\bs m)\in\mathbb A_{M^d} \;\Bigg| \;
\begin{matrix*}[r]
n_j+k_j=m_j, \\
n_j-k_j=m_j, \\
n_j=m_j,\\
\end{matrix*}
\qquad
\begin{matrix*}[l]
n_j\neq 0, \\
n_j > k_j, \\
\\
\end{matrix*}
\qquad
\begin{matrix*}[l]
\mbox{ if } j\in P\\
\mbox{ if } j\in S\setminus P\\
\mbox{ if } j\notin S
\end{matrix*}
\Bigg\}.
\qe
Recall that  $\mathbb A_{M^d}=\mathcal C_M\times (\N^d\setminus\mathcal C_M)$
where $\mathcal C_M$ is defined in \eqref{def C_M}. Clearly, if we set
\eq
\label{def C_M[P_+]}
\mathcal C_M[P]=\Big\{\bs n\in\mathcal C_M : \mbox{ there exists $\bs m\in\N^d\setminus\mathcal C_M$}, \;(\bs n,\bs m)\in \mathbb A_{M^d}^*[P]\Big\}
\qe
then $(\bs n,\bs m)\mapsto \bs n$ is a bijection from $ \mathbb A_{M^d}^*[P]$ to $\mathcal C_M[P]$.

We  claim that if for any $p\in P$  we set
\eq
\label{def C_M[P_+] alpha}
\mathcal C_M^{(p)}[P]=\Bigg\{\bs n\in\mathcal C_M \; \Bigg|\;
\begin{matrix*}[r]
1\leq n_j\leq M-1\\
k_j< n_j\leq M-1\\
M-k_p\leq  n_p\leq M-1
\end{matrix*}
\;
\begin{matrix*}[l]
\mbox{ if } j\in P\\
\mbox{ if } j\in S\setminus P\\\\
\end{matrix*}\;
\Bigg\},
\qe
then we have
\eq
\label{C_M[P_+] identity}
\mathcal C_M[P]=\bigcup_{p\in P}\mathcal C_M^{(p)}[P].
\qe
Indeed, let $\bs n\in \mathcal C_M[P]$. By definition there exists $\bs m\in \N^d\setminus\mathcal C_M$ such that $(\bs n,\bs m)\in\mathbb A_{M^d}^*[P]$. This provides, see \eqref{A* easy}, that $1\leq n_j\leq M-1$ if $j\in P$, that  $k_j< n_j\leq M-1$ if $j\in S\setminus P$, and the existence of $p\in\{1,\ldots,d\}$ satisfying $m_p\geq M$. Since $\bs n\in\mathcal C_M$ then $n_p\leq M-1$ and thus $p\in P$ because otherwise $m_p\leq n_p$. Together with the equation $n_p+k_p=m_p$ this finally yields that  $M-k_p\leq n_p \leq M-1$, namely $\bs n\in \mathcal C_M^{(p)}[P]$ for some $p\in P$. As for the reverse inclusion, if $\bs n\in \mathcal C_M^{(\alpha)}[P]$ for some $\alpha\in P$ then  set
\[
m_j=
\begin{cases} n_j+k_j & \mbox{ if } j\in P,\\
 n_j-k_j & \mbox{ if } j\in S\setminus P,\\
 n_j & \mbox{ if } j\notin S,
 \end{cases}
\]
and observe that $\bs m\in \N^d\setminus\mathcal C_M$ since $m_p\geq M$ and $n_j-k_j\geq 0$ for every $j\in S\setminus P$. Thus, since clearly $(\bs n,\bs m)\in\mathbb A_{M^d}^*[P]$, we have shown  $\bs n\in \mathcal C_M[P]$ and \eqref{C_M[P_+] identity} is proved.

Next, since for every  distinct  $p_1,\ldots,p_q\in P$ the definition \eqref{def C_M[P_+] alpha} yields
\[
\big| \mathcal C_M^{(p_1)}[P]\cap \,\cdots\,\cap  \mathcal C_M^{(p_q)}[P] \big|= k_{p_1}\cdots k_{p_q} M^{d-q}+\cO(M^{d-q-1}),
\]
then (a) follows from  \eqref{C_M[P_+] identity} and the inclusion-exclusion principle.

We now turn to (b) and fix $\alpha\in S\setminus P$. Let $\mathcal C_M^{(d)}$ be  the $d$-dimensional discrete hypercube of length $M$ defined as in \eqref{def C_M}. We then set
\[
\mathbb A^{(d-1)}_{M^{d-1}}=\mathcal C_M^{(d-1)}\times (\N^{d-1}\setminus\mathcal C_M^{(d-1)})
\]
and introduce
\[
\mathbb A_{M^{d-1}}^*[P]^{(d-1)}=
\Bigg\{(\bs n,\bs m)\in\mathbb A^{(d-1)}_{M^{d-1}} \;\Bigg| \;
\begin{matrix*}[r]
n_j+k_j=m_j, \\
n_j-k_j=m_j, \\
n_j=m_j,\\
\end{matrix*}
\qquad
\begin{matrix*}[l]
n_j\neq 0, \\
n_j > k_j, \\
\\
\end{matrix*}
\qquad
\begin{matrix*}[l]
\mbox{ if } j\in P\\
\mbox{ if } j\in S\setminus (P\cup\{\alpha\})\\
\mbox{ if } j\notin S\setminus\{\alpha\}
\end{matrix*}
\Bigg\}.
\]
The statement (a) of the  lemma applied in dimension $d-1$ then provides
\eq
\label{lower dim estimates}
\big| \mathbb A_{M^{d-1}}^*[P]^{(d-1)}\big| = o(M^{d-1}).
\qe

Consider  the map $\frak p:\N^d\to\N^{d-1}$ defined by
\[
\frak p(n_1,\ldots,n_d)=(n_1,\ldots, n_{\alpha-1},n_{\alpha+1},\ldots,n_d).
\]
Let $(\bs n,\bs m)\in  \mathbb A_{M^d}^*[P;\alpha]$.  Since $\bs m\in\N^d\setminus \mathcal C_M^{(d)}$ and $m_\alpha\leq k_\alpha<M$, there exists $j\neq \alpha$ such that $m_j\geq M$.  It follows that  $(\frak p(\bs n),\frak p(\bs m))\in \mathcal C_M^{(d-1)}\times (\N^{d-1}\setminus \mathcal C_M^{(d-1)})$ and thus $(\frak p(\bs n),\frak p(\bs m))\in \mathbb A^{*}_{M^{d-1}}[P]^{(d-1)}$. As a consequence, we have the upper bound
\[
\big| A_{M^d}^*[P;\alpha]\big|\leq k_\alpha \big| (\frak p\times \frak p)\big(A_{M^d}^*[P;\alpha]\big)\big|\leq k_\alpha \big| A_{M^{d-1}}^*[P]^{(d-1)}\big|
\]
and thus (b) follows from \eqref{lower dim estimates}.

\end{proof}

\subsection{Covariance asymptotics: the general case}

We now extend Proposition~\ref{cov cheby prop} to the general setting of
measures satisfying the assumptions of Theorem \ref{th CLT general}. More
precisely, we prove the following.

\begin{proposition}
\label{general cov}
Let $\mu=\mu_1\otimes\cdots\otimes\mu_d$, where the $\mu_j$'s are Nevai-class
probability measures on $I$. Let $\bv x_1,\ldots, \bv x_N$ and $\bv
x_1^*,\ldots, \bv x_N^*$ be random
variables drawn from the multivariate OP Ensembles with respective reference
measures $\mu$ and $\mueqd$. Then, given any polynomial functions $P,Q$ on $\R^d$,
\eq
 \lim_{N\rightarrow\infty}\frac{1}{N^{1-1/d}}\,\Bigg|\,\mathbb C{\rm ov}\left[\,\sum_{i=1}^NP(\bv x_i),\sum_{j=1}^NQ(\bv x_i)\right] -\mathbb C{\rm ov}\left[\,\sum_{i=1}^NP(\bv x_i^*),\sum_{j=1}^NQ(\bv x_i^*)\right]\Bigg|=0.
\qe
\end{proposition}

For the proof of the proposition, we use a few ingredients from the Step 2 of the proof of Lemma \ref{convergence CD out diag} to which we refer the reader to.

\begin{proof}[Proof of Proposition \ref{general cov}]
By linearity, it is enough to prove the proposition with
$P(x)=x_1^{\alpha_1}\cdots x_d^{\alpha_d}$ and $Q(x)=x_1^{\beta_1}\cdots
x_d^{\beta_d}$ for any fixed $\alpha_1,\beta_1,\ldots,\alpha_d,\beta_d\in\N$.
Lemma \ref{cov representation} then provides
\begin{align}
\label{CG A}
 \mathbb C{\rm ov}\left[\sum_{i=1}^NP(\bv x_i),\sum_{j=1}^NQ(\bv x_i)\right] \nonumber
& = \sum_{n=0}^{N-1}\sum_{m=N}^\infty \langle x_1^{\alpha_1}\cdots x_d^{\alpha_d} \phi_n,\phi_m\rangle_{L^2(\mu)}\langle x_1^{\beta_1}\cdots x_d^{\beta_d}\phi_n,\phi_m\rangle_{L^2(\mu)}\nonumber\\
 & =  \sum_{(\bs n,\bs m)\in\mathbb A_N}\prod_{j=1}^d\langle x^{\alpha_j}\phi^{(j)}_{n_j},\phi^{(j)}_{m_j}\rangle_{L^2(\mu_j)}\langle x^{\beta_j}\phi^{(j)}_{n_j},\phi^{(j)}_{m_j}\rangle_{L^2(\mu_j)}
\end{align}
where we recall that
\[
\mathbb A_N=\big\{\big (\frak b(n),\frak b(m) \big):\; n\leq N-1,\; m\geq N\big\}\subset\N^d\times\N^d.
\]
In particular, by choosing $\mu=\mueqd$ in \eqref{CG A}, we obtain
\eq
\label{CG A cheby}
\mathbb C{\rm ov}\left[\sum_{i=1}^NP(\bv x_i^*),\sum_{j=1}^NQ(\bv x_i^*)\right]
= \sum_{(\bs n,\bs m)\in\mathbb A_N}
\prod_{j=1}^d\langle x^{\alpha_j}T_{n_j},T_{m_j}\rangle_{L^2(\mu_{eq})}\langle x^{\beta_j}T_{n_j},T_{m_j}\rangle_{L^2(\mu_{eq})}.
\qe
Thus, by combining \eqref{CG A} and \eqref{CG A cheby}, we see that, if we set for convenience
\begin{multline}
\label{Error ab}
E(\bs n,\bs m)=\Bigg| \prod_{j=1}^d\langle x^{\alpha_j}\phi^{(j)}_{n_j},\phi^{(j)}_{m_j}\rangle_{L^2(\mu_j)}\langle x^{\beta_j}\phi^{(j)}_{n_j},\phi^{(j)}_{m_j}\rangle_{L^2(\mu_j)}
\\
 -\prod_{j=1}^d\langle x^{\alpha_j}T_{n_j},T_{m_j}\rangle_{L^2(\mu_{eq})}\langle x^{\beta_j}T_{n_j},T_{m_j}\rangle_{L^2(\mu_{eq})}\Bigg|,
\end{multline}
then proving the proposition amounts to showing that
\eq
\label{todo cov approx}
\lim_{N\rightarrow\infty}\frac{1}{N^{1-1/d}}
\sum_{(\bs n,\bs m)\in\mathbb A_N}
E(\bs n,\bs m)=0.
\qe

Next, for every $j\in\{1,\ldots,d\}$, the three-term recurrence relation reads
\eq
x\phi_n^{(j)}=a_n^{(j)}\phi_{n+1}^{(j)}+b_n^{(j)}\phi_n^{(j)}+a_{n-1}^{(j)}\phi_{n-1}^{(j)},\qquad n\geq 0,
\qe
where we set  $a_{-1}^{(j)}=0$. As in Step 2 of the proof of Lemma
\ref{convergence CD out diag}, we complete  the sequences of recurrence
coefficients $(a_n^{(j)})_{n\in\N}$ and $(b_n^{(j)})_{n\in\N}$ introduced into
sequences $(a_n^{(j)})_{n\in\Z}$, $(b_n^{(j)})_{n\in\Z}$, where we set
$a_n^{(j)}=b_n^{(j)}=0$ for every $n<0$. We thus obtain the representations
\eq
\label{nice path 2}
\langle x^{\alpha}\phi^{(j)}_{n_j},\phi^{(j)}_{m_j}\rangle_{L^2(\mu_j)}=\bv 1_{|n_j-m_j|\leq \alpha} \sum_{\gamma:(0,n_j-m_j)\rightarrow (\alpha,0)} \prod_{e\in\gamma}\omega(e)_{\{(a^{(j)}_{n+m_j}),\,(b^{(j)}_{n+m_j})\}},
\qe
and
\eq
\label{nice path Cheby 2}
\langle x^{\alpha}T_{n_j},T_{m_j}\rangle_{L^2(\mu_{eq})}=\bv 1_{|n_j-m_j|\leq \alpha} \sum_{\gamma:(0,n_j-m_j)\rightarrow (\alpha,0)} \prod_{e\in\gamma}\omega(e)_{\{(a^*_{n+m_j}),\,(b^*_{n+m_j})\}},
\qe
\rev{where we recall that $w(e)$ was introduced in \eqref{weights paths}.}
Since the measures $\mu_j$ are Nevai-class by assumption, we have
$a_n^{(j)}-a_n^*\to 0$ and $b_n^{(j)}-b_n^*\to 0$ as $n\to\infty$ for every
$j\in\{1,\ldots,d\}$. Notice that for every $n_j$, the right-hand side of
\eqref{nice path 2} is a polynomial function of
$a_{m_j-\alpha}^{(j)},b_{m_j-\alpha}^{(j)},\ldots,a_{m_j+\alpha}^{(j)},b_{m_j+\alpha}^{(j)}$
and does not depend on any other recurrence coefficients. Thus, we obtain for every fixed $\alpha\in\N$,
\eq
\label{key rec coef conv}
\sup_{n_j\in\N}\left|\langle x^{\alpha}\phi^{(j)}_{n_j},\phi^{(j)}_{m_j}\rangle_{L^2(\mu_j)} - \langle x^{\alpha}T_{n_j},T_{m_j}\rangle_{L^2(\mu_{eq})}\right|\xrightarrow[m_j\to\infty]{} 0.
\qe

Moreover, we see from \eqref{Error ab}, \eqref{nice path 2} and \eqref{nice path Cheby 2} that $E(\bs n,\bs m)=0$ except when $|n_j-m_j|\leq \min(\alpha_j,\beta_j)$ for every $j\in\{1,\ldots,d\}$. We then split the set of contributing indices into two subsets,
\begin{align*}
\mathbb A_N^{*}& =\Big\{ (\bs n,\bs m)\in\mathbb A_N :\;|n_j-m_j|\leq \min(\alpha_j,\beta_j)\Big\} \cap \Big \{m_j\geq N^{1/2d}\quad \mbox{ for every } j\,\Big\},\\
\mathbb A_N^{0}& =\Big\{ (\bs n,\bs m)\in\mathbb A_N :\;|n_j-m_j|\leq \min(\alpha_j,\beta_j)\Big\}\cap \Big\{ m_j< N^{1/2d}\quad \mbox{ for at least one } j\, \Big\}.
\end{align*}
It then follows from \eqref{key rec coef conv} that
\eq
\label{control E 1}
\lim_{N\to\infty}\sup_{(\bs n,\bs m)\in\mathbb A_N^*}\big|E(\bs n,\bs m)\big| =0
\qe
and that there exists $C>0$ independent on $N$ satisfying
\eq
\label{control E 2}
\sup_{(\bs n,\bs m)\in\mathbb A_N^0}\big|E(\bs n,\bs m)\big|\leq\sup_{(\bs n,\bs m)\in\mathbb A_N}\big|E(\bs n,\bs m)\big| \leq C.
\qe

Next, we write
\begin{align}
& \frac{1}{N^{1-1/d}}
\sum_{(\bs n,\bs m)\in\mathbb A_N}
E(\bs n,\bs m)\nonumber \\
& =\frac{1}{N^{1-1/d}}
\sum_{(\bs n,\bs m)\in\mathbb A_N^*}
E(\bs n,\bs m)+\frac{1}{N^{1-1/d}}
\sum_{(\bs n,\bs m)\in\mathbb A_N^0}
E(\bs n,\bs m)\nonumber\\
&\leq \frac{\big| \mathbb A_N^{*}\big| }{N^{1-1/d}}\sup_{(\bs n,\bs m)\in\mathbb A_N^*}\big|E(\bs n,\bs m)\big|
+
 \frac{\big| \mathbb A_N^{0}\big| }{N^{1-1/d}}\sup_{(\bs n,\bs m)\in\mathbb A_N^0}\big|E(\bs n,\bs m)\big|.
\end{align}
and claim that we have
\eq
\label{A* asympt}
\limsup_{N\to\infty}\frac{\big| \mathbb A_N^{*}\big| }{N^{1-1/d}}<\infty,
\qe
and,  moreover,
\eq
\label{A0 asympt}
\lim_{N\to\infty}\frac{\big| \mathbb A_N^{0}\big| }{N^{1-1/d}}=0.
\qe
Together with \eqref{control E 1}--\eqref{control E 2}, this would prove \eqref{todo cov approx} and thus the proposition.

We finally prove \eqref{A* asympt} and \eqref{A0 asympt} in order to complete the proof of the proposition.
Let us set $\kappa_j=\max(\alpha_j,\beta_j)$ for convenience. Clearly,
\begin{multline}
\label{UB AN* start}
\big|\mathbb A_N^{*}\big| = \Big|\bigcup_{\bs n\in\N^d}\Big\{\bs m\in \N^d :\;  (\bs n,\bs m)\in\mathbb A_N^{*} \Big\}\Big|\\
\leq
\max_{\bs n\in \N^d}\Big|\Big\{\bs m\in \N^d :\;  (\bs n,\bs m)\in\mathbb A_N^{*} \Big\}\Big| \times \Big|\Big\{\bs n\in \N^d :\;   (\bs n,\bs m)\in\mathbb A_N^{*}  \mbox{ for some } \bs m\in\N^d\Big\}\Big|.
\end{multline}
First, since  $|n_j-m_j|\leq \kappa_j$ for every $j$ as soon as  $(\bs n,\bs m)\in \mathbb A_N^*$, we have the upper bound
\eq
\max_{\bs n\in \N^d}\Big|\Big\{\bs m\in \N^d :\;  (\bs n,\bs m)\in\mathbb A_N^{*} \Big\}\Big|\leq \prod_{j=1}^d (2\kappa_j+1).
\qe
Next,  set  $M=\lfloor N^{1/d}\rfloor$  so that $M^d\leq N< (M+1)^d$. If  $(\bs n,\bs m)\in \mathbb A_N^*$, then it satisfies  $\bs n\in \mathcal C_{M+1}$ and $\bs m\in\N^d\setminus\mathcal C_{M}$, where $\mathcal C_M$ has been introduced in \eqref{def C_M}. Namely, it holds that $0\leq n_j\leq M$ for every $j$ and there exists $j_0$ such that $m_{j_0}\geq M$. Together with $|n_{j_0}-m_{j_0}|\leq \kappa_{j_0}$, this yields $M-\kappa_{j_0}\leq n_{j_0}\leq M$ and thus provides the upper bound
\eq
\label{UB AN* final}
 \Big|\Big\{\bs n\in \N^d :\;   (\bs n,\bs m)\in\mathbb A_N^{*}  \mbox{ for some } \bs m\in\N^d\Big\}\Big|
\leq (\max_{{j_0}=1}^d\kappa_{j_0}+1) (M+1)^{d-1}.
\qe
By combining \eqref{UB AN* start}--\eqref{UB AN* final}, we have proved \eqref{A* asympt}.  The proof of  \eqref{A0 asympt} is similar. More precisely, the only difference is that if $(\bs n,\bs m)\in\mathbb A_N^{(0)}$, then there exists $j_1$ such that $m_{j_1}<\sqrt N^{\frac 1d}<\sqrt{M+1}$. Notice that necessarily $j_1\neq j_0$. Using moreover that $|n_{j_1}-m_{j_1}|\leq \kappa_{j_1}$, we  obtain the upper bound
\[
 \Big|\Big\{\bs n\in \N^d :\;   (\bs n,\bs m)\in\mathbb A_N^{(0)}  \mbox{ for some } \bs m\in\N^d\Big\}\Big|
\leq (\max_{\gfrac{j_0=1}{j_0\neq j_1}}^d\kappa_{j_0}+1) (\kappa_{j_1}+\sqrt{M+1}\,) (M+1)^{d-2}
\]
in place of \eqref{UB AN* final}, and  \eqref{A0 asympt} follows.
\end{proof}

\subsection{Extension to $\mathscr C^1$ functions and conclusion}

We consider a reference measure $\mu$ satisfying the assumptions of Theorem~\ref{th
  CLT general} and let $\bx_1,\dots,\bx_N$ be the associated multivariate OP Ensemble. For any $d$-multivariate polynomial $P$, we can write $P=\sum_{\bs k\in\N^d}
\hat P(\bs k) T_{\bs k}$, where the latter sum is finite. As a consequence of Propositions \ref{cov cheby prop} and \ref{general cov}, we then obtain
\begin{align}
\label{prop 33 poly}
&\lim_{N\to\infty} \frac{1}{N^{1-1/d}}\Var\left[\, \sum_{i=1}^{N}P(\bv x_i)\right]\nonumber \\
& = \sum_{\bs k,\bs \ell\in\N^d}\hat P(\bs k)\hat P(\bs \ell) \lim_{N\to\infty} \frac{1}{N^{1-1/d}}\ \Cov\left[\, \sum_{i=1}^{N}T_{\bs k}(\bv x_i),\sum_{i=1}^{N}T_{\bs \ell}(\bv x_i)\right]\nonumber\\
& = \frac12\sum_{\bs k=(k_1,\ldots,k_d)\in\N^d}(k_1+\cdots+k_d)\hat P(\bs k)^2=\sigma_P^2.
\end{align}
Therefore, we have proven Proposition \ref{variance general CLT easy} provided we restrict the test functions to polynomials. We finally extend this result to $\mathscr C^1$ test functions, and thus complete the proof of this proposition, by means of a density argument.

First, a standard computation yields
\eq
\label{variance represent}
\Var\left[\, \sum_{i=1}^{N}f(\bv x_i)\right]\\
= \frac{1}{2}\iint (f(x)-f(y))^2K_{N}(x,y)^2\mu(\d x)\mu(\d y).
\qe
This indeed follows from \eqref{def k cor}--\eqref{def cor DPP} with $k=1,2$,
and that $K_N(x,y)$ is a symmetric reproducing kernel.

Now, for any $f\in\mathscr C^1(I^d,\R)$, we set
\eq
\| f\|_{{\rm Lip}}=\sup_{x\in I^d}\|\nabla f(x) \|,
\qe
so that $|f(x)-f(y)|\leq \| f\|_{{\rm Lip}} \|x-y\|$ for every $x\neq y$. If we
consider the monomials defined by
\eq
\label{ej monom}
e_j(x_1,\ldots,x_d)=x_j,
\qe
then formula \eqref{variance represent} yields
\begin{align*}
\Var\left[\ \sum_{i=1}^N f(\bv x_i)\right] & = \frac12\iint (f(x)-f(y))^2K_N(x,y)^2\mu(\d x)\mu(\d y)\\
& \leq   \| f\|_{{\rm Lip}}^2\sum_{j=1}^d\frac{1}2\iint (e_j(x)-e_j(y))^2K_N(x,y)^2\mu(\d x)\mu(\d y)\\
& =   \| f\|_{{\rm Lip}}^2\sum_{j=1}^d\Var\left[\ \sum_{i=1}^N e_j(\bv x_i)\right]
\end{align*}
and, as a consequence of \eqref{prop 33 poly},
\eq
\limsup_{N\to\infty}\frac{1}{N^{1-1/d}}\Var\left[\ \sum_{i=1}^N f(\bv x_i)
\right]\leq C  \| f\|_{{\rm Lip}}^2,\qquad C=\sum_{j=1}^d \sigma_{e_j}^2.
\qe
Proposition \ref{bound limit var} also provides the upper bound
\eq
\sigma_f^2\leq  \frac1{2}\| f\|_{{\rm Lip}}^2.
\qe
Next, Theorem 5 of \cite{Pee09} yields the existence of a sequence of multivariate polynomials $(P_\epsilon)_{\epsilon>0}$ such that $\|P_\epsilon-f\|_{{\rm Lip}} \leq\epsilon$, and hence
\eq
\label{approx var P eps}
\limsup_{N\to\infty}\frac{1}{N^{1-1/d}}\Var\left[\ \sum_{i=1}^N f(\bv x_i)-\sum_{i=1}^N P_\epsilon(\bv x_i)
\right]\leq C\epsilon^2,\qquad \mbox{ and }\quad\sigma_{f-P_\epsilon}^2\leq \frac{\epsilon^2}{2}.
\qe

Since $(X,Y)\mapsto \Cov(X,Y)$ is a symmetric positive bilinear form, it satisfies the Cauchy-Schwartz inequality, and thus the triangle inequality $\Var(X+Y)^{1/2}\leq \Var(X)^{1/2}+\Var(Y)^{1/2}$  holds true, which in turn yields the inequality
\eq
\label{variance distance}
\left|\Var(X)^{1/2}-\Var(Y)^{1/2}\right|\leq \Var(X-Y)^{1/2}.
\qe
For the same reason, the limiting variance satisfies $\left|\sigma_f-\sigma_g\right|\leq \sigma_{f-g}$. As a consequence, by taking $X=\sum f(\bv x_i)$ and $Y=\sum P_\epsilon(\bv x_i)$ in \eqref{variance distance}, and using these two inequalities together with \eqref{prop 33 poly} and \eqref{approx var P eps}, we obtain by letting $N\to\infty$ and then $\epsilon\to0$ that
\begin{align*}
\lim_{N\to\infty}\frac{1}{N^{1-1/d}}\Var\left[\ \sum_{i=1}^N f(\bv x_i)\right]& = \lim_{\epsilon\to 0}\lim_{N\to\infty}\frac{1}{N^{1-1/d}}\Var\left[\ \sum_{i=1}^N P_\epsilon(\bv x_i)\right]\\
& = \lim_{\epsilon \to 0}\sigma^2_{P_\epsilon}\\
& = \sigma_{f}^2
\end{align*}
and the proof of Proposition \ref{variance general CLT easy} is therefore complete.

\section{Monte Carlo with DPPs: proof of Theorem \ref{DPPMC Th1}}
\label{proof th2 section}
The aim of this section is to prove the following variance decay.

\begin{proposition}
\label{key variance est}
Assume $\mu(\d x)=\omega(x)\d x$ with $\omega$ positive and
$\mathscr C^1$ on $(-1,1)^d$. Assume further that $\mu$ satisfies
Assumption~\ref{a:regularity assumption}. For every $f\in\mathscr C$, we have
\[
\lim_{N\to\infty} \frac{1}{N^{1-1/d}}\mathbb V{\rm ar}\left[\,\sum_{i=1}^N
  \frac{N f(\bx_i)}{K_N(\bx_i,\bx_i)}- \sum_{i=1}^N \frac{\omega(x)f(\bx_i)}{\omega_{eq}^{\otimes d}(\bx_i)} \right]=0.
\]
\end{proposition}
Before proving Proposition~\ref{key variance est}, we argue that it implies
Theorem~\ref{DPPMC Th1}. Indeed, \eqref{variance distance} then implies that
$$ \lim_{N\to\infty} \frac{1}{N^{1-1/d}}\mathbb V{\rm ar}\left[\,\sum_{i=1}^N
  \frac{N f(\bx_i)}{K_N(\bx_i,\bx_i)}\right] = \lim_{N\to\infty} \frac{1}{N^{1-1/d}}\mathbb V{\rm ar}\left[\,\sum_{i=1}^N \frac{\omega(x)f(\bx_i)}{\omega_{eq}^{\otimes d}(\bx_i)} \right]=\Omega_{f,\omega}^2\,,$$
the last equality following from Theorem~\ref{th CLT general}. Now
Theorem~\ref{Sosh} applies with $f_N(x) = Nf(x)/K_N(x,x)$ to yield
Theorem~\ref{DPPMC Th1}.

From now on, we fix $f\in\mathscr C$. It is thus a $\mathscr C^1$ function and
there exists $\epsilon>0$ so that $\mathrm{Supp}(f)\subset I_\epsilon^d$. If we
set for convenience
$$ E_N(x) =  \frac{N}{K_N(x,x)}- \sum_{i=1}^N
\frac{\omega(x)}{\omega_{eq}^{\otimes d}(x)}, \qquad x\in I^d,  $$
then Theorem~\ref{Totik asymp unif} yields $\|fE_N\|_\infty = \sup_{I^d}\vert
fE_N \vert\to0$ as $N\to\infty$.

In order to prove Proposition~\ref{key variance est}, we start with the formula coming from \eqref{variance represent},
\[
\mathbb V{\rm ar}\left[\,\sum_{i=1}^N   f(\bv x_i) E_N(\bv x_i)\right]=\frac12\iint \big(f E_N(x)-f E_N(y)\big)^2K_N(x,y)^2\mu(\d x)\mu(\d y)
\]
and split the integral in several terms that we shall analyse separately.

\subsection{The off-diagonal contribution}
Given any $\delta>0$, we first consider  the contribution
\eq
\label{offdiag}
\frac12\iint_{\|x-y\|>\delta} \big(fE_N(x)-fE_N(y)\big)^2K_N(x,y)^2\mu(\d x)\mu(\d y).
\qe
By rough estimates, we obtain
 \begin{align*}
 \eqref{offdiag} & \leq \| f E_N\|_\infty^2\iint_{\|x-y\|>\delta}K_N(x,y)^2\mu(\d x)\mu(\d y)\\
 & \leq  \frac{1}{\delta^2}\| f E_N\|_\infty^2\sum_{j=1}^d\iint (x_j-y_j)^2K_N(x,y)^2\mu(\d x)\mu(\d y)\\
 & \leq  \frac{2}{\delta^2}\| f E_N\|_\infty^2\sum_{j=1}^d  \mathbb V{\rm ar}\left[\,\sum_{i=1}^N  e_j(\bv x_i)\right],
 \end{align*}
where the monomials $e_j$ were defined in \eqref{ej monom}. As a consequence,
using Proposition~\ref{variance general CLT easy} and that $\| f E_N\|_\infty\to
0$ as $N\to\infty$, we get
\eq
\lim_{N\to\infty} \frac{1}{N^{1-1/d}}\iint_{\|x-y\|>\delta} \big(fE_N(x)-fE_N(y)\big)^2K_N(x,y)^2\mu(\d x)\mu(\d y)=0
\label{e:offDiagonal}
\qe
for every $\delta>0$.

\subsection{The diagonal contribution}

By \eqref{e:offDiagonal}, it is sufficient to show
\eq
\label{diag obj}
\limsup_{\delta\to 0}\limsup_{N\to\infty} \frac{1}{N^{1-1/d}}\iint_{\|x-y\|\leq\delta} \big(fE_N(x)-fE_N(y)\big)^2K_N(x,y)^2\mu(\d x)\mu(\d y)=0
\qe
in order to complete the proof of Proposition~\ref{key variance est}.

Set for convenience
\eq
\label{hey Dj !}
\mathscr D_g(x,y)= \frac{g(x)-g(y)}{\|x-y\|}, \qquad x,y\in I^d,
\qe
so that, $|\mathscr D_g(x,y)|\leq \|g\|_{\rm{Lip}}$. For every $\delta>0$ small enough, we have for any $x,y$ satisfying $\|x-y\|\leq \delta$,
\begin{align*}
\mathscr D_{f E_N}(x,y)^2 & =(\mathscr D_{f}(x,y) E_N(x)+\mathscr D_{E_N}(x,y) f(y))^2\\
& \leq 2\mathscr D_{f}(x,y)^2 E_N(x)^2+2\mathscr D_{E_N}(x,y)^2 f(y)^2\\
& \leq 2\|f\|_{\rm{Lip}}^2 \|\bs 1_{I^d_{\epsilon/2}}E_N\|_\infty^2+2\|f\|_\infty^2 \mathscr D_{E_N}(x,y)^2\ \bs 1_{I^d_{\epsilon/2}\times I^d_{\epsilon/2}}(x,y).
\end{align*}
Indeed, notice that if $x\in I^d_\epsilon$ or $y\in I^d_\epsilon$, then $x,y\in
I^d_{\epsilon/2}$ for every $\delta>0$ small enough. Since $f$ is by assumption
supported on $I^d_\epsilon$, we know that $\mathscr D_{f E_N}(x,y)$ vanishes
outside of $I^d_{\epsilon/2}\times I^d_{\epsilon/2}$. This is the reason for the
presence of $\IND_{I^d_{\epsilon/2}}$ in the last inequality.

With the notation $Q_N$ introduced in \eqref{e:defQN}, we thus obtain
\begin{align}
\label{ineq diag A}
&\frac1{2N^{1-1/d}}\iint_{\|x-y\|\leq \delta} \big(fE_N(x)-fE_N(y)\big)^2K_N(x,y)^2\mu(\d x)\mu(\d y)\nonumber\\
 & = \frac12\iint_{\|x-y\|\leq \delta} \mathscr D_{f E_N}(x,y)^2Q_N(\d x,\d y)\nonumber\\
 & \leq  \; \|f\|_{\rm{Lip}}^2 \|\bs 1_{I^d_{\epsilon/2}}E_N\|_\infty^2\iint_{\|x-y\|\leq \delta} Q_N(\d x,\d y)\\
 & \qquad +\|f\|_\infty^2 \iint_{I^d_{\epsilon/2}\times I^d_{\epsilon/2},\ \|x-y\|\leq \delta} \mathscr D_{E_N}(x,y)^2 Q_N(\d x,\d y)\nonumber.
\end{align}
Moreover,  because $\omega$ is $\mathscr C^1$ on $I_{\epsilon/2}^d$ by assumption, and because $\omega_{eq}^{\otimes d}$ is also $\mathscr C^1$ and positive there, one similarly has, for every $x,y\in I_{\epsilon/2}^d$,
\eq
\label{ineq diag B}
\mathscr D_{E_N}(x,y)^2\leq 2 \mathscr D_{N}(x,y)^2 +2 \mathscr D_{\omega/\omega_{eq}^{\otimes d}}(x,y)^2\leq 2 \mathscr D_{N}(x,y)^2 +2\|\omega/\omega_{eq}^{\otimes d}\|_{\mathrm{Lip}},
\qe
where $\mathscr D_N$ is defined in \eqref{e:defDN}.

Next, we have for every $C>0$,
\begin{align}
& \iint_{I^d_{\epsilon/2}\times I^d_{\epsilon/2},\ \|x-y\|\leq \delta} \mathscr D_{N}(x,y)^2 Q_N(\d x,\d y)\nonumber\\
& \leq C^2\iint_{\|x-y\|\leq \delta} Q_N(\d x,\d y)\nonumber\\
& \qquad + \iint_{I^d_{\epsilon/2}\times I^d_{\epsilon/2},\ \|x-y\|\leq \delta} \bv 1_{|\mathscr D_{N}(x,y)|> C}\ \mathscr D_{N}(x,y)^2 Q_N(\d x,\d y).\label{alligator427}
\end{align}

We now make use of the following lemma, the proof of which is deferred to Section~\ref{s:lemmaProof}.
 \begin{lemma}
  \label{diagonal est}
   \eq
 \lim_{\delta\to0}\limsup_{N\to\infty} \iint_{\|x-y\|\leq \delta} Q_N(\d x,\d y)=0.
 \qe
\end{lemma}

As a consequence, \eqref{ineq diag A}, \eqref{ineq diag B}, and
\eqref{alligator427} together yield, for every $C>0$,
\begin{multline}
\limsup_{\delta\to 0}\limsup_{N\to\infty}\frac1{2N^{1-1/d}}\iint_{\|x-y\|\leq \delta} \big(fE_N(x)-fE_N(y)\big)^2K_N(x,y)^2\mu(\d x)\mu(\d y)\\
\leq 2\|f\|_\infty^2 \limsup_{\delta\to
  0}\limsup_{N\to\infty}\iint_{I^d_{\epsilon/2}\times I^d_{\epsilon/2},\
  \|x-y\|\leq \delta} \mathscr D_{N}(x,y)^2 \bv 1_{|\mathscr D_{N}(x,y)|> C} \,Q_N(\d x,\d y).
\end{multline}

Assumption~\ref{a:regularity assumption} allows us to conclude the proof of
Proposition~\ref{key variance est}, up to the proof of Lemma~\ref{diagonal est}.

\subsection{Proof of Lemma \ref{diagonal est}}
\label{s:lemmaProof}

 \begin{proof}
 First,
 \eq
 \iint_{\|x-y\|\leq \delta}  Q_N(\d x,\d y) \leq \frac1{N^{1-1/d}}\sum_{j=1}^d
 \iint_{|x_j-y_j|\leq \delta}  (x_j-y_j)^2K_N(x,y)^2\mu(\d x)\mu(\d y).
\label{e:toolLemma62Bis}
 \qe
We fix $j\in\{1,\dots,d\}$ and use the notation of the proof of Lemma~\ref{l:separability}. It comes
 \[
K_N(x,y)= \sum_{[\bs k]\in \Gamma_N/\sim}K^{(j)}_{N_j([\bs
  k])+1}(x_j,y_j)\prod_{\alpha\neq j}\phi_{k_\alpha}(x_\alpha)\phi_{k_\alpha}(y_\alpha) .
  \]
  Squaring, integrating and using the orthonormality relations,
  \begin{align}
&  \iint_{|x_j-y_j|\leq \delta}  (x_j-y_j)^2K_N(x,y)^2\mu(\d x)\mu(\d y)\nonumber\\
  &  = \sum_{[\bs k],[\bs \ell] \in \Gamma_N/\sim} \bs 1_{\sigma(\bs k)=\sigma(\bs \ell)}\nonumber\\
  & \qquad  \times\iint_{|x_j-y_j|\leq\delta} (x_j-y_j)^2K^{(j)}_{N_j([\bs k])+1}(x_j,y_j)K^{(j)}_{N_j([\bs \ell])+1}(x_j,y_j)\mu_j(\d x_j)\mu_j(\d y_j)\nonumber\\
    &   = \sum_{[\bs k]\in \Gamma_N/\sim}\iint_{|x_j-y_j|\leq\delta} (x_j-y_j)^2K^{(j)}_{N_j([\bs k])+1}(x_j,y_j)^2\mu_j(\d x_j)\mu_j(\d y_j).
\label{e:toolLemma62}
  \end{align}
Recall $M=\lfloor N^{1/d}\rfloor$ and $\mathcal{C}_M\subset\Gamma_N\subset\mathcal{C}_{M+1}$. By definition of
$\frak b$ we have, for every $1\leq m\leq M-2$,
$$ \vert \{ [\bs k]\in \Gamma_N/\sim \;:\; N_j([\bs k]) = m \}\vert \leq dM^{d-2}.$$
Notice also that \eqref{LN CD} yields
$$
\iint (x_j-y_j)^2 K^{(j)}_{m}(x_j,y_j)^2\mu_j(\d x_j)\mu_j(\d y_j) = 2a_m^2,
$$
which is bounded for every $m$ since $a_m\rightarrow 1/2$ by assumption. Now
\begin{eqnarray}
\eqref{e:toolLemma62} &=& \left[ \sum_{\gfrac{[\bs k]\in
                             \Gamma_N/\sim}{N_j([\bs k])< \sqrt{M}}} + \sum_{\gfrac{[\bs k]\in
                             \Gamma_N/\sim}{N_j([\bs k])\geq \sqrt{M}}}  \right] \iint_{|x_j-y_j|\leq\delta} (x_j-y_j)^2K^{(j)}_{N_j([\bs k])+1}(x_j,y_j)^2\mu_j(\d x_j)\mu_j(\d y_j) \nonumber\\
&\leq& \cO(M^{d-2+1/2}) \nonumber\\
&&\quad+ M^{d-1}\max_{\sqrt{M}\leq m \leq M}
       \iint_{|x_j-y_j|\leq\delta} (x_j-y_j)^2K^{(j)}_{m+1}(x_j,y_j)^2\mu_j(\d
       x_j)\mu_j(\d y_j).
\label{e:toolLemma62Ter}
\end{eqnarray}
Moreover, Lemma~\ref{convergence CD out diag} yields
\[
\max_{\sqrt{M}\leq m \leq M} \iint_{|x_j-y_j|\leq\delta} (x_j-y_j)^2K^{(j)}_{m+1}(x_j,y_j)^2\mu_j(\d
 x_j)\mu_j(\d y_j) \rightarrow \iint_{|x-y|\leq\delta}L(x,y)\d x\d y
\]
as $M\rightarrow \infty$. Combined with \eqref{e:toolLemma62Bis}--\eqref{e:toolLemma62Ter}, we obtain
$$
\limsup_{N\rightarrow \infty}\iint_{\|x-y\|\leq \delta}  Q_N(\d x,\d y) \leq d\iint_{|x-y|\leq\delta}L(x,y)\d x\d y.
$$
Finally, since $L$ is integrable, the lemma follows by letting $\delta\to0$.
\end{proof}
The proof of Proposition~\ref{key variance est} is therefore complete.

\newpage
\section{Additional figures for the experiments of Section~\ref{s:experiments}}
\label{a:figures}

\rev{In Figure~\ref{f:resultsIS}, we display the results of the linear regression in Section~\ref{s:ass}. In Figure~\ref{f:resultsAss}, we display those of the linear regression in Section~\ref{s:ass}.}

\begin{figure}[h]
\subfigure[$d=1$]{
\includegraphics[width=\twofig]{\figdir/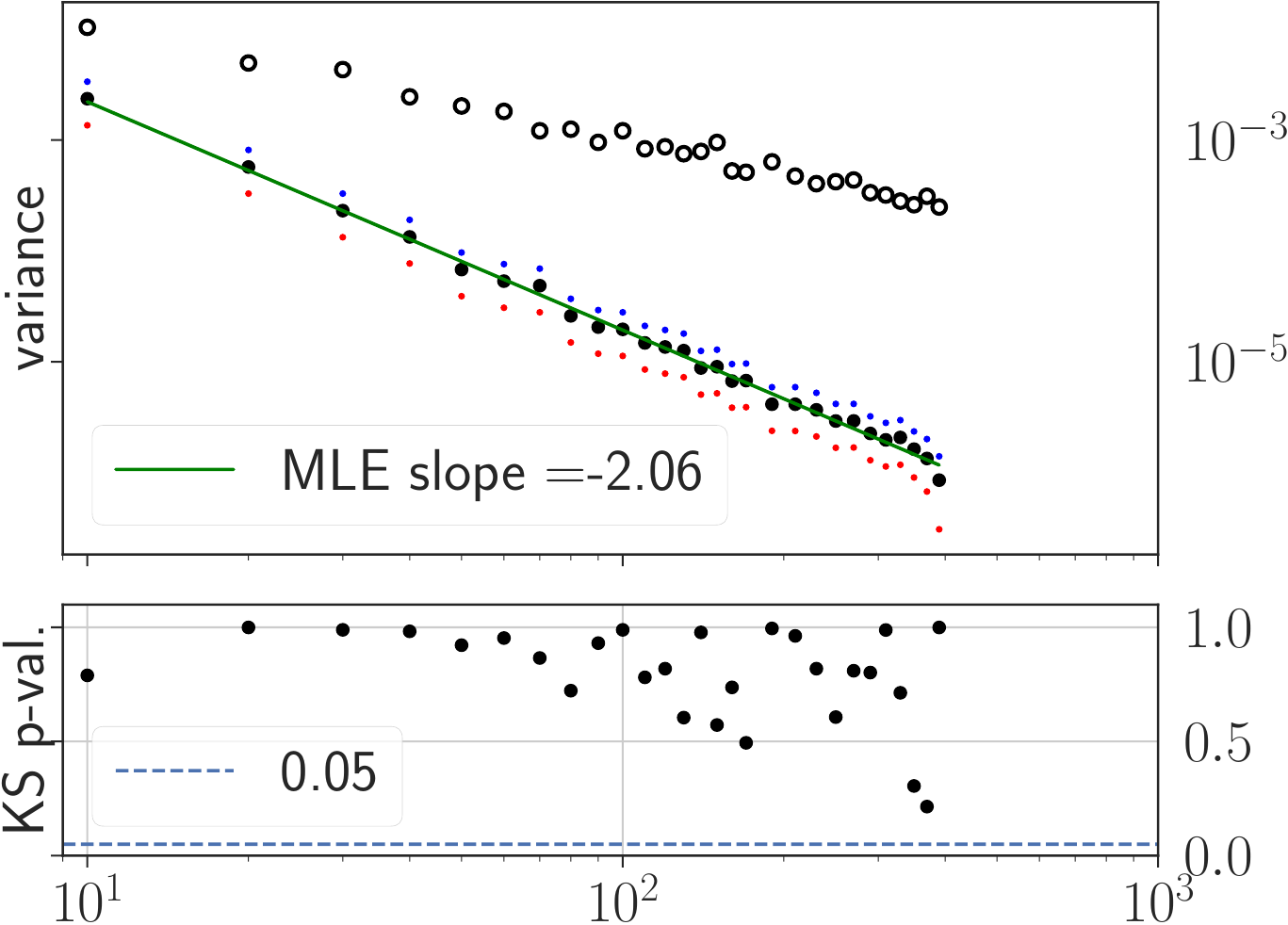}
\label{f:resultsIS1D}
}
\subfigure[$d=2$]{
\includegraphics[width=\twofig]{\figdir/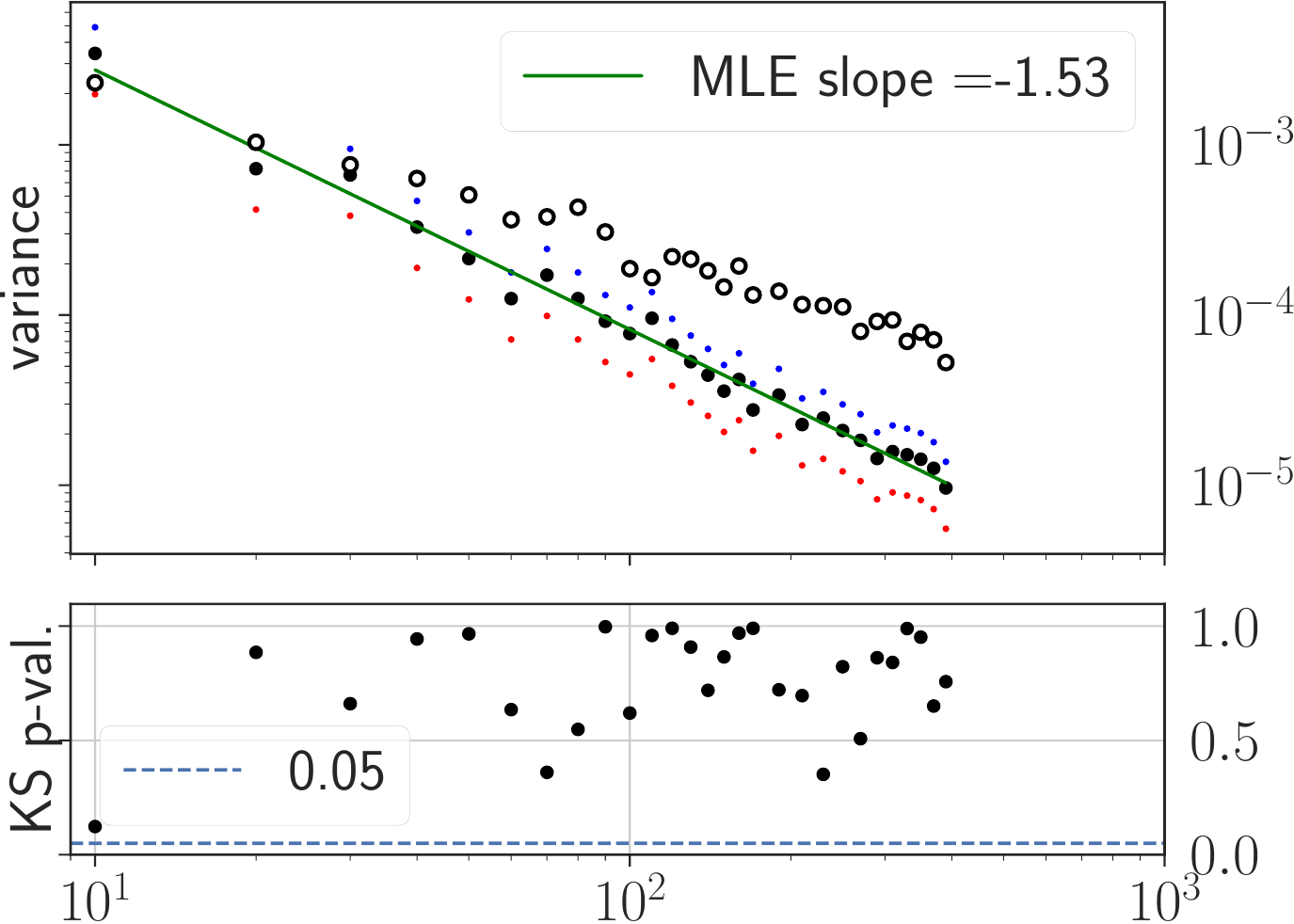}
\label{f:resultsIS2D}
}\\
\centering
\subfigure[$d=3$]{
\includegraphics[width=\twofig]{\figdir/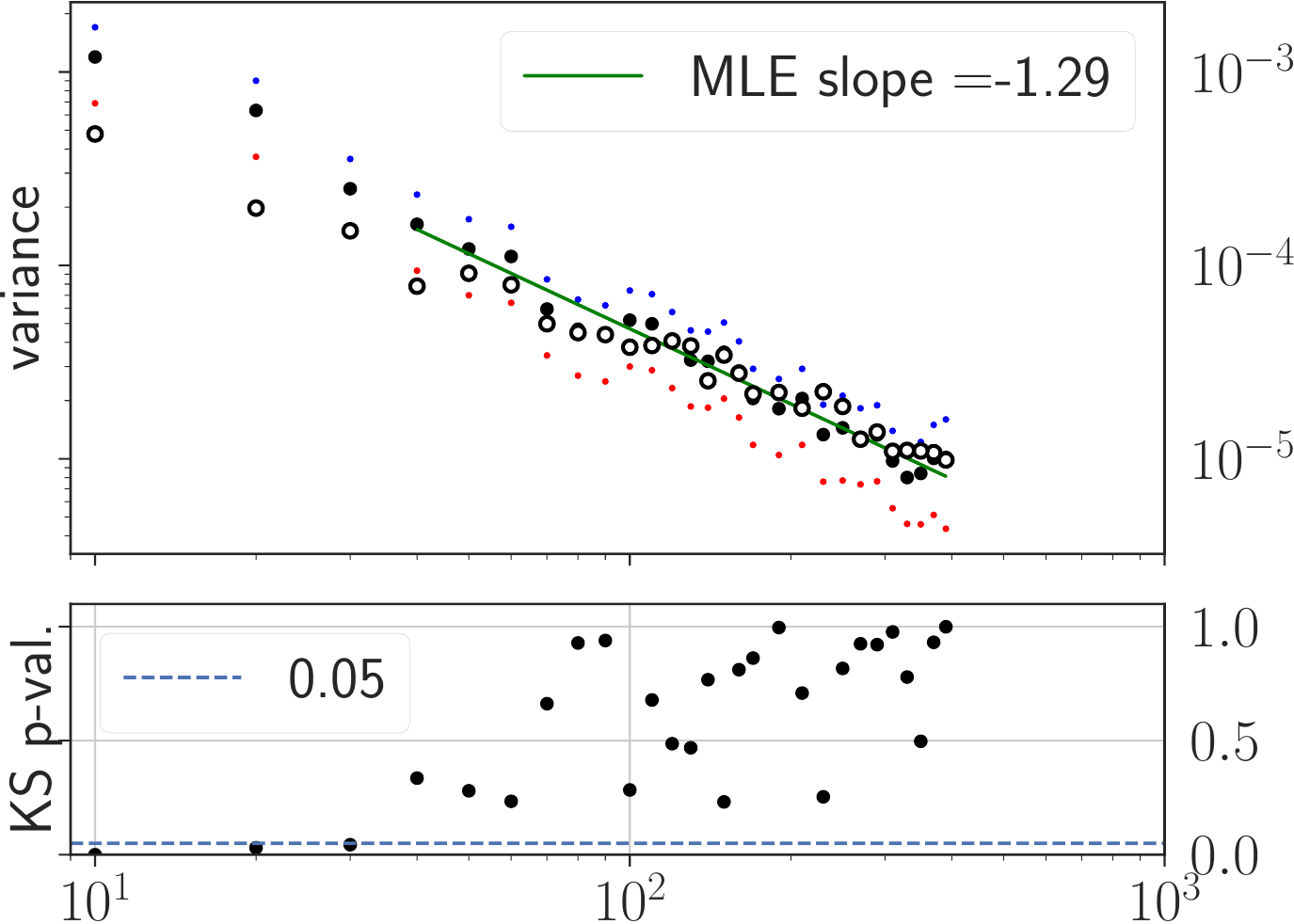}
\label{f:resultsIS3D}
}
\caption{Summary of the importance sampling results.}
\label{f:resultsIS}
\end{figure}

\begin{figure}
\subfigure[$d=1$]{
\includegraphics[width=\twofig]{\figdir/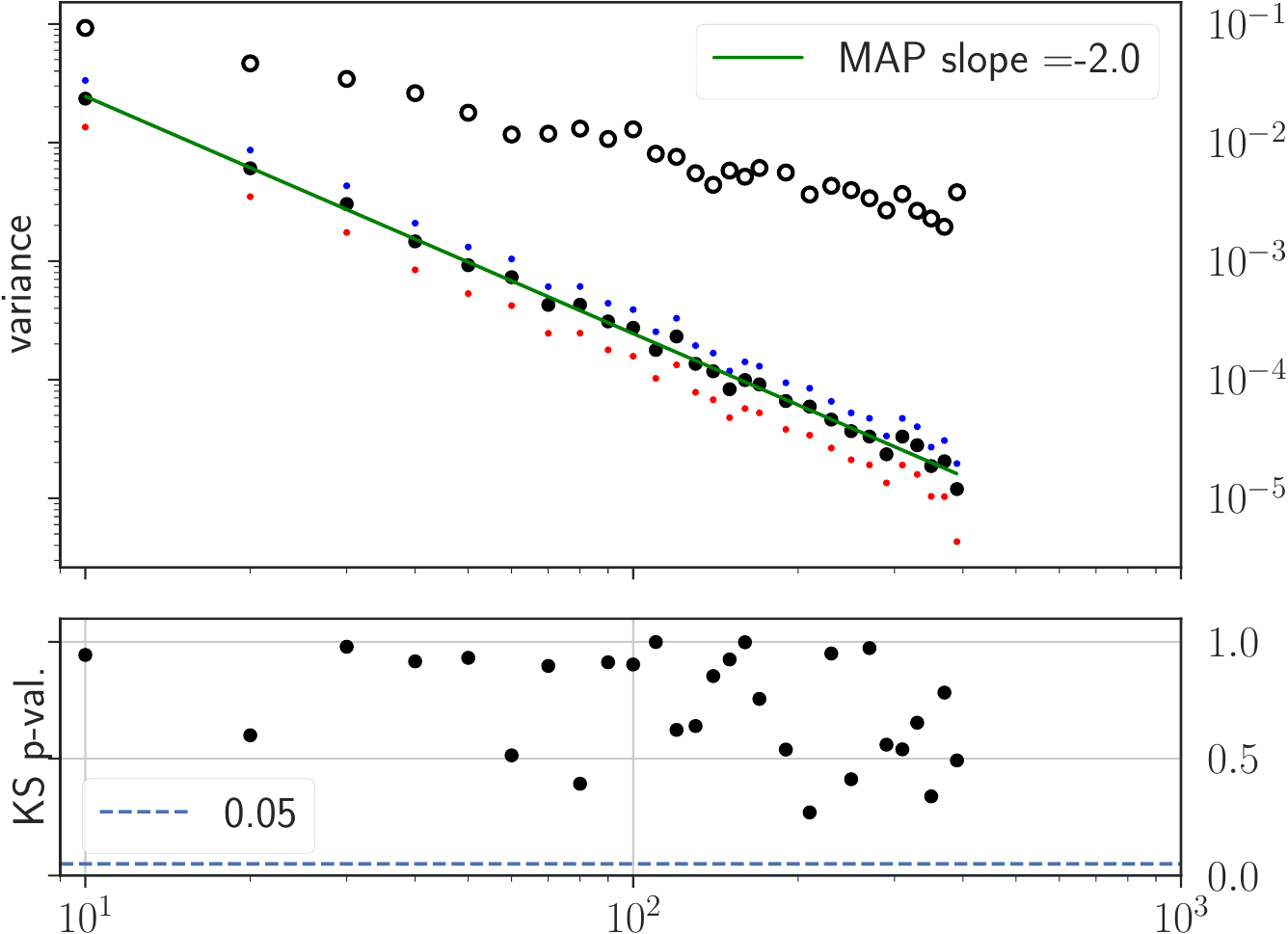}
\label{f:resultsAss1D}
}
\subfigure[$d=2$]{
\includegraphics[width=\twofig]{\figdir/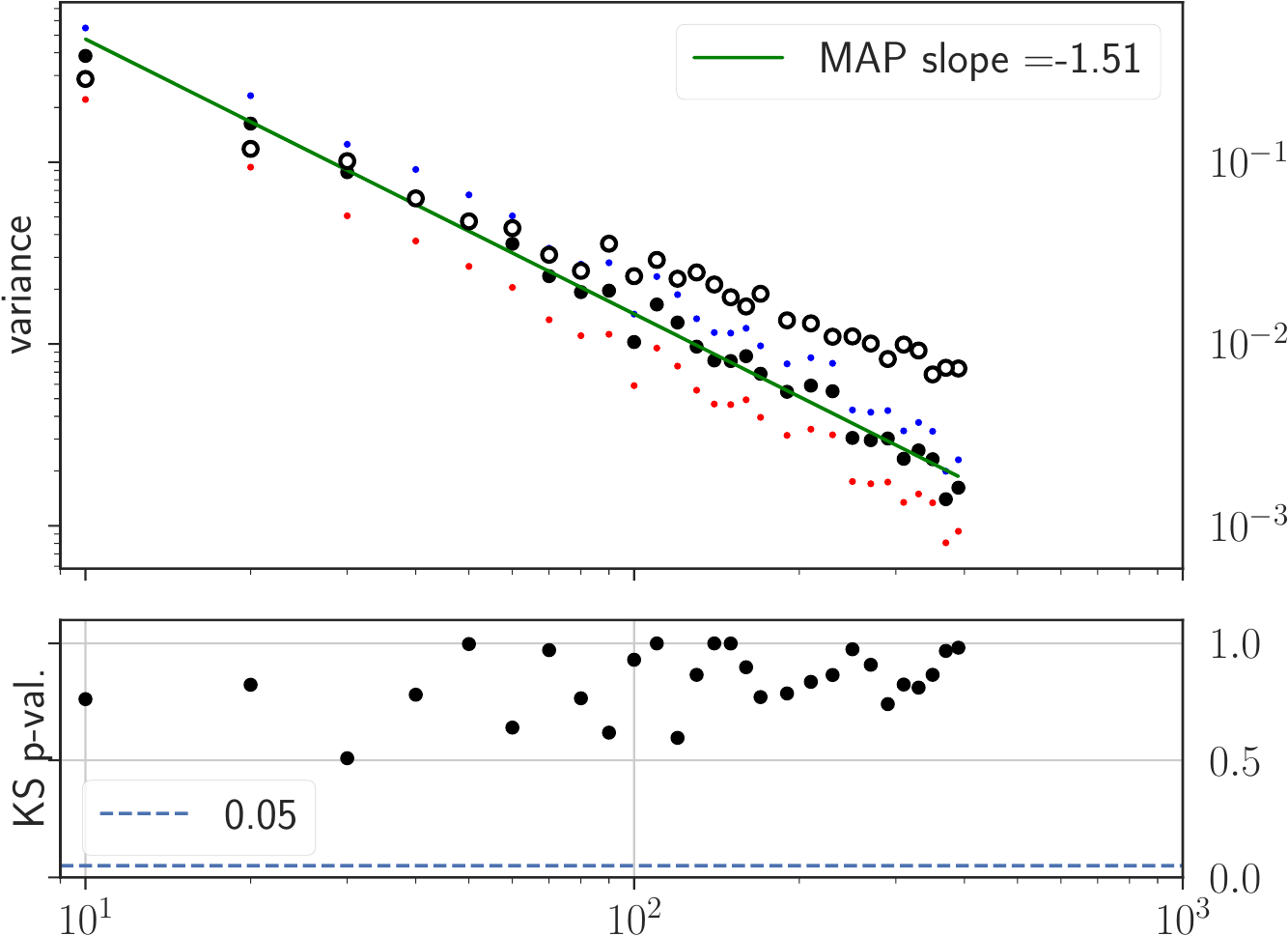}
\label{f:resultsAss2D}
}\\
\centering
\subfigure[$d=3$]{
\includegraphics[width=\twofig]{\figdir/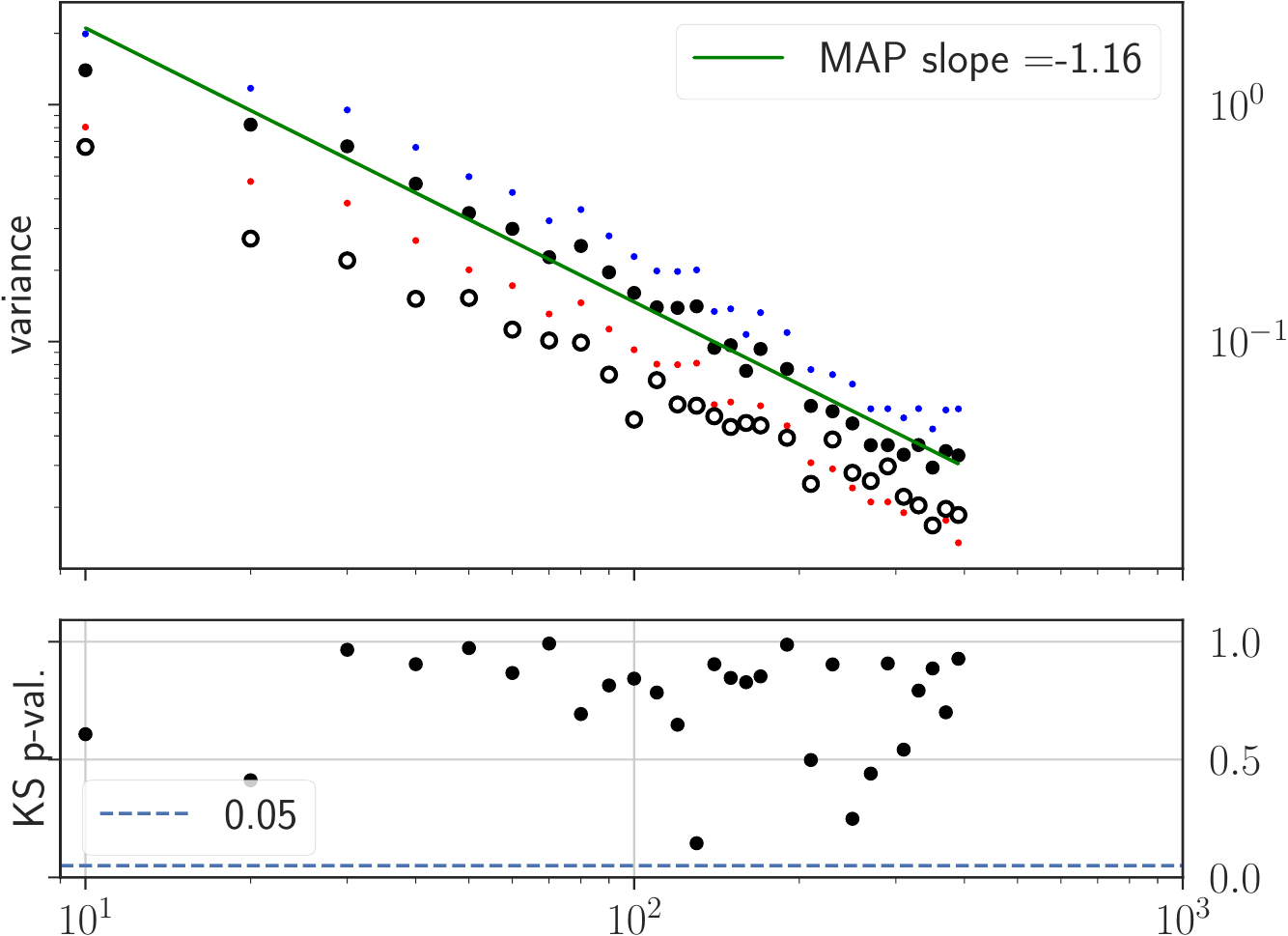}
\label{f:resultsAss3D}
}

\caption{Summary of the crude Monte Carlo results for a test function that violates the assumptions of Theorem~\ref{DPPMC Th1}.}
\label{f:resultsAss}
\end{figure}

\newpage
\bibliographystyle{plainnat}
\bibliography{submitted.bbl}

\end{document}